\documentclass[11pt]{article}
\usepackage[utf8]{inputenc}
\usepackage{amsmath, amssymb, amsthm}
\usepackage{graphicx}
\usepackage{geometry}
\usepackage{authblk}
\usepackage{xr-hyper}
\usepackage{algorithm}
\usepackage[noend]{algpseudocode}
\usepackage{caption}
\usepackage{subcaption}
\usepackage{fancyhdr}
\usepackage{booktabs}
\usepackage{dsfont}
\usepackage{tikz}
\usepackage[authoryear]{natbib}
\usepackage[colorlinks=true,linkcolor=blue,citecolor=black,urlcolor=blue]{hyperref}

\usepackage{color}
\usepackage{ulem}
\usepackage{moreverb}
\usepackage{tikz}
\usepackage{fix-cm}

\newcommand{\Cov}{\mathop{\operatorname{Cov}\/}}
\newcommand{\Var}{\mathop{\operatorname{Var}\/}}

\usepackage{array}

\makeatletter
\newcommand{\thickhline}{%
    \noalign {\ifnum 0=`}\fi \hrule height 1.3pt
    \futurelet \reserved@a \@xhline
}
\newcolumntype{"}{@{\hskip\tabcolsep\vrule width 1pt\hskip\tabcolsep}}
\makeatother

\usepackage{enumitem}
\usepackage{bm} 
\newlist{todolist}{itemize}{2}
\setlist[todolist]{label=$\square$}
\usepackage{enumitem}
\usepackage{bm} 

\newcommand{\wh}{\widehat}

\newcommand{\EE}{\mathbb{E}}
\newcommand{\RR}{\mathbb{R}}

\newcommand{\PP}{\mathbb{P}}

\newcommand{\ZZ}{\mathbb{Z}}

\newcommand{\NN}{\mathbb{N}}

\newtheorem{assumption}{Assumption}

\newcommand{\bu}{\mathbf{u}}
\newcommand{\bx}{\mathbf{x}}
\newcommand{\by}{\mathbf{y}}
\newcommand{\Lip}{\operatorname{Lip}}

\newcommand{\C}{\mathcal{C}}
\newcommand{\D}{\mathcal{D}}

\newcommand{\F}{\mathcal{F}}

\newcommand{\K}{\mathcal{K}}

\newcommand{\N}{\mathcal{N}}

\renewcommand{\S}{\mathcal{S}}

\newcommand{\W}{\mathcal{W}}

\newcommand{\Xb}{\textbf{X}}

\newtheorem{proposition}{Proposition}
\newtheorem{lemma}{Lemma}
\newtheorem{theorem}{Theorem}
\theoremstyle{definition}

\newtheorem{remark}{Remark}
\newtheorem{definition}{Definition}
\newtheorem{corollary}[proposition]{Corollary}
\newtheorem{example}{Example}

\begin{document}

\title{\huge{Ordinal Patterns Based Change Point Detection}}

\author{
 \textbf{Annika Betken}$^{1}$, \textbf{Giorgio Micali}$^{1}$,  and \textbf{Johannes Schmidt-Hieber} $^{1}$ \\
  \texttt{ a.betken@utwente.nl, g.micali@utwente.nl, a.j.schmidt-hieber@utwente.nl}
}
\date{}  
\maketitle
\footnotetext[1]{University of Twente, Faculty of Electrical Engineering, Mathematics, and Computer Science (EEMCS), Drienerlolaan 5, 7522 NB Enschede, Netherlands}

\begin{abstract}
The ordinal patterns of a fixed number of consecutive values in a time series is the spatial ordering of these values. Counting how often a specific ordinal pattern occurs in a time series provides important insights into the properties of the time series. In this work, we prove the asymptotic normality of the relative frequency of ordinal patterns for time series with linear increments.
Moreover, we apply  ordinal patterns to detect changes in the distribution of a time series. 
\end{abstract}

\section{Introduction}\label{sec1}
Ordinal patterns encode the spatial order of temporally-ordered data points. Specifically, the ordinal pattern of order $r+1$ of time series data $\xi_0, \ldots, \xi_{r}$ refers to the permutation $(\pi_0,\ldots, \pi_r)$, where $\pi_j$ denotes the rank of $\xi_j$ within the values $\xi_0, \ldots \xi_r$. For simplicity, we assume that the values of the data points are all different. Mathematically speaking, denoting with $\S_r$ the set of all $(r+1)!$  permutations of $\{0,\ldots, r\}$:
\begin{align}\label{eq:def_op}
\Pi: \RR^{r+1} \longrightarrow \S_r, \quad (\xi_0,  \ldots, \xi_r ) \mapsto (\pi_0, \ldots, \pi_{r}),
\end{align}
where $\pi_j$ denotes the rank of $\xi_j$ within the values $\xi_0, \ldots, \xi_r$. We call $\Pi(\xi_0, \ldots, \xi_r)$ the ordinal pattern of $\xi_0, \ldots, \xi_r$. The six ordinal patterns of order $r+1=3$ are visualized in Figure \ref{fig:op}.
\begin{figure}
\centering
\begin{tikzpicture}[x=1pt,y=1pt]
\path[use as bounding box,fill=black,fill opacity=0.00] (0,0) rectangle (361.35, 72.27);
\begin{scope}
\path[draw=black,line width= 0.6pt] ( 10.68, 10.19) --
	( 34.93, 29.58) --
	( 59.17, 48.97);
\path[draw=black,line width= 0.4pt,line join=round,line cap=round,fill=black] ( 10.68, 10.19) circle (  2.50);
\path[draw=black,line width= 0.4pt,line join=round,line cap=round,fill=black] ( 34.93, 29.58) circle (  2.50);
\path[draw=black,line width= 0.4pt,line join=round,line cap=round,fill=black] ( 59.17, 48.97) circle (  2.50);
\end{scope}

\begin{scope}
\path[draw=black,line width= 0.6pt,line join=round] ( 69.53, 10.19) --
	( 93.77, 48.97) --
	(118.02, 29.58);
\path[draw=black,line width= 0.4pt,line join=round,line cap=round,fill=black] ( 69.53, 10.19) circle (  2.50);
\path[draw=black,line width= 0.4pt,line join=round,line cap=round,fill=black] ( 93.77, 48.97) circle (  2.50);
\path[draw=black,line width= 0.4pt,line join=round,line cap=round,fill=black] (118.02, 29.58) circle (  2.50);
\end{scope}

\begin{scope}
\path[draw=black,line width= 0.6pt,line join=round] (128.37, 48.97) --
	(152.62, 10.19) --
	(176.87, 29.58);
\path[draw=black,line width= 0.4pt,line join=round,line cap=round,fill=black] (128.37, 48.97) circle (  2.50);
\path[draw=black,line width= 0.4pt,line join=round,line cap=round,fill=black] (152.62, 10.19) circle (  2.50);
\path[draw=black,line width= 0.4pt,line join=round,line cap=round,fill=black] (176.87, 29.58) circle (  2.50);
\end{scope}

\begin{scope}
\path[draw=black,line width= 0.6pt,line join=round] (187.22, 48.97) --
	(211.47, 29.58) --
	(235.72, 10.19);
\path[draw=black,line width= 0.4pt,line join=round,line cap=round,fill=black] (187.22, 48.97) circle (  2.50);
\path[draw=black,line width= 0.4pt,line join=round,line cap=round,fill=black] (211.47, 29.58) circle (  2.50);
\path[draw=black,line width= 0.4pt,line join=round,line cap=round,fill=black] (235.72, 10.19) circle (  2.50);
\end{scope}

\begin{scope}
\path[draw=black,line width= 0.6pt,line join=round] (246.07, 29.58) --
	(270.32, 48.97) --
	(294.57, 10.19);
\path[draw=black,line width= 0.4pt,line join=round,line cap=round,fill=black] (246.07, 29.58) circle (  2.50);
\path[draw=black,line width= 0.4pt,line join=round,line cap=round,fill=black] (270.32, 48.97) circle (  2.50);
\path[draw=black,line width= 0.4pt,line join=round,line cap=round,fill=black] (294.57, 10.19) circle (  2.50);
\end{scope}

\begin{scope}
\path[draw=black,line width= 0.6pt,line join=round] (304.92, 29.58) --
	(329.17, 10.19) --
	(353.42, 48.97);
\path[draw=black,line width= 0.4pt,line join=round,line cap=round,fill=black] (304.92, 29.58) circle (  2.50);
\path[draw=black,line width= 0.4pt,line join=round,line cap=round,fill=black] (329.17, 10.19) circle (  2.50);
\path[draw=black,line width= 0.4pt,line join=round,line cap=round,fill=black] (353.42, 48.97) circle (  2.50);
\end{scope}

\definecolor{text}{gray}{0.10}
\begin{scope}
\node[text=text,anchor=base,inner sep=0pt, outer sep=0pt, scale=  0.80] at ( 34.93, 56.08) {(0, 1, 2)};
\end{scope}

\begin{scope}
\node[text=text,anchor=base,inner sep=0pt, outer sep=0pt, scale=  0.80] at ( 93.77, 56.08) {(0, 2, 1)};
\end{scope}

\begin{scope}
\node[text=text,anchor=base,inner sep=0pt, outer sep=0pt, scale=  0.80] at (152.62, 56.08) {(2, 0, 1)};
\end{scope}

\begin{scope}
\node[text=text,anchor=base,inner sep=0pt, outer sep=0pt, scale=  0.80] at (211.47, 56.08) {(2, 1, 0)};
\end{scope}

\begin{scope}
\node[text=text,anchor=base,inner sep=0pt, outer sep=0pt, scale=  0.80] at (270.32, 56.08) {(1, 2, 0)};
\end{scope}

\begin{scope}
\node[text=text,anchor=base,inner sep=0pt, outer sep=0pt, scale=  0.80] at (329.17, 56.08) {(1, 0, 2)};
\end{scope}
\end{tikzpicture}
\caption{The six  ordinal patterns of order $r+1=3$.}
\label{fig:op}
\end{figure}
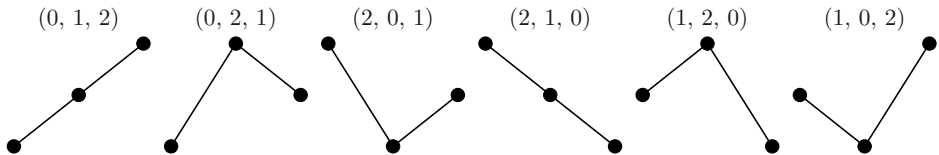
In  modern time series literature, the concept of ordinal patterns was first introduced by \citet{Bandt-Pompe} against the background of defining  permutation entropy as a complexity measure of time series data.
The latter is the Shannon
entropy of the ordinal pattern distribution. As an estimate for ordinal  pattern probabilities we choose the relative frequencies of ordinal patterns in the underlying time series. The corresponding estimator has been studied
by \citet{sinn2011estimation} (for short-range dependent Gaussian time series), and by \citet{betken2021ordinal} (for long-range dependent subordinated Gaussian time series).
Moreover, \citet{schnurr2017testing} measured non-linear correlation of two time series  counting the number of coincident  patterns in both time series, and as a byproduct also study this estimator for a class of short-range dependent time series. 

One application of ordinal patterns is the classification of sleep stages based on EEG signals, crucial for medical purposes such as sleep quality assessment and sleep disorder diagnosis. Manual classification of sleep stages is tedious, subjective, time-consuming, and error-prone. To address this, a plethora of machine learning-based techniques have emerged in recent years, often achieving high accuracy in classification; see, for example,  \citet{supratak2017deepsleepnet, chambon2018deep,CHENG2024106020}. For a comprehensive review of automated sleep stage classification methods see \citet{ZHANG2024651}. However, these methods typically lack interpretability and theoretical justification, may inherit biases present in training data, and are often sensitive to small perturbations of the input; see \citet{lipton2016mythos}. 
To address these challenges, \citet{bandt2020order} introduces an alternative approach to frequency classification in machine learning, based on the so-called turning rate.  In a time series, the turning rate corresponds to the relative number of local maxima and minima in a fixed epoch of the series.
More precisely, the turning rate corresponds to the frequency of observing one of the ordinal patterns $(0, 2, 1)$, $(1, 0, 2)$, $(1, 2, 0)$, $(2, 0, 1)$, each representing a local minimum or maximum, as illustrated in Figure \ref{fig:op}. 

The aim of this work is to propose a hypothesis test for detecting changes in the distribution of a time series, based on variation of the corresponding turning rate series. For this, we derive limit theorems for estimators of ordinal patterns, assuming that the increments of a time series form a linear process, thereby 
allowing for different distributions and dependencies between observations than in previous works on ordinal patterns (a more explicit comparison of our work to previous results will be given in Section \ref{subsec:rel_work}).
As theoretical background, we establish
 empirical process limit theory for short- and long-range dependent multivariate  linear time series. 
These general results, which are of independent interest, then serve as the basis for deriving the asymptotic properties of estimators for ordinal pattern probabilities.

The article is structured as follows: Section \ref{sec:main} establishes central limit theorems for empirical processes of short- and long-range dependent linear time series. In this section, ordinal patterns are formally introduced, and the derived limits are applied to obtain the limiting distribution for estimators of ordinal pattern probabilities. Section \ref{sec:appl} extends these results to define the turning rate and proposes a statistical test for detecting structural changes in the distribution of a time series. Section \ref{sec:sim} presents numerical experiments, including an application to EEG time series data. All proofs are provided in the appendix.

\subsection{Notation} A generic time series is denoted as \((\xi_t)_{t\geq 0}\) and  \(\bm{\xi}_t = (\xi_t, \xi_{t+1}, \ldots, \xi_{t+r})^\top\) represents the random vector of \(r+1 \geq 2\) consecutive time series values. The time series generated by the  increments of \((\xi_t)_{t\geq 0}\) is denoted by \((X_t)_{t\geq 1}\), where \(X_t := \xi_t - \xi_{t-1}\) and \(\mathbf{X}_t = (X_t, \ldots, X_{t+r-1})^\top\).  The set of invertible matrices of size \(r\) is denoted by \(\text{GL}(\mathbb{R}, r)\). The symbol \(\bar{X}_n\) denotes the average \(\bar{X}_n := \frac{1}{n}\sum_{j=1}^n X_j\).  The space \(\ell^m\) denotes the set of real sequences \((a_j)_{j\in \mathbb{N}}\) satisfying \(\sum_{j\in \mathbb{N}}|a_j|^m < \infty\), \(L^\infty(\mathbb{R})\) is the set of all real-valued functions defined on \(\mathbb{R}\) that are almost surely bounded, and \(\C^1(\mathbb{R})\) consists of continuously differentiable functions on \(\mathbb{R}\). The gradient is 
$\nabla f (\mathbf{x}) := \left( \partial_1 f(\mathbf{x}), \ldots, \partial_p f(\mathbf{x}) \right)\;.$ An i.i.d. sequence of centered random vectors  $(\mathbf{Z}_j)_{j\in \mathbb{Z}}$ is called  multivariate zero-mean white noise process with covariance matrix $\Sigma=\Cov(\mathbf{Z}_1)$.
  Lastly, the indicator function of an event $A$ is denoted by $\bm{1}(A)\;.$
\section{Main theoretical results}\label{sec:main}
Let $X_t=\sum_{j=0}^\infty a_j Z_{t-j},$ $t \in \mathbb{Z}$, be a linear process with deterministic coefficients $a_j$ and i.i.d.\ random variables $(Z_j)_{j \in \mathbb{Z}}$. We also assume that $\EE[Z_j]=0$ and $\Var(Z_j)=\sigma_Z^2$ for all $j,$ implying that $X_t$ is centered. By Kolmogorov's three-series theorem,  $X_t$ exists almost surely if $\sum_{j=0}^\infty a_j^2< \infty$; see \citet{wu2002central}. Since the innovations $Z_j$ are not assumed to be Gaussian, the process $X_t$ can be non-Gaussian. The monograph of \citet{MR1742357} highlights that allowing for non-Gaussianity makes the class of processes much richer.

It follows from the definition that the linear process $(X_t)_{t\geq 1}$ is moreover (strictly) stationary. As such, it admits an autocovariance function $\gamma_X(k)$ (for $k \in \mathbb{Z}$) and a corresponding spectral density $f_X(\lambda) = (2\pi)^{-1} \sum_{k=-\infty}^{\infty} \gamma_X(k) \exp(-ik\lambda)$ with $\lambda \in [-\pi,\pi]$. There are various definitions for short- and long-range dependence; see for instance  \citet{STO-004}. While sometimes more specific growth conditions on the spectral density are imposed, they all agree that the process \((X_t)_{t\geq 1}\) can be categorized as exhibiting long-range dependence, short-range dependence, or antipersistence, if, as $|\lambda| \to 0$, the spectral density converges to infinity, a finite positive constant, or zero, respectively. Since $2\pi f_X(0) = \sum \gamma_X(k)$, these conditions can also be equivalently stated in terms of the sum of autocovariances $\gamma(k)$. 
For a linear process, $\gamma_X(k)=\EE[X_tX_{t+k}]=\sum_{j, i\geq 0} a_j a_i \Cov(Z_{t-j},Z_{t+k-i})=\sum_{j\geq 0} a_j a_{|k|+j}$ and $\sum_{k=-\infty}^\infty \gamma_X(k)=(\sum_{j\geq 0} a_j)^2.$
This shows that long-range dependence happens if $\sum_{j\geq 0} a_j=\infty,$ short-range dependence occurs if $\sum_{j\geq 0} a_j$ is a finite constant, and antipersistence corresponds to $\sum_{j\geq 0} a_j=0.$ Standard textbooks restrict these classes further by adding decay conditions; see e.g. Section 2.1.1.3 in \citet{beran2013long}.
\(\Xb_t = \sum_{j\in \mathbb{Z}} A_j \mathbf{Z}_{t-j},\) 
where \((\mathbf{Z}_j)_{j \in \mathbb{Z}}\) is an i.i.d.\ multivariate white noise process, is said to have short memory if $\sum_{j \in \mathbb{Z}}\|A_j\|<\infty. $ In this work, we adopt the notion of long-range dependence proposed by \citet{Kerchagian_Pipiras} (Definitions 2.1 and 2.2). 
Proposition 3.1 in \citet{Kerchagian_Pipiras} shows that if the sequence of coefficient matrices \( A_j \) satisfies 
\begin{equation}  
\label{eq:A_j_matrices}  
A_j \sim \operatorname{diag}\left(j^{d - \frac{1}{2}}\right) A_\infty \quad \text{as } j \to \infty,  
\end{equation}  
for some matrix \( A_\infty \) and a vector \( d = (d_1, \ldots, d_r)^T \in \left(0, \frac{1}{2}\right)^r \), then the linear process \( (\Xb_t)_{t\geq 1} \) exhibits long-range dependence. Here, for $r\times r$ matrices \(U_j\) and \(V_j\),  we write \(U_j \overset{j \to \infty}{\sim} V_j\), if \(u_{ps,j}/v_{ps,j} \to 1\) as \(j \to \infty\) for all entries \((p,s)\) of \(U_j\) and \(V_j\). 

Our aim is to establish (functional) central limit theorems for the relative frequencies of linear processes and apply these to  estimators of ordinal pattern probabilities.
In this work, we focus on the increments $X_t=\xi_t -\xi_{t-1}$ of the process \((\xi_t)_{t \geq 0}\).  Lemma \ref{lemma:matrix_form_OP} shows that the increment process allows us to interpret the ordinal patterns \((\pi_0, \ldots, \pi_r)\) of \((\xi_t)_{t \geq 0}\) as the ordinal patterns of a linear transformation of the increments \((X_t)_{t\geq 1}\). Assuming that the increments form a linear process implies that the vectors  \(\mathbf{X}_t = (X_{t+1}, \ldots, X_{t+r})^\top\) form a multivariate linear process, and for an $s \times r$ matrix $V$, $(V\mathbf{X}_t)_t$ forms a multivariate linear process as well. 
Based on this setup, we establish limiting theorems for multivariate linear processes \((\mathbf{X}_t)_{t \geq 1}\) (see Section \ref{sec.CLTs_general}), using a martingale-decomposition approach.
\subsection{Central limit theorems for  relative frequencies in linear processes}\label{subsec:clt}
\label{sec.CLTs_general}
Let \( r \) be a positive integer, and assume that for \( n \geq r - 1 \), we observe a linear process \( (X_t)_{t \geq 1} \) for \( t = 1, \ldots, n + r - 1 \). A common approach to estimate the probability \( p(u_0, \ldots, u_{r-1}) := \mathbb{P}(X_{t} \leq u_0, X_{t+1} \leq u_1, \ldots, X_{t+r-1} \leq u_{r-1}) \) is by using the relative frequency of this event in the sample, expressed as:
\begin{equation}
    \label{eq.rel_freq_est}
    \wh p_n(u_0, \ldots, u_{r-1}) := \frac{1}{n}
    \sum_{t=1}^{n} \bm{1}\big(X_{t} \leq u_0, \ldots, X_{t+r-1} \leq u_{r-1}\big).
\end{equation}
By taking the expectation, we conclude that this is an unbiased estimator for the probability \( p(u_0, \ldots, u_{r-1}) \).

We now proceed to derive (functional) central limit theorems for the relative frequencies \eqref{eq.rel_freq_est}. To account for the \( r \) inequalities, it is convenient to first reformulate the univariate linear time series as an \( r \)-dimensional multivariate linear process.

\begin{lemma}
\label{lemma:vectorization}
 Given a linear process defined by $X_t=\sum_{j= 0 }^\infty a_j Z_{t-j}$, the multivariate process $\Xb_t:=\left(X_t, X_{t+1}, \ldots, X_{t+r-1}\right)^{\top}$ is linear and satisfies 
    \begin{equation}
        \Xb_t= \sum_{j= 0 }^\infty A_j \mathbf{Z}_{t-j}\;
        \label{consecutive-lineare-representation}
    \end{equation}
    with diagonal coefficient matrices
\begin{align*}
    A_j = \left( \begin{array}{ccccc}
        a_{j-r+1} & 0 &\cdots & 0  \\
        0& a_{j-r+2} & \ddots  & \vdots\\
        \\
        \vdots &\ddots & \ddots& 0 \\
        0 & \cdots& 0 & a_j\\
   \end{array}\right)
\end{align*}
(setting $a_i:=0$ whenever $i <0$) and  i.i.d. innovations  $\mathbf{Z}_{t-j}= Z_{t-j+r-1}\left(1, \ldots, 1\right)^{\top} $ with variance $\mathbf{E}$, where $\mathbf{E}$ denotes the $r\times r$ matrix with all entries equal to $1$.
\end{lemma}
\begin{proof}
Changing $j$ to $j-r+s+1,$ we find for any $s=0,\ldots, r-1,$ $X_{t+s}=\sum_{j= 0 }^\infty a_j Z_{t+s-j}=\sum_{j=0}^\infty a_{j-r+s+1}Z_{t-j+r-1}$ with $a_i=0$ whenever $i<0.$    
\end{proof}
We can now rewrite the relative frequency estimator $\wh p_n(u_1,\ldots,u_r)$ defined in \eqref{eq.rel_freq_est} as 
\begin{align}
    \wh p_n(\bu) :=\frac{1}n
    \sum_{t=1}^{n} \bm{1}\big(\Xb_t\leq \bu\big),
    \label{eq.rel_freq_est2}
\end{align}
with $\bu:=(u_1,\ldots,u_r)^\top$ and $\leq$ understood component-wise. This is an estimator for the probability $
    p(\bu) :=P\big(\Xb_t\leq \bu\big). $

In a next step, we prove a functional central limit theorem for multivariate linear processes with general covariance matrix for $\mathbf{Z}_{j}$ and general coefficient matrix $A_j$ satisfying the following assumption:
\begin{assumption}
\label{assump.1}
There exists a $J\in \mathbb{N}$ and an invertible $r\times r$ matrix $D$ such that $D\sum_{j=0}^J A_j \mathbf{Z}_{t-j}$ is a vector of independent random variables with bounded Lebesgue density. Furthermore, $A_j \neq \mathbf{0}_{r\times r}$ for $j=0,\ldots, J\;.$
\end{assumption}
If $A_0$ is an invertible matrix and $\mathbf{Z}_{t-j}$ consists of independent random variables with bounded Lebesgue density, the condition is satisfied with $J=0$ and $D=A_0^{-1}.$ If the multivariate process has been generated by a univariate linear process as in the setting of Lemma \ref{lemma:vectorization}, $a_0\neq 0$, and the innovations admit a bounded Lebesgue density, then the condition holds with $J=r-1.$ Indeed, note that
\begin{align}
  \sum_{j=0}^{r-1} A_j \mathbf{Z}_{t-j} = 
  \sum_{j=0}^{r-1} Z_{t-j+r-1} A_j
  \left( \begin{array}{c}
  1 \\
  1 \\
  \vdots \\
  1
  \end{array}\right)
  =\left( \begin{array}{ccccc}
        a_0 & 0 &\cdots & 0  \\
        a_1& a_0 & \ddots  & \vdots\\
        \vdots & & \ddots& 0 \\
        a_{r-1} & \cdots& a_1 & a_0\\
   \end{array}\right) \left( \begin{array}{c}
        Z_t  \\
        Z_{t+1} \\
        \vdots\\
        Z_{t+r-1}
   \end{array} \right)=: B\mathbf{Z}_{t,r}.
   \label{expression_density_of_X_MAIN}
\end{align}
Since $a_0\neq 0,$ the triangular matrix on the right-hand side is invertible and $D$ can be taken as its inverse which then gives $D\sum_{j=0}^J A_j \mathbf{Z}_{t-j}=(Z_t,\ldots,Z_{t+r-1})^\top=:\mathbf{Z}_{t,r}.$ Since by assumption, the innovations $Z_t$ admit a bounded Lebesgue density, this verifies Assumption \ref{assump.1} in this case. 

Let $\|\cdot\|$ be the operator norm and denote by $\overset{\D[0,1]}{\Longrightarrow} $ the convergence in distribution in the Skorokhod space $\D[0,1]$ with respect to the  Skorokhod topology; see \citet{billingsley1968convergence}.
\begin{theorem}[Short-Range Dependence]
\label{theorem:SRD_multivariate}
Let $\Xb_t= \sum_{j= 0 }^\infty A_j \mathbf{Z}_{t-j}$ be a multivariate linear process satisfying $\sum_{j=0}^\infty \|A_j\|< \infty$ and Assumption \ref{assump.1}. Then, for any $r$-dimensional vector $\bu=(u_0, \ldots, u_{r-1})^\top,$
 \begin{equation}
  \frac{1}{\sqrt{n}}\sum_{t=1}^{[n\tau]} \Big(\bm{1}\big(\Xb_t\leq \bu\big)
  -p(\bu)\Big)\overset{\D[0,1]}{\Longrightarrow} \sigma B(\tau)\;, \quad \tau \in [0,1]\;,   
 \end{equation}
with variance $\sigma^2:= \Var\left( \bm{1}( \mathbf{X}_1 \leq \mathbf{u}) \right) +2\sum_{j=1}^{\infty}\Cov \left(  \bm{1}( \mathbf{X}_1 \leq \mathbf{u}),  \bm{1}( \mathbf{X}_{1+j} \leq \mathbf{u}) \right) $. 
\end{theorem}
In particular, for $\tau=1$, we obtain
\begin{equation}
    \sqrt{n}\Big(\widehat{p}_n(\bu) - p(\bu)\Big)\xrightarrow{\D} \N(0,\sigma^2)\;. 
    \label{eq:asymptotic_normality_SRD}
\end{equation}
The proof of Theorem \ref{theorem:SRD_multivariate} is provided in Appendix \ref{appendix:SRD}. The key ingredient is to verify the conditions of a modified version of Theorem \ref{Theorem:Furmanczyk} in \citet{Furmanczyk}.

We now discuss the case where the underlying multivariate linear process exhibits long-range dependence. 
\begin{theorem}[Long-Range Dependence]
\label{thm.2}
Let \( \mathbf{X}_t = \sum_{j=0}^\infty A_j \mathbf{Z}_{t-j} \) be a multivariate linear process satisfying \( A_j \sim j^{d-1} A_\infty \) as \( j \to \infty \), where \( A_\infty \in \text{GL}(\mathbb{R}, r) \) and \( d \in (0,1/2) \). The innovations \( ( \mathbf{Z}_j )_{j \in \mathbb{Z}} \) are i.i.d.\ with variance \( \Sigma \), and are assumed to satisfy the moment condition \( \mathbb{E}[\|\mathbf{Z}_1\|^4] < \infty \). Define the cumulative distribution function \( p_s(\cdot) = \mathbb{P}( \sum_{j=0}^s A_j \mathbf{Z}_{t-j} \leq \cdot) \). If there exists a positive integer \( s_0 \) such that  
\begin{equation}
    \sup_{\bx\in \mathbb{R}^r} \, \max_{s\geq s_0} \left(  |p_s(\mathbf{x}) |
    + \sum_{i=1}^r
    | \partial_i p_s(\mathbf{x}) |
    + \sum_{i, j=1}^r
    | \partial^2_{i,j} p_s(\mathbf{x}) | \right)<\infty,
    \label{eq:LIP-condition_ORIGINAL}
\end{equation}
then, 
\begin{equation}
    n^{1/2-d}\Big(\widehat{p}_n(\bu) - p(\bu)\Big)\xrightarrow{\D} \N\bigg(0,\frac{\Gamma(d)^2 }{\Gamma(2d+2)\cos (\pi d)}( \nabla p(\bu))^{\top} A_\infty \Sigma A_\infty^{\top} \nabla p(\bu) \bigg)\;. 
    \label{eq:asymptotic_normality_LRD}
\end{equation}
\end{theorem}
The proof of Theorem \ref{thm.2} can be found in Appendix \ref{appendix:LRD}. 
A key ingredient of the proof is to establish a so-called reduction principle, stating that for any $\mathbf{u} \in \mathbb{R}^r,$
\begin{equation}
    n^{\frac{1}{2} - d} \left| \frac{1}{n} \sum_{t=1}^n \bm{1}\big(\mathbf{X}_t \leq \bu\big) - p(\bu) + (\nabla p(\bu))^\top \bar{\mathbf{X}}_n \right| \xrightarrow{\mathbb{P}} 0.
\end{equation}
\subsection{Ordinal patterns }
\label{preliminaries}

We apply the central limit theorems derived in the previous section to ordinal patterns. Consider a univariate time series $(\xi_t)_{t\geq 0}$. 
\begin{definition}
\label{def:ordinal_pattern}
Let \( S_r \) denote the set of permutations of \(\{0, \dots, r\}\), which we write as \((r + 1)\)-tuples containing each of the numbers \( 0, \dots, r \) exactly once. By the \textit{ordinal pattern} of order \( r \), we refer to the permutation
\[
\Pi(\xi_0, \dots, \xi_r) = (\pi_0, \dots, \pi_r) \in S_r,
\]
which satisfies
\[
\xi_{\pi_0} \geq \cdots \geq \xi_{\pi_r},
\]
and \(\pi_{i-1} > \pi_i\) if \(\xi_{\pi_{i-1}} = \xi_{\pi_i}\) for \( i = 1, \dots, r - 1 \). We say that the time series $(\xi_t)_{t\geq 0}$ has ordinal pattern $(\pi_0,\ldots,\pi_r)\in S_r$ at time $t,$ if $\Pi(\xi_t, \cdots, \xi_{t+p})=(\pi_0, \ldots, \pi_{r})\;.$
\end{definition}
Ordinal patterns look for specific orderings of the time series values over $r+1$ consecutive time instances. The frequency at which they occur can provide some insights about the distributional properties of the time series. 

Following \citet{sinn2011estimation} and \citet{betken2021ordinal}, we can rewrite ordinal patterns as inequalities of the increment process $X_t:=\xi_t-\xi_{t-1}.$ Let the row vector $e_k$ be the $k$-th standard basis vector of $\mathbb{R}^{r+1}$ and recall that $\bu \leq 0$ for a vector $\bu$ means that all entries of $\bu$ are $\leq 0$.
\begin{lemma}
\label{lemma:matrix_form_OP}
The time series $(\xi_t)_{t\geq 0}$ has ordinal pattern $\pi=(\pi_0,\ldots,\pi_r)\in S_r$ at time $t,$ if and only if the stacked increment process $\Xb_{t+1}=(X_{t+1},\ldots,X_{t+r})^\top$ satisfies $V_\pi \Xb_{t+1}\leq 0$
with 
\begin{equation}
V_\pi :=
\underbrace{\left( \begin{array}{cccccc}
    -1 & 1 & 0 & 0 & \cdots & 0\\
     0 & -1 & 1 & 0 & \cdots & 0\\
     \vdots & &\ddots & \ddots & & \vdots \\
     0 & \cdots &0  & -1 & 1 & 0\\
0 & \cdots &0 &0 & -1 & 1 \\
\end{array}\right)}_{r \times (r+1)}
\underbrace{
 \left( \begin{array}{c}
      e_{\pi_0+1} \\
            e_{\pi_1+1} \\
      \vdots \\ \vdots \\
      e_{\pi_{r}+1} \\
\end{array}\right)}_{(r+1) \times (r+1)} \underbrace{ \left( \begin{array}{cccc}
    0 & \cdots &\cdots &  0   \\
     1& \ddots& &  \vdots\\
     \vdots & \ddots & \ddots & \vdots\\
     \vdots && \ddots &  0\\
     1 & \cdots &\cdots &  1  \\
\end{array}\right)}_{(r+1)\times r}\;.
\label{matrix_W}
\end{equation}
In other words, $\left\{\Pi(\xi_t, \xi_{t+1}, \ldots, \xi_{t+r})  =\pi\right\} = \{ V_\pi \mathbf{X}_{t+1}\leq \mathbf{0}\}\;.$ The matrix $V_{\pi}$ is moreover invertible.  
\end{lemma}
\begin{proof} Observe that the matrix vector product of the right-most matrix in \eqref{matrix_W} with $\Xb_{t+1}$ gives the column vector $(\xi_t-\xi_t,\xi_{t+1}-\xi_t,\ldots,\xi_{t+r}-\xi_t)^\top.$ Multiplying this vector from the left with the $(r+1)\times (r+1)$ permutation matrix that occurs in the middle of the matrix product that defines $V_{\pi},$ reorders the entries and gives $(\xi_{t+\pi_0}-\xi_t,\xi_{t+\pi_1}-\xi_t,\ldots,\xi_{t+\pi_r}-\xi_t)^\top.$ Thus, $V_{\pi}\Xb_{t+1}=(\xi_{\pi_1}-\xi_{\pi_0},\ldots,\xi_{\pi_r}-\xi_{\pi_{r-1}})^\top.$ Hence $V_{\pi}\Xb_{t+1}\leq 0$ if and only if $\xi_{t+\pi_{i+1}}- \xi_{t+\pi_{i}}\leq 0,$ for all $i=0, \ldots, r-1.$ 
To see that $V_{\pi}$ is invertible, one can use again the definition of $V_{\pi}$ as a product of three matrices in \eqref{matrix_W} and observe that the kernel of the product of the two matrices on the left is $(1,\ldots,1)^\top$ whereas the image of the right-most matrix consists of all vectors with the first entry equaling zero. Since the intersection of this kernel and this image is the zero-vector, the kernel of $V_{\pi}$ is trivial and thus, $V_{\pi}$ is invertible and the lemma holds.
\end{proof}
The relative frequency of the ordinal patterns $\pi$ occurring in the time series $\xi_0,\ldots,\xi_{n+r-1}$ is 
\begin{align}
    \wh p_n(\pi) 
    = \frac{1}{n}\sum_{t=0}^{n-1} \bm{1} \left( \Pi(\xi_t, 
    \ldots, \xi_{t+r})= \pi \right) =\frac{1}{n}\sum_{t=1}^n \bm{1}\big(V_{\pi}\Xb_t\leq \mathbf{0}\big).
\label{ordinal_pattern_estimator}
\end{align}
This is an estimator of the probability that the ordinal pattern $\pi$ occurs, which is,
\begin{align}
    p(\pi):=\PP \left( \Pi(\xi_0, 
    \ldots, \xi_{r})= \pi \right) = \mathbb{P}\big(V_\pi \Xb_1\leq \mathbf{0}\big).
\end{align}
To analyze the estimator $\hat{p}_n(\pi)$, we assume that the increment process $(X_t)_{t\geq 1}$ is a linear process $X_t=\sum_{j=0}^\infty b_j Z_{t-j}.$ Applying Lemma \ref{lemma:vectorization}, it then follows that $\Xb_t=(X_t,X_{t+1},\ldots,X_{t+r-1})^\top$ is a multivariate linear process that can be written as $\Xb_t=\sum_{j=0}^\infty B_j \mathbf{Z}_{t-j}$ with diagonal coefficient matrix 
\begin{align}
    B_j = \left( \begin{array}{ccccc}
        b_{j-r+1} & 0 &\cdots & 0  \\
        0& b_{j-r+2} & \ddots  & \vdots\\
        \\
        \vdots &\ddots & \ddots& 0 \\
        0 & \cdots& 0 & b_j\\
   \end{array}\right)
    \label{eq.Bj_def}
\end{align}
(setting $b_\ell:=0$ whenever $\ell <0$) and i.i.d. innovations  $\mathbf{Z}_{t-j}= Z_{t-j+r-1}\left(1, \ldots, 1\right)^{\top} $ with variance $\mathbf{E}$, where $\mathbf{E}$ denotes the $r\times r$-matrix with all entries equal to $1$. Moreover, $V_\pi \Xb_t=\sum_{j=0}^\infty A_j \mathbf{Z}_{t-j}$ with $A_j=V_{\pi}B_j.$ Thus, also $(V_\pi \Xb_t)_t$ is a multivariate linear process. By applying now the central limit theorems from the previous section to $(V_\pi \Xb_t)_t$, we obtain central limit theorems for the relative frequencies $\widehat p_n(\pi)$  of ordinal patterns. A consequence of Theorem \ref{theorem:SRD_multivariate} is the following theorem:
\begin{theorem}[Short-Range Dependence]
\label{theorem:convergence_OP_SRD}
Let $(\xi_t)_{t\geq0}$ be a time series whose increments $X_t=\xi_t-\xi_{t-1}$ form a linear process $X_t=\sum_{j=0}^\infty b_jZ_{t-j}$ with $\sum_{j=0}^\infty |b_j|<\infty,$ and $b_0\neq 0.$ If $Z_1$ admits a continuous and bounded probability density function, then, 
 \begin{equation}
  \frac{1}{\sqrt{n}}\sum_{t=1}^{[n\tau]} \Big(\bm{1} \left( \Pi(\xi_{t-1}, \ldots, \xi_{t+r-1})=\pi \right) - p(\pi)\Big)\overset{\D[0,1]}{\Longrightarrow} \sigma_\pi B(\tau)\;, \quad \tau \in [0,1]\;, 
\label{eq:uniform_asymptotic_normality_op} 
 \end{equation}
 with variance $\sigma_\pi^2= \Var \left( \bm{1} \left( \Pi(\xi_{0}, \ldots, \xi_{r})=\pi \right)\right) +2\sum_{j=1}^\infty \Cov \left( \bm{1} \left( \Pi(\xi_{0}, \ldots, \xi_{r})=\pi \right), \bm{1} \left( \Pi(\xi_{j}, \ldots, \xi_{j+r})=\pi \right) \right)\;. $
\end{theorem}
For $\tau=1$, we obtain
\begin{equation}
    \sqrt{n}\Big(\hat{p}_n(\pi) - p(\pi)\Big)\xrightarrow{\D} \N(0,\sigma_\pi^2)\;. 
    \label{eq:asymptotic_normality_OP_SRD}
\end{equation}
\begin{proof}[Proof of Theorem \ref{theorem:convergence_OP_SRD}]
As a consequence of Lemma \ref{lemma:vectorization} and relation \eqref{ordinal_pattern_estimator}, the proof applies Theorem \ref{theorem:SRD_multivariate} to $(V_{\pi}\Xb_t)_{t\geq 1}\;.$ By assumption $b_0\neq 0.$ In \eqref{expression_density_of_X_MAIN} we have already verified Assumption \ref{assump.1} for the linear process $(\Xb_t)_t$ with $\Xb_t=\left( X_t, \ldots, X_{t+r-1} \right)^\top.$ Since $V_{\pi}$ is invertible, Assumption \ref{assump.1} also holds for the linear process $(V_{\pi}\Xb_t)_{t\geq 1}.$ All the assumptions needed for Theorem \ref{theorem:SRD_multivariate} are satisfied, concluding the proof. 
\end{proof}
We now derive the limit distribution of the estimator for ordinal pattern probabilities for a class of processes whose increments exhibit long-range dependence. 
\begin{theorem}[Long-Range Dependence]
\label{theorem:convergence_OP_LRD}
Let $(\xi_t)_{t\geq0}$ be a time series whose increments $X_t=\xi_t-\xi_{t-1}$ form a linear process $X_t=\sum_{j=0}^\infty b_jZ_{t-j}$ with $b_j \overset{j\to \infty}{\sim} j^{d-1}$ for $d\in (0,1/2)$. If $Z_1$ admits a density $f$ such that $f\in L^{\infty}(\RR)\cap C^1(\RR)$ and $f' \in L^\infty(\RR)$, with  finite fourth moment \(\EE[|Z_1|^4] < \infty\), then
\[
n^{ \frac{1}{2}-d} \left( \hat{p}_n(\pi) - p(\pi) \right) \overset{\D}{\longrightarrow} \N (0, \sigma_\pi^2),
\]
where \(\sigma_\pi^2 = \C_d (\nabla \tilde{p}(\mathbf{0}))^{\top} V_\pi \mathbf{E} V_\pi^{\top} \nabla \tilde{p}(\mathbf{0})\), with \(\C_d = \frac{\Gamma(d)^2 }{\Gamma(2d+2)\cos (\pi d)}\), \(\tilde{p}(\cdot) := \PP(V_\pi \Xb_1 \leq \cdot)\), and \(\mathbf{E}\) the \(r \times r\) matrix with all entries equal to 1.
\end{theorem}
Since $\mathbf{E}=\bm{1}\bm{1}^\top$ with $\bm{1}=(1,\ldots,1)^\top,$ we can also write $((\nabla \tilde{p}(\mathbf{0}))^{\top} V_\pi \mathbf{E} V_\pi^{\top} \nabla \tilde{p}(\mathbf{0})=(\nabla \tilde{p}(\mathbf{0}))^{\top} V_\pi \bm{1})^2.$
\begin{proof}[Proof of Theorem \ref{theorem:convergence_OP_LRD}]
We apply Theorem \ref{thm.2}. 
Since $b_j\sim j^{d-1}$ and $(\mathbf{Z}_{j})_{j \in \mathbb{Z}}$ forms an i.i.d. sequence with variance $\mathbf{E}$, Lemma \ref{lemma:vectorization} shows that the multivariate linear process $\Xb_t = (X_t, \ldots, X_{t+r-1})^\top=\sum_{j=0}^\infty B_j \mathbf{Z}_{t-j}$ satisfies all the assumptions of Theorem \ref{thm.2}, and it inherits the decay characteristics in the sense that $B_j \sim j^{d-1} I_r$ as $j \to \infty,$ and $V_\pi I_r $ is invertible.  Moreover, in Proposition \ref{prop:Lipschitzness_gradient}, we show that in this case \eqref{eq:LIP-condition1} is satisfied, such that Theorem \ref{thm.2} can be applied for $\left(V_\pi \Xb_t\right)_{t\geq 1}$ and $\mathbf{u}=\mathbf{0}$, leading to
\begin{align}
 n^{1/2-d}( \hat{p}(\pi)-p(\pi))
    \xrightarrow{\D} \N\bigg(0,\frac{\Gamma(d)^2 }{\Gamma(2d+2)\cos (\pi d)}( \nabla \tilde{p}(\mathbf{0}))^{\top} V_\pi E V_\pi^\top \nabla \tilde{p}(\mathbf{0}) \bigg)\;.
\end{align}
\end{proof}
We conclude this section by providing two examples. In the first example, the underlying time series $(\xi_t)_{t\geq0}$ exhibits long-range dependence but the increment process $(X_t)_{t\geq 1}$ is short-range dependent. This allows us then to get the parametric $\sqrt{n}$ convergence rate.  The second example shows that Theorem \ref{theorem:convergence_OP_LRD} applies to a class of FARIMA processes. 
\begin{example}(Short-range dependent increments)
\label{example:degeneracy_FARIMA}
Let $d\in (0,1/2)$ and consider a linear process $\xi_t = \sum_{j=0}^\infty a_j Z_{t-j}$ with $a_j\sim j^{d-1}$ with $a_0\neq 0$ and $Z_1$ admitting a continuous and bounded probability density function. We now show that the increment process $X_t=\xi_t-\xi_{t-1}$ is a short-range dependent linear process. Indeed, it holds that
\begin{align*}
    X_t = \sum_{j=0}^\infty b_j Z_{t-j}\; \quad \text{with} \ b_j = \begin{cases}
        a_j- a_{j-1}\;, \quad j\geq 1 \\
        a_0\neq 0\;, \quad j=0\\
    \end{cases}\;.
\end{align*}
Moreover, $a_j \sim j^{d-1}$ as $j\to \infty$ and the mean value theorem imply $b_j \sim (d-1)j^{d-2}$ as $j\to \infty.$ Because of $d\in (0,1/2),$ $d-2 \in (-2,-3/2)$ and $(b_j)_{j \in \NN}$ is summable. Since $Z_1$ has continuous and bounded probability density function, Theorem \ref{theorem:convergence_OP_SRD} yields for any ordinal pattern $\pi$,
     \begin{equation*}
     \label{non_degeneracy_farima}
         \sqrt{n}\left(\hat{p}_n(\pi) - p(\pi) \right)\xrightarrow{\D} \N(0, \sigma_\pi^2) \;,
     \end{equation*}
     with $\sigma_\pi^2$ defined as in Theorem \ref{theorem:convergence_OP_SRD}.
\end{example}
\begin{example}(Long-range dependent increments)
We consider a FARIMA\((p,d,q)\) process with \(d < 1/2\), defined by \(\phi(B)X_t = \theta(B)(I-B)^{-d}Z_t\), where \(B^jX_t = X_{t-j}\). As shown by \citet{pipiras_taqqu_2017}, if \(\phi(z)\) and \(\theta(z)\) have no common roots, and \(\phi(z)\) has no zeros on the unit circle, \((X_t)_{t \geq 1}\) exhibits long-range dependence and can be expressed as:
\[
X_t = \sum_{j=0}^\infty b_j Z_{t-j}, \quad b_j \sim \frac{\theta(1)}{\phi(1)}\frac{j^{d-1}}{\Gamma(d)}.
\]
Then, the processes $\xi_t=\sum_{j=1}^t X_j$ satisfies the conditions of Theorem \ref{theorem:convergence_OP_LRD}.
\end{example}

\subsection{Some related works}\label{subsec:rel_work}
Empirical processes of linear models have been widely explored in the literature. Notable contributions include the work of \citet{Hsing}, who developed asymptotic expansions for the empirical process of long-range dependent linear processes, and \citet{giraitis1999central}, who established functional non-central limit theorems for linear processes with long-range dependence. These results, which rely on the reduction principle, were further extended by \citet{wu2003empirical}.

Limiting theorems for empirical processes in linear time series with short-range dependence have been studied by various authors. \citet{Furmanczyk} presented a central limit theorem for \( g(\mathbf{X}_j) \), where \( (\mathbf{X}_t)_{t\geq 1} \) is a multivariate linear process. Under mild conditions on the subordinated function \( g \) and the finite second moment of the innovations \( \mathbf{Z}_1 \), Furmańczyk concluded that 
\(\frac{1}{\sqrt{n}}\sum_{j=1}^{\lfloor n\tau \rfloor} g(\mathbf{X}_j)\Longrightarrow (B(\tau))_{\tau \in [0,1]}\). 
A similar setting was discussed in \citet{wu2002central}, who expanded on the results of \citet{Ho1997LimitTF}. In their Theorem 4.1, Ho and Hsing presented a central limit theorem for univariate linear processes with short-range dependence, imposing technical conditions and convergence similar to those in our Theorem \ref{theorem:convergence_OP_SRD}, specifically the condition \(\sum_j |a_j|<\infty\), where \(a_j\) are the coefficients of the process. \citet{wu2002central} obtained the same result under less stringent conditions. Additionally, in his Theorem 4, he derived the limiting distribution for empirical processes with the indicator function \( g(x_1, \ldots, x_p)=\bm{1}_{\{x_1\leq s_1, \ldots, x_p\leq s_p \}} - \PP( X_1\leq s_1, \ldots, X_p\leq s_p ) \), which directly applies to \eqref{ordinal_pattern_estimator}. Wu assumed that the characteristic function \( \phi_Z \) of \( Z_1 \) satisfies \( \int |\phi_Z(t)|^r\,dt < \infty \) for some \( r \in \mathbb{N} \). He also defined \( A_k(\delta) := \sum_{t=k}^{\infty} |a_t|^\delta \), and required that \( \sum_{n=1}^{\infty} \sqrt{A_n(\delta)/n} < \infty \) to derive \eqref{eq:asymptotic_normality_OP_SRD}. With the additional assumption \( A_n(\delta) = \mathcal{O}(n^{-q}) \) for some \( q > 1 \), he further derived \eqref{eq:uniform_asymptotic_normality_op}.
While Wu's approach imposes less stringent conditions on the innovations of the process (allowing, for example, discrete innovations), the only assumption we impose on the coefficients is \(\sum_j |a_j|<\infty\), which is weaker than Wu's condition \( \sum_{n=1}^{\infty} \sqrt{A_n(\delta)/n} < \infty \). Indeed, Wu's Lemma 1 shows \(\sum_j|a_j|<\sum_{n=1}^{\infty} \sqrt{A_n(\delta)/n}\). For instance, the imposed conditions in this paper are suitable for processes with summable coefficients and innovations with \( L^1 \) characteristic functions, which, by the well known inversion formula, would admit continuous and bounded densities, thereby fulfilling our conditions. Conversely, there are examples where Wu's conditions hold, such as when the process is \( m \)-dependent and the innovations are sufficiently regular. Another example are uniform densities, which are not continuous but still satisfy Wu's conditions.

For the case of short-range dependence, \citet{schnurr2017testing} establish the asymptotic distribution of the ordinal patterns estimator $\hat{p}_n(\pi)$ for 1-approximating functionals of the absolutely regular process $(Z_j)_{j\in \NN}$. 
  They assume that \((X_t)_{t \geq 1}\) is a 1-approximating functional of \((Z_{j})_{j \in \mathbb{N}}\), with summable mixing coefficients \((\beta_j)_{j \in \mathbb{N}}\) (i.e., \(\sum_{j} \beta_j < \infty\)), indicating that \((X_t)_{t\geq 1}\) exhibits short-range dependence. Furthermore, if the 1-approximating coefficients \((k_j)_{j \in \mathbb{N}}\) of $(X_t)_{t\geq 1}$ are such that \(\sum_{j} \sqrt{k_j} < \infty\), and the distributions of \(X_i - X_1\) are  Lipschitz continuous for \(i \in \{1, \ldots, r\}\) then, \eqref{eq:asymptotic_normality_OP_SRD} holds. 
  However, their result does not cover the full spectrum of short-range dependent processes. For instance, consider $X_t= \sum_{j=0} a_j Z_{t-j}$ with $a_0=1 $ and $a_n = \frac{1}{n^{\alpha}}$. For $\alpha>1$,  $(a_j)_{j \in \NN}\in \ell^1 \subset \ell^2$. Suppose the innovations are i.i.d. and $Z_1\sim f$ where $f$ is continuous and bounded. Under these conditions, the assumptions of Theorem \ref{theorem:convergence_OP_SRD} are satisfied for all $\alpha>1$, however $\sum_{j}\sqrt{k_j}<\infty$ only holds for $\alpha > 5/2\;.$ To see this, applying their Lemma 1, the 1-approximating sequence $(k_m)_{m \in \NN}$ for $X_t$, is given by $k_m= \left( \sum_{n=m+1}^\infty a_n^2 \right)^{1/2}\;.$ Therefore,
  \begin{align*}
k_m^2=\sum_{n=m+1}^\infty a_n^2 = \sum_{n=m+1}^\infty \frac{1}{n^{2\alpha}} \sim \int_{m}^\infty x^{-2\alpha }\, dx = \frac{1}{2\alpha-1} m^{1-2\alpha}\;.
  \end{align*}
Thus, the condition $\sum_{m=1}^\infty \sqrt{k_m} < \infty$  is satisfied only for $ 1-2\alpha< -4 $, that is, $\alpha > 5/2\;.$

 Lastly, \citet{Beran_Telkmann} demonstrated a reduction principle for multivariate empirical processes under long-range dependence. Their work provided a heuristic proof. Similar to their approach, and based on the martingale decomposition of Ho and Hsing, we reprove a pointwise multivariate reduction principle tailored to the ordinal patterns map. Importantly, our result does not hold uniformly but for each fixed point. This choice is motivated by the fact that, as discussed, the  estimator for the ordinal pattern probabilities \eqref{ordinal_pattern_estimator} is the empirical sum process evaluated at \(\mathbf{0}\). This simplification also allows us to eliminate the smoothness assumption on the innovations imposed by \citet{Beran_Telkmann}.

\section{Applications of ordinal patterns}\label{sec:appl}
 Electroencephalography (EEG) is a non-invasive method widely utilized in medical and scientific research to measure and analyze the brain's electrical activity. EEG records the electrical potentials generated by neural activity through electrodes placed on the scalp. An application of EEG is the identification and characterization of sleep stages.
According to the classical methodology of \citet{kales1968manual}, adult sleep is divided into six stages: wakefulness (W), stage 1 (S1), stage 2 (S2), stage 3 (S3), stage 4 (S4), and rapid eye movement (REM). Each stage corresponds to distinct patterns of brain activity observed in EEG recordings. 
The classification of sleep stages is based on the segmentation of EEG recordings into non-overlapping intervals, referred to as epochs (typically 30 seconds); see \citet{kales1968manual}. For each epoch, the corresponding sleep stage is determined by experts.
For a comprehensive overview of techniques used to analyze EEG data see \citet{ZHANG2024651} and \citet{Gonen}.

Ordinal patterns have recently been used for sleep stage classifications in EEG data. \citet{sinn2013segmentation} classified sleep-stages by computing the distribution of ordinal patterns in different parts of the time series, and then detect the locations of breaks by the Maximum Mean Discrepancy statistics of the sequence of ordinal pattern distributions. Another methodology that has been receiving great attention in the field is the analysis of EEG via permutation entropy (PeEn). Introduced by \citet{Bandt-Pompe}, PeEn is defined as the Shannon entropy of the distribution of ordinal patterns of a fixed length, i.e. as the quantity $\text{PeEn}(p(\pi^1), \ldots, p(\pi^{(r+1)!}))= -\sum_{i=1}^{r!} p(\pi^i) \log( p(\pi^i)) $, where $\{\pi^1,\ldots, \pi^{(r+1)!}\}=\mathcal{S}^r$ is the set of permutations of length $r+1$. The value of the PeEn is computed for each epoch, and considerable differences in value can be observed for different sleep stages. Thus, the complex dynamics of EEG is then measured by looking at the series of permutation entropies across non-overlapping epochs of EEG recordings.  A more refined analysis of the permutation entropy is due to  \citet{BergerSebastian}. Their data analysis of the PeEn with $r=2$ in one stage of EEG data shows that, for $\hat{p}$ being the ordinal patterns estimator \eqref{ordinal_pattern_estimator}, 
\begin{equation}
   \EE[\hat{p}(0,1,2)]\approx \EE[\hat{p}(2,1,0)]\quad \text{and}\quad \EE[\hat{p}(0,2,1)]\approx \EE[\hat{p}(2,1,0)]\approx \EE[\hat{p}(1,2,0)]\approx \EE[\hat{p}(1,0,2)]\;. \label{eq:symmetries_EEG} 
\end{equation}
On a population level, assuming stationary Gaussian observations is sufficient for equality, in the sense that if $(\xi_t)_{t\geq 0}$ is stationary Gaussian then
\begin{equation}
    p(0, 1, 2)=p(2,1,0)  \quad \text{and}\quad p(0,2,1)=p(2,1,0)=p(1,2,0)=p(1,0,2)\;.
    \label{eq:symmetries_gaussian} 
\end{equation}
Thus, in order to study the dynamics of PeEn, we can study the set of patterns $$\mathcal{T}=\{ (0,2,1),(2,1,0),(1,2,0),(1,0,2)\}\;.$$ An element of this set occurs with probability $q:= p(0,2,1)+p(2,1,0)+p(1,2,0)+p(1,0,2)$.
Likewise, the probability to observe the all-raising pattern $(0,1,2)$ or the all-falling pattern $(2,1,0)$ is $1-q$. It follows that the permutation entropy
$$\text{PeEn}=q\log \frac{4}{q} + (1-q) \log \frac{2}{1-q}\;$$
only depends on $q.$
The corresponding estimator of $q$ for an epoch \(\left(\xi_t, \ldots, \xi_{t+m+1}\right)\) of \(m+2\) observations is
\begin{align}
    \hat{q}_m:=& \frac{1}{m} \sum_{i=0}^{m-1} \sum_{\gamma \in \mathcal{T}} \bm{1} \left(\{\Pi(\xi_{t+i}, \xi_{t+i+1}, \xi_{t+i+2}) = \gamma \}\right)=\hat{p}_m(0,2,1)+\hat{p}_m(2,1,0)+\hat{p}_m(1,2,0)+\hat{p}_m(1,0,2)\;.
    \label{eq:turning_rate_1}
\end{align}
The probability $q = \sum_{\gamma \in \mathcal{T}} p(\gamma)$ is called \textit{turning rate} and the quantity $\hat{q}_m$ is called \textit{turning rate estimator}; see \citet{bandt2017crude, bandt2020order}.  Based on the turning rate, Bandt empirically identified different stages of sleep in EEG recordings that closely aligned with expert stage annotations.
For this, Bandt partitions the EEG data recording into consecutive blocks of a length corresponding to 30s, i.e. $(\xi_1,\ldots, \xi_{30s}), (\xi_{31s}, \ldots, \xi_{60s}),\ldots$ and  computes the turning rate estimate for each block via \eqref{eq:turning_rate_1}. 
The corresponding turning rate series  is plotted in Figure \ref{turning-rate-ins4}, where the  sleep cycles become immediately visible.
 \begin{figure}[ht]
     \centering
     \includegraphics[scale=0.25]{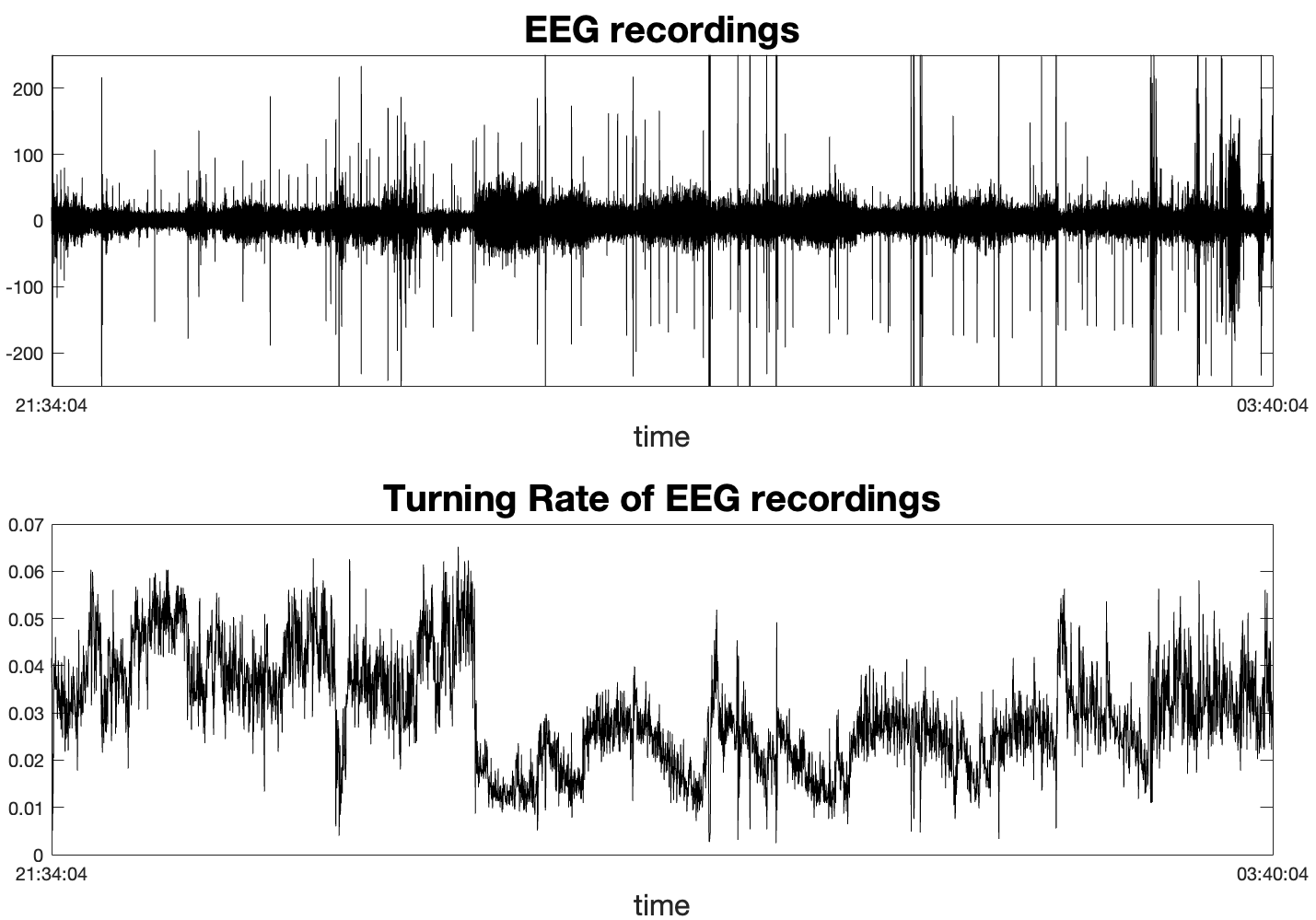}
     \caption{Top: EEG recordings of a healthy individual (4th patient from the CAP Sleep Database \citet{terzano2001atlas}), originally sampled at 512 Hz, resulting in a time series of 19,553,792 data points. Due to the high resolution and length, detailed features are difficult to discern.  Bottom: Corresponding turning rate series, where each data point represents an 8-second segment of the EEG recording. The sleep cycles are visible.}
     \label{turning-rate-ins4}
 \end{figure}
 
\citet{Kedem} showed that the turning rate of a time series is closely related to the centroid of its power spectrum. Most notably, if $\xi_0, \ldots, \xi_{n+1}$ is a zero-mean stationary process, the turning rate is equivalent to the well-established Zero-Crossing estimator.
For \(( \xi_t)_{t \geq 0}\) a stationary zero-mean Gaussian time series, using Lemma 1 from  \citet{sinn2011estimation} and the results of \citet{Kedem}, the following can be shown
\begin{equation}
    \cos{(\pi \EE[\hat{q}_m])} = \rho(1) = \omega^B\;,
\label{Spectral_centroid_Hergoltz}
\end{equation}
where $\rho(1)$ is the autocorrelation of $(X_t)_{t\geq 1}$ at lag 1 and \(\omega^B\) is a frequency of the spectrum of \((X_t)_{t \geq 1}\) known as the \textit{Spectral Centroid} or \textit{Barycenter of the spectrum}. This relationship, expressed in equation \eqref{Spectral_centroid_Hergoltz}, explains that a change in the autocorrelation \(\rho(1)\) corresponds to a change in the frequency that dominates the spectrum. 
To illustrate this concept, Figure \ref{turning-rate-ins4} shows how the turning rate tracks the dominant frequencies over time. For example, when working with signals dominated by specific frequency bands (such as EEG signals characterized by alpha, beta waves, etc.), variations in the turning rate reveal shifts in the most prominent frequency components. 

The analysis in \citet{bandt2017crude} emphasizes the visual inspection of turning rate plots across multiple channels, demonstrating that these plots exhibit significant overlap with the corresponding doctor-annotated sleep stage diagrams. Complementary to Bandt's analysis, in the next section we rigorously evaluate the likelihood of transitions between stages with formal statistical confidence. Specifically, we introduce a hypothesis test designed to detect changes in the mean of the turning rate series. In the case of Gaussian distributions, a significant shift in the mean directly corresponds to a change in the parameter $\rho(1)$, signaling a transition to a different sleep stage. 

\subsection{Change-point detection via turning rate analysis}
In this section, we address the problem of testing whether a given time series \(\xi_0, \ldots, \xi_{n+1}\) exhibits stationarity in its increments \(X_1, \ldots, X_{n+1}\), or whether there is a structural break in their distribution at some unknown point \(1 \leq k^\ast \leq n+1\). The core idea of the test is to analyze and compare the ordinal pattern distributions across different segments of the time series. Under the null hypothesis, the stationarity of the increments ensures the stationarity of their ordinal patterns. Consequently, any observed change in the ordinal pattern distribution must reflect a change in the underlying time series distribution. This insight motivates the turning rate estimator as a test statistic.

If the increments $X_1, \ldots, X_{n+1}$ form a stationary sequence, the estimator \eqref{eq:turning_rate_1} can be expressed in terms of $\mathbf{X}_t=(X_t, X_{t+1})^\top$ as 
  \begin{equation}
      \hat{q}_n -q = \frac{1}{n}\sum_{t=0}^{n-1} h( \mathbf{X}_t)\;, \quad h(\mathbf{X}_t):=  \sum_{\gamma \in \mathcal{T}} \bm{1}\left( V_{\gamma   }\Xb_t \leq \mathbf{0} \right)-q\;,
\label{turning_rate_estimator_version2}
  \end{equation}
where the matrices $V_{\gamma }$ for $\gamma \in \mathcal{T}$ are 
 \begin{align}
 \label{eq:the_four_matrices}
     V_{(0,2,1) }= \left(\begin{array}{cc}
          0&1  \\
         -1 &-1 
     \end{array} \right)\;, \quad   V_{ (2,0,1) }= \left(\begin{array}{cc}
          -1&-1  \\
         1&0 
     \end{array} \right)\;,  \quad   V_{(1,2,0) }= \left(\begin{array}{cc}
          -1&0  \\
         1&1 
     \end{array} \right)\;, \quad   V_{(1,0,2) }= \left(\begin{array}{cc}
          1&1  \\
         0&-1 
     \end{array} \right)\;.
 \end{align}
We will formalize the test in terms of ordinal patterns by identifying the change-points in the corresponding turning rate series. 
Given time series data $\xi_0, \ldots, \xi_{n+1}$, the \textit{turning rate series} is defined as the collection of $n_b = \left\lfloor \frac{n+2}{m+2} \right\rfloor$ random variables $\hat{q}_{1,m}, \ldots, \hat{q}_{n_b,m}$, where each $\hat{q}_{j,m}$ represents the turning rate (relation~\eqref{eq:turning_rate_1}) computed over non-overlapping, consecutive blocks of length $m+2$ extracted from $\xi_0, \ldots, \xi_{n+1}$. Formally, $\hat{q}_{j,m}$ is defined by
\begin{equation}
    \label{turning_rate_block}
    \hat{q}_{j,m} = \frac{1}{m} \sum_{i=0}^{m-1} \sum_{\gamma \in \mathcal{T}} \bm{1} \Big( \Pi\big(\xi_{(j-1)(m+2)+i}, \xi_{(j-1)(m+2)+i+1}, \xi_{(j-1)(m+2)+i+2}\big) = \gamma \Big)\;, \quad \text{for } j = 1, \ldots, n_b\;.
\end{equation}
We consider the following test problem:
\begin{align}
    & \mathcal{H}_0:\, \EE[\hat{q}_{1,m}]=\cdots=\EE[\hat{q}_{n_b,m}] \quad \text{vs} \nonumber \\
  & \mathcal{H}_1: \, \EE[\hat{q}_{1,m}]=\cdots=\EE[\hat{q}_{\lfloor n_b  \tau^\star  \rfloor, m}]\neq \EE[\hat{q}_{\lfloor n_b  \tau^\star  \rfloor +1,m}]=\cdots=\EE[\hat{q}_{  n_b,m}]   \quad \text{for some } \tau^\star \in [0,1] \;. 
   \label{test2}
\end{align}
The test problem (\ref{test2}) is framed akin to a conventional mean change-point detection problem. 
For this purpose, we can employ the CUSUM statistic
\begin{align}
\label{cusum}
    \max\limits_{k=1, \ldots, n_b -1}\left|\sum\limits_{j=1}^{k}\hat{q}_{j,m}-\frac{k}{n_b}\sum\limits_{j=1}^{n_b}\hat{q}_{j,m}\right|\;,
\end{align}
as our test statistic. The asymptotic distribution of (\ref{cusum}) is typically obtained through an application of the continuous mapping theorem to
\begin{align}
   \frac{m}{\sqrt{n}} \sum\limits_{j=1}^{\left\lfloor n_b \tau \right\rfloor}(\hat{q}_{j,m} -q), \, \text{as } \, n\rightarrow \infty\;,
    \label{averages_turning_rate2}
\end{align}
where $\tau \in [0,1]$. The convergence result for (\ref{averages_turning_rate2}) is detailed in Theorem \ref{theorem:asymptotic_turning_rates_series} of Appendix \ref{appendix:turning_rate}. The asymptotic limit of the CUSUM statistics is the content of the next theorem. 
\begin{theorem}
\label{theorem:cusum_statistics_limit} 
Let $\xi_0, \ldots, \xi_{n+1}$ be a time series whose increments $X_1,\ldots, X_{n+1}$ form a linear process $X_t=\sum_{j=0}^\infty a_j Z_{t-j}$ with $\sum_{j=0}^\infty |a_j|< \infty$ and $Z_1$ admitting a continuous and bounded density and finite second moment $\EE \left[|Z_1|^2\right]<\infty$.  Consider the turning rate generated by blocks of size $m+2,$ and corresponding number of blocks $n_b=\lfloor n/m\rfloor$. If $m/\sqrt{n}\to \infty$, then
\[
\frac{m}{\sqrt{n}}\max_{k=1, \ldots, n_b-1} \left| \sum_{j=1}^k \hat{q}_{j,m} - \frac{k}{n_b} \sum_{j=1}^{n_b} \hat{q}_{j,m} \right| \xrightarrow{\D} \sigma \sup_{\tau \in [0,1]} \left| B(\tau) - \tau B(1) \right|\;,\,\, \text{as }\, n\to \infty\;,
\]
with variance  
\begin{equation}
    \sigma ^2 =  \Var \left( \sum_{\gamma \in \mathcal{T}} \bm{1}\left( \Pi(\xi_0,\xi_1,\xi_2)=\gamma   \right) \right) + 2 \sum_{j=1}^\infty \Cov \left( \sum_{\gamma \in \mathcal{T}} \bm{1}\left( \Pi(\xi_0,\xi_1,\xi_2)=\gamma\right), \sum_{\gamma \in \mathcal{T}} \bm{1}\left( \Pi(\xi_{j},\xi_{j+1},\xi_{j+2})=\gamma\right) \right)\;.
    \label{eq:variance_cusum}
\end{equation}
\end{theorem}
\begin{remark}
If \(X_1, \ldots, X_{n+1}\) forms a Gaussian time series, test \eqref{test2} becomes equivalent to testing for a change in the autocorrelation parameter \(\rho(1)\), as shown in \eqref{Spectral_centroid_Hergoltz}. In this context, a significant shift in the mean of the turning rate series directly corresponds to a change in \(\rho(1)\). For EEG time series, such a shift is indicative of a transition between different sleep stages.  
\end{remark}

\subsection{Variance estimation}

The distribution of the stochastic limit $\sup_{\tau \in [0,1]} |B(\tau) - \tau B(1)|$ can be approximated through Monte-Carlo simulations. However, the long-run variance in \eqref{eq:variance_cusum} is generally unknown and requires estimation. 
A standard approach to estimating $\sigma^2$ involves kernel-based methods. However, these methods are sensitive to the choice of the kernel bandwidth, with data-dependent bandwidth selection often leading to nonmonotonic statistical power, as shown in \citet{vogelsang1998} and \citet{crainiceanu2007spectral}.
To address these limitations, we estimate $\sigma^2$ using the self-normalization technique introduced by \citet{shao2010change_points} and \citet{shao:2010}; see also \citet{betken2016testing}. 

Taking the possibility  of a structural change at time $k$ into consideration,
 a normalization for the two-sample CUSUM statistic is obtained by combining the values of empirical variances computed with respect to  the separate  samples $\hat{q}_{1,m}, \ldots, \hat{q}_{k,m}$ and  $\hat{q}_{k+1,m}, \ldots, \hat{q}_{n_b,m}$.
Accordingly, we define
\begin{align}\label{V_{k,n}}
 V^2_{k, n_b}:=\frac{1}{n_b}\sum_{t=1}^k S_t^2(1,k)+\frac{1}{n_b}\sum_{t=k+1}^n S_t^2(k+1,n_b)
\end{align}
with
\begin{align*}
S_{t}(j, k):=\sum_{h=j}^t\left(\hat{q}_{h,m}-\bar{q}_{j, k}\right),  \quad  
\bar{q}_{j, k}:=\frac{1}{k-j+1}\sum_{t=j}^k\hat{q}_{t,m},
\end{align*}
as  normalizing sequence
and  we  define the self-normalized CUSUM statistic  by
\begin{align}
SC_{n_b}:=\max_{1\leq k\leq n_b-1}\frac{\left|\sum_{j=1}^k\hat{q}_{j,m}-\frac{k}{n}\sum_{j=1}^{n_b}\hat{q}_{j,m}\right|}{V_{k,n_b}^2} \;.
\label{eq:SC_n}
\end{align} 
For testing the hypothesis of a change in the level of the turning rate series on the basis of the self-normalized CUSUM statistic $SC_{n_b}$, we need to set critical values for a corresponding hypothesis test. For this purpose, we establish the asymptotic distribution of the statistic as a corollary of
Theorem \ref{theorem:convergence_OP_SRD}:
\begin{corollary}
\label{self_normalized_cusum_turning_rate}   Assume that $\xi_0, \ldots, \xi_{n+1}$ satisfies the assumptions of Theorem \ref{theorem:cusum_statistics_limit}.
Consider the turning rate series generated by blocks of size \( m+2 \) and \( n_b  \) blocks, denoted by  $\hat{q}_{1,m}, \ldots, \hat{q}_{n_b,m}$.  If $m/\sqrt{n}\to \infty$, then as $n\to \infty$ 
\begin{align}
   SC_{n_b} \overset{\D}{\longrightarrow} \sup_{\tau \in [0,1]} \frac{|B(\tau)-\tau B(1)|}{ \left[\int_0^\tau \left( B(s) - \frac{s}{ \tau} B(\tau) \right)^2 \, ds+ \int_{\tau}^1 \left( B(s) - B(\tau) -\frac{s-\tau}{1-\tau} \left( B(1)-B(\tau) \right)\right)^2 \, ds\right]^{1/2}}\;.
   \label{limit_of_self_normalized_cusum}
\end{align}
\end{corollary}
The proof can be found in Appendix \ref{appendix:turning_rate}.

\section{Simulation studies}\label{sec:sim}
We complement the theoretical results with simulation studies and an application to real data. We illustrate Corollary \ref{self_normalized_cusum_turning_rate} using two simulated MA(1) time series: one without a change-point and one where the moving average parameter changes, corresponding to a shift at lag 1 of the autocorrelation. Further, we examine the power of the test on simulated AR(1) time series with lengths of 500, 1000, and 2000, introducing breaks at 1/4, 1/3, and 1/2 of the data under various non-normally distributed innovations. Finally, we apply the test to EEG data to detect transitions between two sleep stages. 
\begin{example}[Setting A]
   We simulated two MA(1) time series, defined as \( X_t = Z_t + \theta Z_{t-1} \), each consisting of 5000 data points with normally distributed innovations (see Figure \ref{Fig:histograms}). For the first series, the MA parameter \( \theta \) is set to \( 0.4 \). In the second time series, $\theta$ changes from $0.4$ to $0.7$ after 2500 observations. The value $SC_{n_b}$ is computed for both series.  Figure \ref{Fig:histograms} shows the distribution of the test statistics based on 1000 repetitions. 
 \begin{figure}[htbp]
    \centering
    \begin{minipage}{0.45\textwidth}
        \centering
        \includegraphics[width=\textwidth]{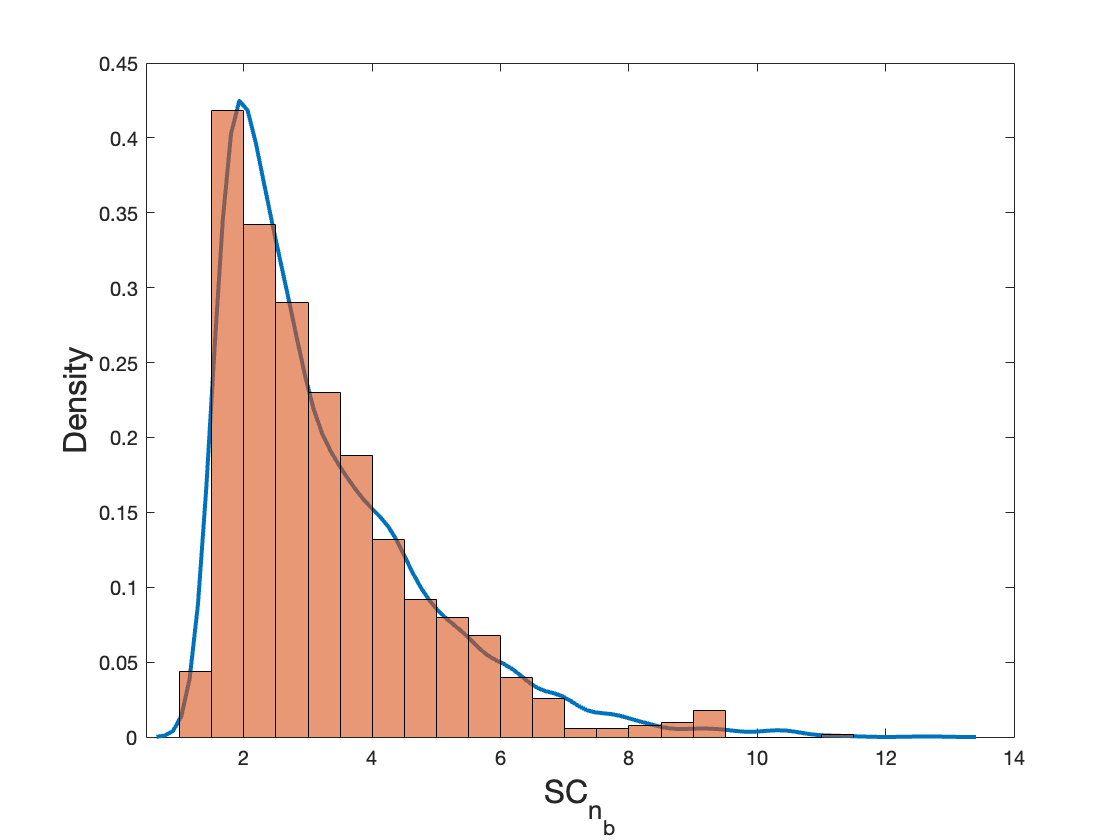} 
    \end{minipage}
    \hfill
    \begin{minipage}{0.45\textwidth}
        \centering
        \includegraphics[width=\textwidth]{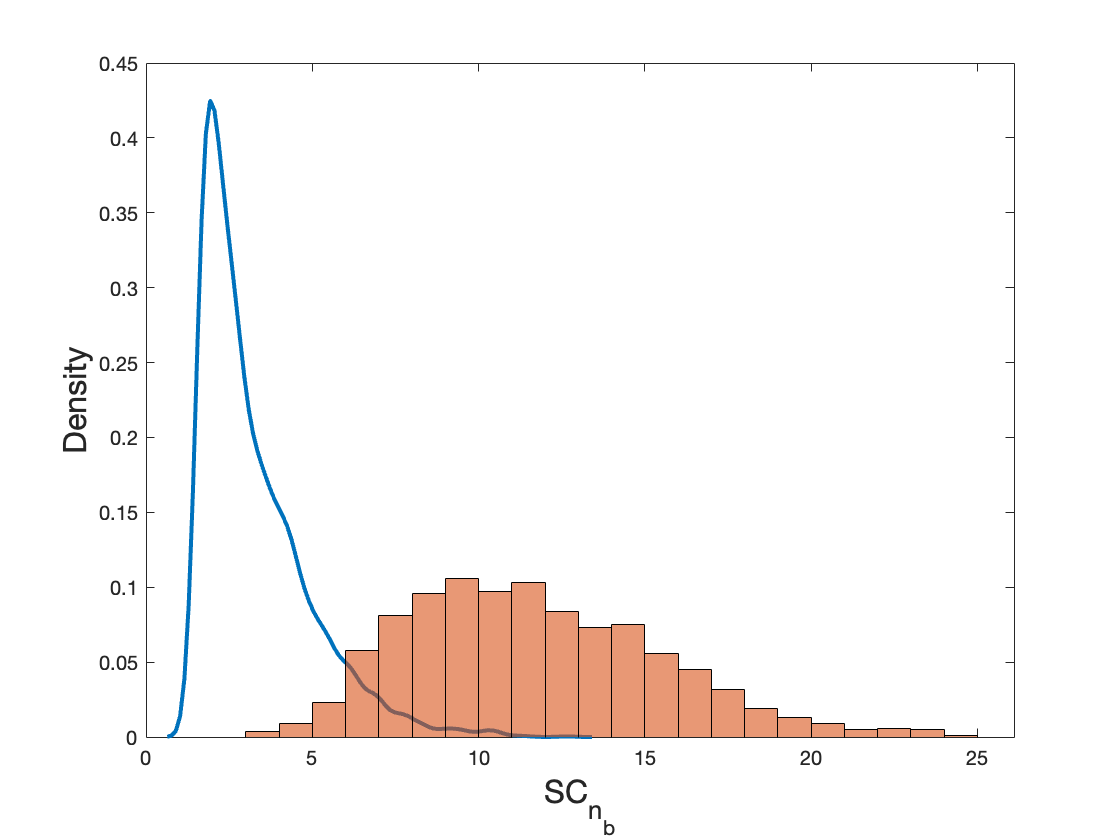} 
    \end{minipage}
     \caption{Histogram of $SC_{n_b}$ for $n=5000$ and 1000 simulations of MA(1)  without (left) and with (right) change of $\rho(1)$.  In the right plot, the autoregressive parameter changes from $\theta=0.4$ to $\theta=0.7$ after 50\% of the observations. 
     Under the null hypothesis the estimated 0.95 quantile is 6.335. }
     \label{Fig:histograms}
\end{figure}
\end{example}
\begin{example}[Setting B]
We simulate an AR(1) process $(X_t)_{t\geq 1}$ of the form
\begin{equation*}
    X_t = \begin{cases}
        \phi_1 X_{t-1} + Z_t \quad t=1, \ldots,\lfloor (n+1)\tau \rfloor\;,\\
        \phi_2 X_{t-1} + Z_t \quad t=\lfloor (n+1) \tau \rfloor\ +1, \ldots, n+2
    \end{cases}
\end{equation*}
    and set $\xi_t=\sum_{i=1}^t X_i$. For $h=\phi_2-\phi_1$ and $\phi_1=0.4$, Figure \ref{fig:power_test} provides the frequency of detected changes for the  values $h\in \{0.1,0.2,0.3,0.4,0.5\}$.
Further, for fixed $h=0.4$, we consider time series of lengths 500, 1000 and 2000 data points. We analyze  breaks 
after a fraction $\tau \in \{ 1/10, 1/4, 1/2 \}$ of the data for different innovation distributions. The frequency of detected changes are summarized in Table \ref{tab:table1}.

\begingroup
\renewcommand{\arraystretch}{1.2} 
\begin{table}[ht]
\centering
\setlength{\tabcolsep}{5pt} 
\begin{tabular}{cccccccccccc}
\thickhline
\multicolumn{4}{c}{$n=500$} & \multicolumn{4}{c}{$n=1000$} & \multicolumn{4}{c}{$n=2000$} \\ 
break & $\mathcal{N}(0,1)$ & $t_2$ & Lap$(0,4)$ &
break & $\mathcal{N}(0,1)$ & $t_2$ & Lap$(0,4)$ &
break & $\mathcal{N}(0,1)$ & $t_2$ & Lap$(0,4)$ \\
\hline
50  & 10.2 & 11.4 & 7.8 & 100  & 13.5  & 16.1   & 13.6  & 200  & 29.4 & 31 & 30.5 \\
125 & 40   & 52.3  & 46  & 250  & 69 & 75 & 69.2 & 500  & 91.6 & 95.1 & 91.2 \\
250 & 70   & 74.8 & 73.2 & 500  & 93.2 & 94.8   & 94.3 & 1000 & 99.8 & 99.7 & 99.5   \\
\thickhline
\end{tabular}
\caption{\label{tab:table1} Frequencies of detected changes in an AR(1) time series and for $h=\phi_2-\phi_1=0.4$.  }
\end{table}
\endgroup

\begin{figure}[htbp]
    \centering
    \begin{minipage}{0.45\textwidth}
        \centering
        \includegraphics[width=\textwidth]{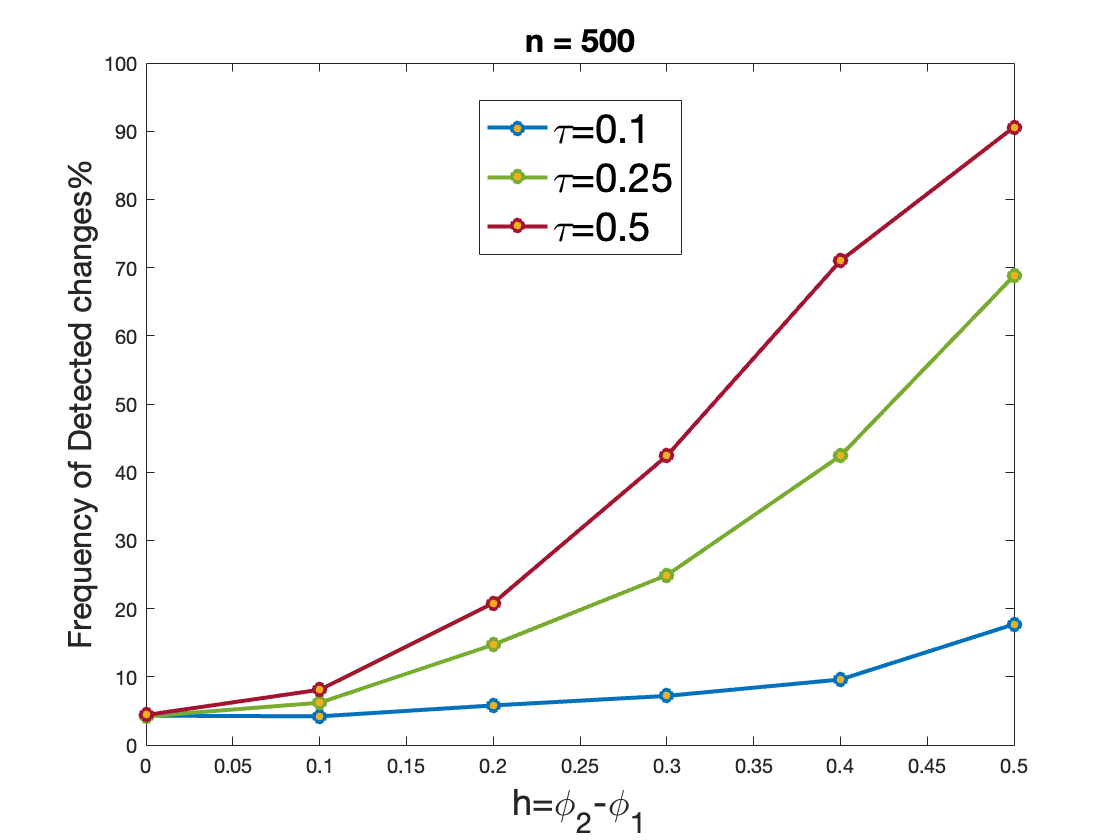} 
        
    \end{minipage}
    \hfill
    \begin{minipage}{0.45\textwidth}
        \centering
        \includegraphics[width=\textwidth]{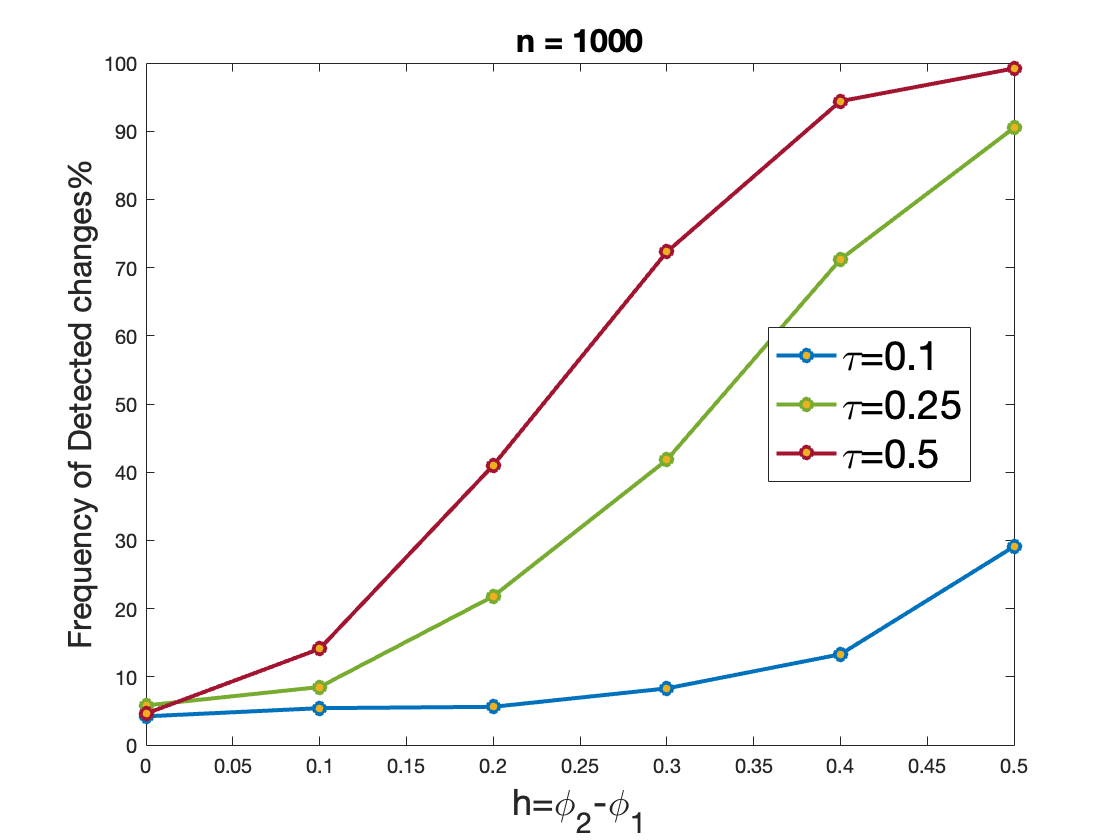} 
    \end{minipage}
     \caption{ Power of the test for different values of  $h=\phi_2-\phi_1$ with $n=500$ and $n=1000$ data points and Laplace distributed innovations.
     The curves correspond to different values of $\tau$, where changes occur at $1/10$ for the blue curve $(\tau=0.1)$, at $1/4$ for the green curve $(\tau=0.25)$, and at $1/2$ for the red curve $(\tau=0.5)$. $\phi_1$ is set to $0.4$. For $h=0$ (no change), the power corresponds to the significance level  $5\%$.   }
     \label{fig:power_test}
\end{figure}
\end{example}
\begin{example}[Real Data]
The upper plot of Figure \ref{EEG_from_REM_to_S2} depicts a segment of EEG recordings from a single patient sourced from the dataset by \citet{terzano2001atlas}. The time series comprises data points sampled at 512 Hz over a 39-minute interval. In this dataset, each 30-second batch is manually labeled with the corresponding sleep stage by an expert. In the depicted segment, the patient transitions from the REM phase to the S2 phase. Our objective is to employ the proposed method to statistically validate this transition and potentially pinpoint the exact time of occurrence. 

The lower part of Figure \ref{EEG_from_REM_to_S2} shows the change in the corresponding turning rate series (dashed red line in the upper figure). Assuming the EEG to be Gaussian and stationary under the null hypothesis of no change, we reject the null hypothesis at significance level 0.05 and  with a p-value of $7.29\times 10^{-5}\;.$ The same test applied to the corresponding two subsamples (from the beginning to the change-point and from the change-point to the end) fails to reject the null hypothesis with $p$-values $0.986$ and $0.138$, respectively. 
The source code for the simulations is available on GitHub at \url{https://anonymous.4open.science/r/Turning-rate-time-series-424E}. 

\begin{figure}
    \centering
    \includegraphics[width=1\linewidth]{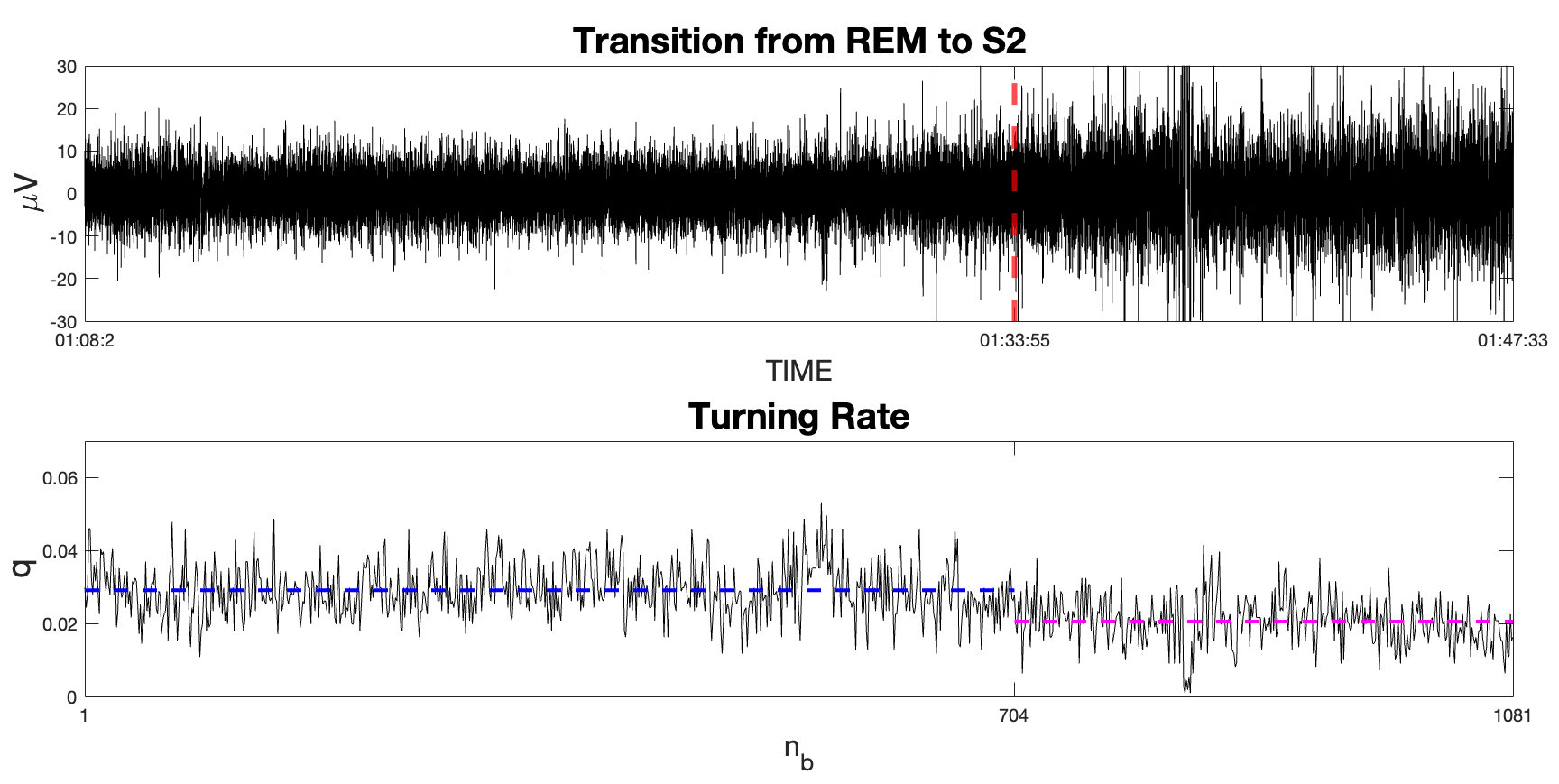}
       \caption{Extract of EEG recordings for the $5$-th patient of the dataset \citet{terzano2001atlas}. The recordings cover approximately 30 minutes of observations extracted from the C4-P4 channel, in the temporal window going from 01:08:2 (REM) to 01:47:33 (S2). The time series contains $1.2\times 10^6$  data points. 
    } \label{EEG_from_REM_to_S2}
\end{figure}
\end{example}
\section{Discussion and outlook}

The main theoretical contribution of this work corresponds to the establishment of central limit theorems for  relative frequencies in linear processes; see Section \ref{subsec:clt}. An intriguing question is whether one can extend these  limit theorems to estimators of the form $\tfrac{1}{n}\sum_{t=1}^n \bm{1}(\Xb_t\in A)$ for more general sets $A.$
In machine learning parlance, an ordinal patterns is a feature. While in the presented framework the ordinal patterns is fixed beforehand, machine learning learns the features from data. In view of the application to EEG data, one might design methods in future work that also selects a suitable linear combination of the most relevant ordinal patterns from either supervised or unsupervised data. In view of the central limit theorems for relative frequencies presented in Section \ref{sec.CLTs_general}, one might also want to directly learn features $\bm{1}(W\Xb_t\leq \mathbf{v})$ for learnable weight matrix $W$ and shift vector $\mathbf{v}.$ The problem is then closely connected to training of neural networks with Heaviside activation function $\sigma(x)=\bm{1}(x\geq0)$.

\section*{Acknowledgements}
Annika Betken gratefully acknowledges financial support from the Dutch Research Council (NWO) through VENI grant 212.164. 

\nocite{*}
\bibliography{sample}

\appendix 
\section*{Appendix}
\textbf{Notation: } The Euclidean norm in \(\mathbb{R}^r\) is denoted by \(\|\cdot\|\). For \( i_1, \ldots, i_h \geq 1 \), let \( \partial^h_{i_1, \ldots, i_h} f \) denote the partial derivative of \( f \) taken \( h \) consecutive times with respect to the variables \( i_1, \ldots, i_h \). We denote with $\nabla f$ the transposed  gradient of $f$, given by $
\nabla f (\mathbf{x}) = \left( \partial_1 f(\mathbf{x}), \ldots, \partial_p f(\mathbf{x}) \right),$
and with $\nabla^2 f$ the corresponding Hessian matrix, expressed as $
\nabla^2 f(\mathbf{x}) = \left( \partial^2_{ps} f(\mathbf{x}) \right)_{p,s}.$  The symbol \(\C\) denotes a generic numerical constant that may change upon each appearance.

We define the truncated linear process $\Xb_{t,j}$ as $\Xb_{t,j}:=\sum_{i=0}^j A_i \mathbf{Z}_{t-i}$.  Specifically, $\Xb_{t,0} = A_0\mathbf{Z}_t$, and $\Xb_t = \Xb_{t,\infty}$. Further, let $\mathbf{R}_{t,j}= \mathbf{X}_{t}-\mathbf{X}_{t,j}=\sum_{i=j+1}^\infty A_i \mathbf{Z}_{t-i}$. For fixed $\mathbf{u}=(u_0, \ldots, u_{r-1})^\top$,
\begin{align}
    p_j(\mathbf{u})=\mathbb{P}(  \mathbf{X}_{t,j} \leq \mathbf{u})\;, \quad  p(\mathbf{u}) = \mathbb{P}( \mathbf{X}_1 \leq \mathbf{u})\;.
\end{align}
$p_j$ only depends on $j$ as $(X_{t,j})_{t\geq 1}$ forms a stationary process. Moreover, we define a function $g:\mathbb{R}^r \to \mathbb{R}$ as Lipschitz if, for all $\bx, \by \in \mathbb{R}^r$, it holds that $|g(\bx) - g(\by)| \leq L \|\bx - \by\|$, with the smallest constant $L$ denoted by $\Lip(g)$. Lastly, the symbol  $\mathcal{C}$ denotes a generic constant and the cumulative distribution function of $\mathbf{Z}_1$ is denoted by $\mathcal{G}(\cdot)\;.$ 

The proofs of Theorem \ref{theorem:convergence_OP_SRD} and Theorem \ref{theorem:convergence_OP_LRD} rely on the following Martingale decomposition (consequence of Lemma \ref{sigma_algebra}). Let $\mathcal{F}_t=\sigma(\mathbf{Z}_t, \mathbf{Z}_{t-1}, \ldots)$ be the sigma field generated by the past values of $\mathbf{Z}_t$. Then, 

\begin{align}
\label{eq:martingale_decomposion}
    \bm{1}\left( \Xb_t \leq \mathbf{u} \right) - p(\mathbf{u}) = \sum_{j=0}^\infty U_{t,j} \;, \quad U_{t,j}:&= \mathbb{E}\big[\bm{1}(\Xb_t\leq \bu) \big| \mathcal{F}_{t-j}\big] - \mathbb{E}\big[\bm{1}(\Xb_t\leq \bu) \big| \mathcal{F}_{t-j-1}\big]\;, \\
    &= p_{j}(\mathbf{u} -\mathbf{R}_{t,j}) - p_{j+1}(\mathbf{u} -\mathbf{R}_{t,j+1}) \nonumber\;.
\end{align}

\section{Proofs  Short-Range Dependence}
\label{appendix:SRD}
\begin{lemma}
\label{lem.Lip}
Let $\Xb_t= \sum_{j= 0 }^\infty A_j \mathbf{Z}_{t-j}$ be a multivariate linear process. If Assumption \ref{assump.1} holds for $J\geq 0$, then $\sup_{j\geq J} \Lip(p_j)\leq \Lip(p_{J})  <\infty.$ 
\end{lemma}
\begin{proof}[Proof of Lemma \ref{lem.Lip}]
In order to prove the Lipschitzness of $p_j$, we shall use the following relation: 
\begin{equation}
\label{eq:recursive_G}
    p_{j+1}(\mathbf{u}) = \mathbb{P} \left(\mathbf{X}_{t,j+1} \leq \mathbf{u}\right)   = \mathbb{P} \left( \mathbf{X}_{t,j} \leq - A_{j+1} \mathbf{Z}_{t-j-1}+\mathbf{u} \right)=\int p_j (\mathbf{u} -  A_{j+1} \mathbf{t} ) \, d\mathcal{G}(\mathbf{t}) \quad \text{for all } j\geq 0\;,
\end{equation}
If there exists $j\geq 0$ for which $\text{Lip}(p_j)< \infty$, then for any $\mathbf{u},\mathbf{v} \in \RR^r$
\begin{align*}
    |p_{j+1}(\mathbf{u}) - p_{j+1}(\mathbf{v})|\leq \int \left|p_j(\mathbf{u}-   A_{j+1} \mathbf{t} ) - p_j (\mathbf{v}-  A_{j+1} \mathbf{t} ) \right|\, d\mathcal{G}(\mathbf{t})\leq \text{Lip}(p_j) \|\mathbf{u}-\mathbf{v}\|\;,
\end{align*}
thus $\text{Lip}(p_{j+1})\leq 
\text{Lip}(p_j)\;.$ 
To conclude  that $\sup_{j\geq J} \Lip(p_j)\leq \Lip(p_{J})  <\infty\;,$

it is enough to prove that $\text{Lip}(p_{J})< \infty.$ By definition the function $p_{J}$ is the cumulative distribution function of the variable $\Xb_{t,J}$. By Assumption \ref{assump.1}, $D\Xb_{t,J}$ is a vector of random variables with bounded Lebesgue density for an invertible matrix $D$. Such probability density function is denoted by $f_{J}\;.$ Then, we can write, for $\mathbf{u}=(u_0,\ldots, u_{r-1})$
\begin{align*}
    p_{J}(\mathbf{u})= \int_{-\infty}^{u_0} \cdots \int_{-\infty}^{u_{r-1}} f_{J}(t_0, \ldots, t_{r-1})\, dt_0,\ldots dt_{r-1}\;.  
\end{align*}
For any $\mathbf{u}=(u_0, \ldots, u_{r-1})^\top$ and $\mathbf{v}=(v_0, \ldots, v_{r-1})^\top$,  we have $\{ \mathbf{X}_{t,J}\leq \mathbf{u}\}/ \{ \mathbf{X}_{t,J} \leq \mathbf{v} \}\subset \bigcup_{i=0}^{r-1} \{ u_i \wedge v_i \leq \mathbf{X}_{t,J,i} \leq u_i \vee v_i \}$ with $\mathbf{X}_{t,J,i}$ the $i-$th component of $\mathbf{X}_{t,J}$. Then, 
\begin{align*}
     \PP(\mathbf{X}_{t,J}\leq \mathbf{u})-  \PP(\mathbf{X}_{t,J}\leq \mathbf{v})=p_{J}(\mathbf{u})-   p_{J}(\mathbf{v})  \leq \sum_{i=0}^{r-1} \int_{ u_i \wedge v_i }^{  u_i \vee v_i } f_{J}^{(i)}(t_i) \, dt_i\;,
\end{align*}
where \( f_{J}^{(i)} \) is the marginal density obtained from the \(i\)-th entry of the vector \(\mathbf{X}_{t,J}\), denoted by $\mathbf{X}_{t,J}^{(i)}\;.$
Let $e_i$ be the $i-$th vector of the canonical basis of $\RR^r$, and note that each entry  $\mathbf{X}_{t,J}^{(i)}$ is obtained as
$\mathbf{X}_{t,J}^{(i)} = e_i ^\top  \mathbf{X}_{t,J}\;.$ Furthermore, $\mathbf{X}_{t,J}^{(i)} $ admits a bounded probability density function. In fact, by Assumption \ref{assump.1}, there exist an invertible matrix $D$ and i.i.d. random variables $Y_0,\ldots, Y_{r-1}$ admitting bounded density, such that 
$$\mathbf{X}_{t,J}^{(i)} = e_i ^\top  \mathbf{X}_{t,J} = e_i^\top D^{-1} \left( \begin{array}{c}
     Y_1\\
     \vdots\\
     Y_r
\end{array}\right)= \sum_{j=0}^{r-1} (e_i^\top \cdot D^{-1}_{(j)}) Y_j \;,$$
and $D^{-1}_{(j)}$ denoting the $j-$th column of $D^{-1}\;.$ The right-hand side of the previous expression is the linear combination of independent and identically distributed random variables with bounded probability density. From the convolution formula, the resulting density is also bounded (by a constant denoted as $\C$). Therefore, for any \( i = 0, \ldots, r-1 \), we have \( \sup_{x\in \RR^r}f_{J}^{(i)}(x) \leq \C \) for a positive constant \(\C\), every integral is upper bounded as follows (shown here for \(i = 0\)):
$$
     \int_{u_0 \wedge v_0}^{u_0 \vee v_0} \, \underbrace{\left(\,  \int_{-\infty}^{\infty}\cdots \int_{-\infty}^{\infty} f_\mathbf{J}(\mathbf{u}) \, dt_1 \ldots dt_{r-2} \right)}_{f_{J}^{(0)} (u_0)}d t_0
    \leq  \C \| \mathbf{u} -\mathbf{v} \|_{\infty}\;.$$
\end{proof}
As we study functionals that are indicators and can be bounded by one, the next result shows that the conclusion of Lemma 3.4 of  \citet{Furmanczyk} still holds although the Lipschitz condition only applies for $p_j$ with $j\geq J.$
\begin{lemma} 
\label{lemma:lemma3.4Furmanczyk}
Let $\Xb_t= \sum_{j= 0 }^\infty A_j \mathbf{Z}_{t-j}$ be a multivariate linear process satisfying $\sum_{j=0}^\infty \|A_j\|< \infty$ and Assumption \ref{assump.1}. Then, for $U_{t,j}$ defined in \eqref{eq:martingale_decomposion}

    \begin{enumerate}
        \item[(i)] $\EE[\, U_{t,0}^2\,] \leq P(\Xb_1\leq \bu) \;,$ 
        \item[(ii)] $\EE[\,  U_{t,j}^2\, ] \leq \C \|A_{j+1}\|^2\;.$
    \end{enumerate}
\end{lemma}    
    \begin{proof}
  (i) follows from expending the square, 
\begin{align*}
     \EE \left[ \,  ( \bm{1}(\Xb_t\leq \bu) - \EE[\bm{1}(\Xb_t\leq \bu)|\F_{t-1} ] )^2 \,\right]
     = P(\Xb_t\leq \bu)-\EE\big[\EE[\bm{1}(\Xb_t\leq \bu)|\F_{t-1} ]^2\big]
     \leq P(\Xb_1\leq \bu)\;.   
    \end{align*}
    To prove (ii), notice that $|U_{t,j}|\leq 1=\|A_j\|^{-1}\|A_j\|,$ and thus, for $j\leq  J-1,$ $\EE[ \, |U_{t,j}|^2 \, ]\leq \C_1 \|A_{j+1}\|^2$ with $\C_1= \max\limits_{j=0,\ldots,J-1} \|A_j\|^{-2}.$ Since $J$ is fixed, $\C_1$ is a finite constant as Assumption \ref{assump.1} holds and $\|A_j \|\neq 0$ for $j=0,\ldots, J$  . 
        To treat the case $j\geq J,$ notice that by \eqref{eq:martingale_decomposion} $U_{t,j}=p_{j}(\mathbf{u} -\mathbf{R}_{t,j}) - p_{j+1}(\mathbf{u} -\mathbf{R}_{t,j+1})$. Given the independence between \(A_{j+1} \mathbf{Z}_{t-(j+1)}\) and \(\mathbf{R}_{t,j+1}\), we have
        $p_{j+1}(\mathbf{u} -\mathbf{R}_{t,j+1})=\int p_j (\mathbf{u}-\mathbf{R}_{t,j+1} - \mathbf{t}) \, dF_{1,j}(\mathbf{t})$,  where \(F_{j}\) is the distribution function of \(A_{j+1} \mathbf{Z}\). 
     In Lemma \ref{lem.Lip}, we have shown that $p_j$ is Lipschitz for $j\geq J $, with constant $\text{Lip}(p_j)\leq \text{Lip}(p_{J})$. Thus, for $j\geq J,$
        \begin{align*}
            |U_{t,j}|\leq& \int |p_{j}(\mathbf{u} -\mathbf{R}_{t,j}) - p_{j+1}(\mathbf{u} -\mathbf{R}_{t,j+1} -\mathbf{t})|\, dF_{1,j}(\mathbf{t})\\
            \leq& \text{Lip}(p_j) \int  \|\mathbf{R}_{t,j}-\mathbf{R}_{t,j+1} - \mathbf{t}\| \, dF_{1,j+1}(\mathbf{t})\\
            \leq& \text{Lip}(p_{J}) \left( \|\mathbf{R}_{t,j}-\mathbf{R}_{t,j-1}\| + \int \|\mathbf{t}\| \, dF_{1,j}(\mathbf{t}) \right)\\
            \leq& \text{Lip}(p_{J}) \left( \|A_{j+1} \mathbf{Z}_{t-j-1}\| + \EE \left[\,\| A_{j+1} \mathbf{Z}_{t-j-1}\| \, \right] \right)\\
            \leq& \text{Lip}(p_{J}) \|A_{j+1}\| \left(\|\mathbf{Z}_{t-j-1}\|+ \sqrt{\EE\left[ \|\mathbf{Z}_{t-j-1} \|^2\right]}\right)\;.
        \end{align*}
        Taking expectation, it follows that $\EE[\,  U_{t,{j}}^2\, ] \leq \C_2 \|A_{j+1}\|^2$ for a constant $\C_2$ that is independent of $j\geq J.$ Thus, (ii) holds with $\C=\max(\C_1,\C_2).$
    \end{proof}
    
 \begin{theorem}[Theorem 4.2 in \citet{billingsley1968convergence}]
\label{theorem:theorem4.2Billingsley}
    Let $B:=\left(B(\tau)\right)_{\tau \in [0,1]}$ a standard Brownian motion on $[0,1]$. Let $W=\left( W(\tau)\right)_{\tau \in [0,1]}$ be a $\D[0,1]$-valued stochastic process. If there exists a process $V_{un}$ such that $u\in \mathbb{N}$:
    \begin{itemize}
        \item [(i)] $V_{u,n} \overset{\D[0,1]}{\Longrightarrow}  \sigma_u B$ for a certain $\sigma_u>0$.
        \item [(ii)]There exists finite the limit $\sigma^2 =\lim_{u\to \infty} \sigma_u^2 >0\;,$
        \item [(iii)]
        $ \liminf\limits_{u\to \infty}\limsup\limits_{n\to \infty} \PP \left( |V_{u,n}-W_n|\geq \varepsilon\right)=0\quad \text{for any }\varepsilon>0\;.$
    \end{itemize}
    Then, 
    $$ W_n \overset{\D[0,1]}{\Longrightarrow}\sigma B\quad\text{as }n\to \infty\;.$$
\end{theorem}
\begin{proof}[Proof of Theorem \ref{theorem:SRD_multivariate}]
We apply Theorem \ref{theorem:theorem4.2Billingsley} (Theorem 4.2 of Billingsley).  The vectorized linear form  $\mathbf{X}_t=(X_t, \ldots, X_{t+r-1})^\top$, represented as a linear process of the form \(\mathbf{X}_t = \sum_{j=0}^\infty A_j \mathbf{Z}_{t-j}\), permits the martingale difference decomposition \eqref{eq:martingale_decomposion}. Then, we set
\begin{align*}
    W_n:&=\frac{1}{\sqrt{n}}\sum_{t=1}^{\lfloor n \tau \rfloor} \sum_{j=0}^\infty U_{t,j}\;, \quad V_{u,n}:=\frac{1}{\sqrt{n}} \sum_{t=1}^{\lfloor n \tau \rfloor} \sum_{j=0}^{u-1} U_{t,j}\;, \\
    H_{u,n}:&=\frac{1}{\sqrt{n}}\sum_{j=0}^{u-1} \left( \sum_{t=1}^j U_{t,j} - \sum_{t=\lfloor n \tau \rfloor+1}^{\lfloor n \tau \rfloor +j} U_{t,j}\right)\;,
\end{align*}
We will show that $V_{u,n}$ satisfies Theorem \ref{theorem:theorem4.2Billingsley}.

(i) By construction, 
    \begin{equation}
       V_{u,n}= \frac{1}{\sqrt{n}} \sum_{t=1}^{\lfloor  n \tau \rfloor} \sum_{j=0}^{u-1} U_{t+j,j}+H_{u,n}\;. 
       \label{eq:V_{u,n}}
    \end{equation}
    The advantage of formulation \eqref{eq:V_{u,n}} is that, for every fixed $u$,  $\left( \sum_{j=0}^{u-1} U_{t+j,j}, \mathcal{F}_t \right)_t$ forms a stationary, ergodic and centered martingale differences sequence. In fact, by setting $M_t(u):=\sum_{j=0}^{u-1} U_{t+j,j}$, we have $\EE[M_t(u)]<\infty$ and, for all $s\leq t-1$  
    
    \begin{align*}
        \EE \left[ M_t(u)|\, \mathcal{F}_s \right] =&  \EE \left[ \sum_{j=0}^{u-1} U_{t+j,j} \Bigg|\, \mathcal{F}_s \right] =  \sum_{j=0}^{u-1}\EE \left[ U_{t+j,j} |\, \mathcal{F}_{s} \right]\\
        =& \sum_{j=0}^{u-1}\EE \left[ \, \left( \EE \left[ \bm{1}(\mathbf{X}_{t+j}\leq \mathbf{u}) \Big| \mathcal{F}_t\right] -\EE \left[ \bm{1}(\mathbf{X}_{t+j}\leq \mathbf{u}) \Big| \mathcal{F}_{t-1}\right] \right) \Bigg|\, \mathcal{F}_{s} \right]\\
        =& \sum_{j=0}^{u-1} \EE \left[ \bm{1}(\mathbf{X}_{t+j}\leq \mathbf{u}) \Big| \mathcal{F}_s\right] -\EE \left[ \bm{1}(\mathbf{X}_{t+j}\leq \mathbf{u}) \Big| \mathcal{F}_{s}\right] \quad \text{since }s\leq t-1\\
        =&0\;.
    \end{align*} 

    Since $|U_{t,j}|\leq 1$,  we have $\max_{t,j}\EE[ U_{t,j}^2]\leq 1$ and 
    $$ H_{u,n}=\mathcal{O}_{\mathbb{P}}(n^{-1/2})\quad \text{for every }u \in \mathbb{N}\;.$$
    In particular this implies that for every fixed $u$, $H_{u,n}\xrightarrow{\PP}0$. Lastly, for $\sigma^2_u=\Var \left( M_1(u)\right)$, we have $\frac{1}{\sqrt{n}}\sum_{t=1}^{\lfloor n\tau \rfloor} M_t(u) \overset{\D[0,1]}{\Longrightarrow}\sigma_u B(\tau) $, as a consequence of the central limit theorem for martingale differences (Theorem 23.1 in  \citet{billingsley1968convergence}). Hence, $V_{u,n} = \frac{1}{\sqrt{n}}\sum_{t=1}^{\lfloor n\tau \rfloor} M_t(u) + H_{u,n} \overset{\D[0,1]}{\Longrightarrow}\sigma_u B(\tau) \;.$  

The proof of (ii): We have to show that there exists a finite $\sigma^2:=\lim\limits_{u\to \infty}\sigma^2_u\;.$ By definition 
\begin{align*}
    \sigma^2_u=\Var \left( M_1(u)\right) =& \EE\left[ \left( \sum_{j=0}^{u-1} U_{1+j,j} \right)^2 \right] \leq   \sum_{i,j=0}^{u-1}\EE\left[  |U_{1+j,j} U_{1+i,i }|\right]\\
    \leq & \sum_{i,j=0}^{u-1} \sqrt{ \EE[ U_{1+i,i}^2]} \sqrt{ \EE[ U_{1+j,j}^2]} \quad \text{Cauchy-Schwarz} \\
    =& \left(\sum_{j=0}^{u-1} \sqrt{ \EE[ U_{1+j,j}^2]}  \right)^2 \leq \C \left(\sum_{j=0}^{u-1} \|A_{j+1}\|  \right)^2 \quad \text{Lemma }\ref{lemma:lemma3.4Furmanczyk}\;.
\end{align*}
Finally $\sum_{j=0}^{u-1} \|A_{j+1}\| \leq \sum_{j=0}^{\infty} \|A_{j+1}\|<\infty.$ This already concludes point (ii). However, we can also compute the limit $\sigma^2$ exactly. In fact, from Theorem \ref{lemma6.4}, for all indexes $(t,j),(t',j')$, the expectation  $\EE[ U_{t,j}U_{t',j'}]\neq0$ if and only if $t-j=t'-j'$, thus
$ \EE[ U_{t,j}U_{t',j'}]= \EE[ U_{t,j}U_{t+j'-j,j'}]$. Furthermore, $\left(  U_{t,j}U_{t+j'-j,j'} \right)_{t\geq 1}$ is a strictly stationary sequence, hence for all $t$ we have 
$\EE[ U_{t,j}U_{t+j'-j,j'}]= \EE[ U_{1,j}U_{1+j'-j,j'}]\;.$ Therefore, for the pairs $(1+j,j),(1+i, i)$, 
$$ \EE[ U_{1+j,j}U_{1+i,i}]=\EE[U_{1,j}U_{1+i-j,i}]\;,$$
and
\begin{align*}
    \sigma^2_u=   \sum_{i,j=0}^{u-1}\EE\left[  U_{1+j,j} U_{1+i,i }\right] =  \sum_{i,j=0}^{u-1}\EE\left[ U_{1,j}U_{1+i-j,i}\right] \xrightarrow{u \to \infty}  \sum_{i,j=0}^{\infty}\EE\left[ U_{1,j}U_{1+i-j,i}\right]=:\sigma^2\;.\\
\end{align*}
Moreover, by decomposition \eqref{eq:martingale_decomposion}, $\left(\bm{1}(\mathbf{X}_t\leq \mathbf{u}) - p(\mathbf{u}) \right)\left(\bm{1}(\mathbf{X}_{t'}\leq \mathbf{u}) - p(\mathbf{u}) \right) =   \sum_{i,j=0}^\infty U_{t,j}U_{t',i}\;,$ 

$$   \sum_{i,j=0}^{\infty}\EE\left[ U_{1,j}U_{1+i-j,i}\right] = \sum_{s\in \mathbb{Z}} \sum_{i,j=0}^{\infty}\EE\left[ U_{1,j}U_{1+s,i}\right] = \sum_{s \in \mathbb{Z}} \Cov \left( \bm{1}(\mathbf{X}_s\leq \mathbf{u}), \bm{1}(\mathbf{X}_{1+s}\leq \mathbf{u}) \right) \;.$$

For condition (iii), we proceed as follows. Let $\varepsilon>0$, and using Chebyshev's inequality gives
\begin{align*}
    \PP ( |V_{u,n}-W_n|\geq \varepsilon) = \PP \left( \frac{1}{\sqrt{n}}\Bigg|  \sum_{t=1}^{\lfloor n \tau \rfloor } \sum_{j=u}^\infty U_{t,j}\Bigg| \geq \varepsilon\right) \leq \frac{1}{\varepsilon^2 n} \EE \left[ \left(  \sum_{t=1}^{\lfloor n \tau \rfloor } \sum_{j=u}^\infty U_{t,j} \right)^2 \right]\;.
\end{align*} 
The term of the right-hand side can be expanded out further

\begin{align*}
    \EE \left[ \left(  \sum_{t=1}^{\lfloor n \tau \rfloor } \sum_{j=u}^\infty U_{t,j} \right)^2 \right] =&  \sum_{t,s=1}^{\lfloor n \tau \rfloor } \sum_{i,j=u}^\infty \EE [ U_{t,j}U_{s,i}] =  \sum_{t=1}^{\lfloor n \tau \rfloor } \sum_{i,j=u}^\infty \EE [ U_{t,j}U_{t-j+i,i}] \quad \text{Lemma }\ref{lemma6.4} \\
    \leq&  \sum_{t=1}^{\lfloor n \tau \rfloor } \sum_{i,j=u}^\infty \Big| \EE [ U_{t,j}U_{t-j+i,i}] \Big| \leq  n \max_{1 \leq t\leq \lfloor n \tau \rfloor } \sum_{i,j=u}^\infty  \Big| \EE [ U_{t,j}U_{t-j+i,i}] \Big| \\
    \leq&   n \max_{1 \leq t\leq \lfloor n \tau \rfloor } \sum_{i,j=u}^\infty  \sqrt{ \EE [ |U_{t,j}|^2]}  \sqrt{\EE[|U_{t-j+i,i}|^2]}  \quad \text{Cauchy-Schwarz}\\
    \leq&   n \max_{1 \leq t\leq \lfloor n \tau \rfloor } \C \sum_{i,j=u}^\infty   \|A_{j+1}\|  \|A_{i+1}\|  \quad \text{Lemma }\ref{lemma:lemma3.4Furmanczyk}\\
    =&n \C^2 \left( \sum_{j=u+1}^\infty \|A_j\|\right)^2 \;.
\end{align*}
Combining the previous inequalities, we get 
\begin{align*}
    \PP ( |V_{u,n}-W_n|\geq \varepsilon)  \leq \frac{1}{\varepsilon^2 n} \EE \left[ \left(  \sum_{t=1}^{\lfloor n \tau \rfloor } \sum_{j=u}^\infty U_{t,j} \right)^2 \right] \leq \frac{\C^2}{\varepsilon^2} \left( \sum_{j=u+1}^\infty \|A_j\|\right)^2 \;.
\end{align*}
The term $\sum_{j=u+1}^\infty \|A_j\|$ is the tail of the series $\sum_{j=0}^\infty \|A_j\|$, which converges by assumption to zero ( short-range-dependence). Therefore, 
\begin{align*}
   \liminf_{u\to \infty}\limsup_{n\to \infty} \PP ( |V_{u,n}-W_n|\geq \varepsilon) \leq   \liminf_{u\to \infty} \C^2 \left( \sum_{j=u+1}^\infty \|A_j\|\right)^2 =0\;.
\end{align*}
After having checked all conditions, we can now apply Theorem \ref{theorem:theorem4.2Billingsley} and have thus derived the claim of Theorem \ref{theorem:SRD_multivariate}. 
\end{proof}

\section{Proofs Long-Range dependence }
\label{appendix:LRD}
\begin{theorem}[ \citet{chung_2002}, Theorem 1] 
\label{theorem:chung_theorem1} For $\mathbf{d}=(d_1,\ldots, d_r)^\top \in (0,1/2)^r$, consider the multivariate linear process $\Xb_t=\sum_{j=0}^\infty A_j \mathbf{Z}_{t-j}$ with coefficients $A_j \sim j^{\mathbf{d}-1} A_\infty$ and $A_\infty \in \text{GL}(\RR, r)$ and $(\mathbf{Z}_j)_{j \in \mathbb{Z}}$ an i.i.d. sequence and moment condition $\mathbb{E}[\|\mathbf{Z}_1\|^4]<\infty$. 
Then, 
$$ \text{diag}\left(n^{-\left(\textbf{d}+\frac{1}{2}\right)} \right)   \sum_{j=1}^{\lfloor n \tau \rfloor} \Xb_j \overset{w}{\Longrightarrow}\textbf{B}_\textbf{d}(\tau) \quad \text{for }\tau\in [0,1]\;, \quad \text{as }n \to \infty\;,$$

where $\overset{w}{\Longrightarrow}$ denotes the weak convergence, $\textbf{B}_\textbf{d}(s)$ is a $r-$dimensional fractional Brownian process $(B_{d_1}(s), \ldots, B_{d_r}(s))^\top$ with covariance matrix $[\eta_{u,v}]_{u,v=1, \ldots, r}\circ A_\infty^\top \Sigma A_\infty$. Here, the symbol $\circ$ denotes the Hadamard product, i.e the component wise product, between the matrices $A_\infty^\top \Sigma A_\infty$ and $[\eta_{u,v}]_{u,v=1, \ldots, r}$, whose components are  

$$\eta_{uv}= \frac{\Gamma(d_u) \Gamma(d_v)}{\Gamma(d_u+1) \Gamma(d_v+1)} \left( \frac{1}{1+ d_u+d_v} +\int_0^{\infty} \left( (1+t)^{d_u} - (t)^{d_u}\right) \left( (1+t)^{d_v} - (t)^{d_v}\right)\right)\, dt\;.$$
\end{theorem}
\begin{remark}
For the following identity holds  for $d\in (0,1/2)$ (see \citet{beran2013long}, p. 35):
   \begin{equation}
       \frac{1}{2d+1} +\int_0^\infty \left(  (1+t)^d - t^d)\right)^2\, dt = \frac{\Gamma (1+d)^2}{\Gamma(2d+2) \sin \left(\frac{\pi}{2} +\pi d \right)}\;.
   \end{equation}
Therefore, for $d_1=\ldots=d_r=d$

$$\eta_{ij}=\frac{\Gamma(d)^2 \Gamma(1+d)^2}{\Gamma(1+d)^2 \Gamma(2d+2) \sin \left(\frac{\pi}{2} +\pi d \right) }=\frac{\Gamma(d)^2 }{\Gamma(2d+2) \cos \left(\pi d \right) }\;,$$
and, for $s=1$
\begin{align*}
    n^{\frac{1}{2}-d} \bar{\mathbf{X}}_n \xrightarrow{\D} \mathcal{N}\left( \mathbf{0}, V\right)\;, \quad V:= \frac{\Gamma(d)^2 }{\Gamma(2d+2) \cos \left(\pi d \right) } A_\infty^\top \Sigma A_\infty\;.
\end{align*}
\end{remark}
\begin{proof}[Proof of Theorem \ref{thm.2}]
\label{proof_thm2}
The proof follows from an application of the reduction principle in Appendix \ref{appendix:reduction_principle}. In fact, assumption \eqref{eq:LIP-condition1}  of Theorem \ref{theorem:reduction_principle} is satisfied as it is also one of the assumptions of Theorem \ref{thm.2}, and   
$\Xb_t= \sum_{j=0}^\infty A_j \mathbf{Z}_{t-j}\;,$ and  $(\mathbf{Z}_j)_{j \in \mathbb{Z}}$ forms a multivariate stationary white noise in $\RR^r$ with variance $\Sigma$ and, $\EE[\|\mathbf{Z}_1\|^4]< \infty\;.$ Therefore, for any fixed $\mathbf{x}\in \RR^r$, 
\begin{equation}   
 n^{\frac{1}{2} -d } \Bigg| \frac{1}{n}\sum_{j=1}^n \bm{1}\left( \Xb_{j} \leq \mathbf{x} \right) - p(\mathbf{x}) + \nabla p(\mathbf{x}) \Bar{\Xb}_n \Bigg| \xrightarrow{\PP} 0 \quad \;.
\label{uniform_principle1}
\end{equation}
Furthermore,  Theorem 1 in \citet{chung_2002} (reproduced as Theorem \ref{theorem:chung_theorem1} below) guarantees that 
$$n^{1/2-d} (\nabla p(\mathbf{x}))^{\top} \Bar{\Xb}_n \xrightarrow{\D} \N \left( 0, \sigma^2 \right) \quad   \quad \text{where } \sigma^2:= \frac{\Gamma(d)^2 }{\Gamma(2d+2)\cos{(\pi d)}} \left( \nabla p(\mathbf{x}) \right)^{\top} A_\infty \Sigma A_\infty^{\top}  \nabla p(\mathbf{x}) \;.$$
Thus, by Theorem 4.1 in \citet{billingsley1968convergence}, we obtain 
$$ n^{\frac{1}{2} -d } \Bigg( \frac{1}{n}\sum_{j=1}^n \bm{1}\left( \Xb_
{j}\leq \mathbf{x} \right) - p(\mathbf{x})\Bigg) \xrightarrow{\D} \N (0,\sigma^2)\;.$$
\end{proof}
In the following proposition we give a criteria used to verify  Assumption \eqref{eq:LIP-condition1} when we consider the process $\mathbf{X}_t:=(X_t, \ldots, X_{t+r-1})^\top$  for $X_t=\sum_{j=0}^\infty a_j Z_{t-j}$. More specifically, 
\begin{proposition}
\label{prop:Lipschitzness_gradient}
  Consider the process $X_t=\sum_{j=0}^\infty a_j Z_{t-j}$ with $a_0\neq 0$,  $(Z_j)_{j \in \NN}$  forming an independent and identically distributed sequence with $Z_1 $ admitting a continuous and bounded probability density function $f$. For $r\geq 1$, we define $\Xb_t =(X_t, \ldots, X_{t+r-1})^\top$. If $f$ admits a bounded derivative, then \eqref{eq:LIP-condition1} is satisfied for $\Xb_t$ with $s_0:=r$.
\end{proposition}
\begin{proof}[Proof of Proposition \ref{prop:Lipschitzness_gradient}]
It is enough to show that for any $\mathbf{x}=(x_0, \ldots, x_{r-1})^\top$, there exists a constant $\C>0$ independent of $\mathbf{x}$ such that, for all $j\geq r$
\begin{equation}
   \left| p_j(x_0, \ldots, x_{r-1})\right|\leq \C\;,\quad     \left| \nabla p_j(x_0, \ldots, x_{r-1})\right|\leq \C\, \, \text{and}\, \, \left| \nabla^2 p_j(x_0, \ldots, x_{r-1})\right|\leq \C\;. 
      \label{eq:bounded_derivatives}
  \end{equation}
The idea is to prove \eqref{eq:bounded_derivatives} by induction, using the following recursive relation, which holds for all $j\geq 1$ 
\begin{align*}
    p_{j}( \mathbf{x})=\int_{\RR}   p_{j-1} (\mathbf{x}-  A_j \mathbf{t} )  \, d\mathcal{G}(\mathbf{t})\;.  
\end{align*}
In fact, if \eqref{eq:bounded_derivatives} holds for $j=r$  for some constant $\C$ (under this assumption we can use the dominated convergence theorem since a constant is integrable with respect to $d\mathcal{G}(\cdot)$), then for all $j\geq r+1$
\begin{align*}
     \nabla p_{j}( \mathbf{x}) =&\int_{\RR} \nabla  p_{j-1}(\mathbf{x}-  A_j \mathbf{t} )\, d\mathcal{G}(\mathbf{t}) 
    \leq \C  \int_{\RR} \, d\mathcal{G}(\mathbf{t}) =\C\;,
\end{align*}
and similarly $\nabla^2 p_j (\mathbf{x}) \leq \mathcal{C}\;.$ So, we only have to prove the inductive base for $j=r$. We fix additional notation:
 \begin{enumerate}
  \item For any generic random variable $U$ in $\RR^r$, $f_U (u_0, \ldots, u_{r-1}) = f_U(\mathbf{u})$ denotes its probability density function in $\RR^r$. For instance, $f_{\Xb_{t,r}}(x_0,\ldots, x_{r-1})$ denotes the probability density function of $\Xb_{t,r}$ in $\RR^{r}$.
     \item The coordinates of a generic $m$-dimensional random vector $\mathbf{U}$  are denoted by $\mathbf{U}=(U^{(0)}, \ldots, U^{(m-1)})^\top$.  
     \item For a matrix  $A \in \RR^{a\times b}$, we denote by $A^{-(j)}_{-(i)}$ the sub-matrix obtaining by removing the $i$-th row and the $j-$th column of $A$. The symbol $A^{-(j_1, \ldots, j_{s_1})}_{-(i_1, \ldots, i_{s_2})}$ indicates the sub matrix (or minor) obtained by removing all columns and rows indexed by $(j_1, \ldots, j_{s_1})$ and $(i_1, \ldots, i_{s_2})$. For vectors we will not make a distinction between a row and column vector, and simply denote by $\mathbf{U}^{-(i)}$ the vector obtained by removing the $i-$th coordinate from $\mathbf{U}\;.$ 
     \item We define $ \Xb_{t,r} =  B \mathbf{Z}_{t,r} $ where $\mathbf{Z}_{t,r}:= ( Z_t, Z_{t+1}, \ldots, Z_{t+r-1})^{\top}$ and $B$ is invertible.  The matrix $B$ was defined in \eqref{expression_density_of_X_MAIN}. Since $B$ is invertible, by the transformation formula, $f_{\mathbf{X}_{t,r}}(x_0, \ldots, x_{r-1})=   f_{\mathbf{Z}_{t,r}} (B^{-1} \mathbf{x}) |\det(B^{-1})|\;.$
 \end{enumerate}
Observe that by exploiting the structure of $\mathbf{Z}_{t,r}$ we obtain the following: for all $\mathbf{v}=(v_0, \ldots, v_{r-1})^\top $ and for all $\ell=0,\ldots, r-1\;,$
\begin{align*}
    f_\mathbf{Z_{t,r}}(v_0, \ldots, v_{r-1})&= f_{Z_{t}} (v_0) \cdots  f_{Z_{t+r-1}}(v_{r-1})\\
    &\leq \C f_{Z_{t}}(v_0)\cdots f_{Z_{t+(\ell-1)}}(v_{\ell-1})f_{Z_{t+(\ell+1)}}(v_{\ell+1}) \cdots f_{Z_{t+r-1}}( v_{r-1})= \C f_{\mathbf{Z}_{t,r}^{-(\ell)}} \left( \mathbf{v}^{-(\ell)}\right)\;.
\end{align*}
Similarly, for the derivative,
\begin{align*}
    \partial_\ell  f_\mathbf{Z_{t,r}}(v_0, \ldots, v_{r-1})&= f_{Z_{t}}(v_0)\cdots f_{Z_{t+(\ell-1)}}(v_{\ell-1}) f'_{Z_{t+\ell}}(v_\ell) f_{Z_{t+(\ell+1)}}(v_{\ell+1})\ldots f_{Z_{t+(\ell+r-1)}}(v_{\ell+r-1})
    \leq  \C f_{\mathbf{Z}_{t,r}^{-(\ell)}}(\mathbf{v} ^{-(\ell)})\;. \\
\end{align*}
Since the previous two inequalities hold for any choice of $v_0, \ldots, v_{r-1}$, if $\K_{(0)},\ldots, \K_{(r-1)}$ are the rows of the matrix $B^{-1}$, for any $\ell =0, \ldots, r-1$ we have
\begin{align*}
   f_\mathbf{Z_{t,r}} ( \K \mathbf{u})=& f_\mathbf{Z_{t,r}} ( \K_{(0)} \cdot \mathbf{u}, \ldots, \K_{(r-1)} \cdot \mathbf{u}) \leq   \C f_{\mathbf{Z}_{t,r}^{-(\ell)}} ( \K_{(0)} \cdot \mathbf{u}, \ldots, \K_{(\ell-1)} \cdot \mathbf{u}\,,\, \K_{(\ell+1)} \cdot \mathbf{u},\ldots, \K_{(r-1)} \mathbf{u}) = \C f_{\mathbf{Z}_{t,r}^{-(\ell)}} ( \K_{-(\ell)} \mathbf{u})\;.
\end{align*}
Furthermore, by definition of the matrix-vector product, for all $\ell=0, \ldots, r-1$ we have
$$ \K_{-(\ell)} \mathbf{u} = \K_{-(\ell)} ^{-(\ell)}\mathbf{u}^{-(\ell)} + u^{(\ell)} \K^{(\ell)}_{-(\ell)}\;,$$
so that,
\begin{equation}
     f_\mathbf{Z_{t,r}} ( \K \mathbf{u}) \leq \C f_{\mathbf{Z}_{t,r}^{-(\ell)}} \left( \K_{-(\ell)} ^{-(\ell)}\mathbf{u}^{-(\ell)} + u^{(\ell)} \K^{(\ell)}_{-(\ell)} \right) \quad \text{for } \ell=0, \ldots, r-1\;.
\end{equation}
Lastly, note that $\K_{-(\ell)} ^{-(\ell)}$ is invertible for every $\ell=0,\ldots, r-1$. This follows since $\K^{(\ell)}_{\ell}$ s a submatrix of $B^{-1}$ and since $B$ is a lower triagular Toeplitz matrix. It can be shown that  $B^{-1}$ is of the form
$$B^{-1} =\left( \begin{array}{cccc}
     a_0^{-1}& & &   \\
    \ast  & a_{0}^{-1} & & \\
    \vdots & \ddots&\ddots & \\
    \ast&\cdots &\ast & a_0^{-1}
\end{array}\right)\;.$$
Every submatrix obtained by removing rows and columns corresponding to the same indexes, meaning $\K^{-(i_1,\ldots, i_h)}_{-(i_1,\ldots, i_h)}$, is also a lower triangular Toeplitz matrix, thus it is invertible. 

We are ready to bound the derivatives of $p_r$. 
\begin{enumerate}
    \item \textbf{Partial derivatives} for all $\ell=0, \ldots, r-1$
     \begin{align}
       \partial_\ell p_r(x_0, \ldots, x_{r-1})=& \int_{-\infty}^{x_0}\cdots \int_{-\infty}^{x_{\ell-1}} \int_{-\infty}^{x_{\ell+1}} \cdots  \int_{-\infty}^{x_{r-1}}\,   f_{\mathbf{Z}_{t,r}}(\K^{-(\ell)} \mathbf{u}^{-(\ell)} + x_\ell \K^{(\ell)} )\, \, du_0 \ldots du_{\ell-1}du_{\ell+1}\ldots du_{r-1}\;,  \\
       \leq & \C \int_{\RR^{r-1}} f_{\mathbf{Z}_{t,r}^{-(\ell)}}(\K^{-(\ell)}_{-(\ell)} \mathbf{u}^{-(\ell)} + x_\ell \K^{(\ell)}_{-(\ell)} ) d\mathbf{ u}^{-(\ell)} \leq \C\;.
    \end{align}
    The last inequality follows from the fact that $f_{\mathbf{Z}_{t,r}^{-(\ell)}}$  is integrable and $\K^{-(\ell)}_{-(\ell)} $ is invertible, thus a change of variables yields the result. 
    \item \textbf{Partial second derivatives} for any $\ell=0, \ldots, r-1$ and $s=0, \ldots, r-1$. Firstly, we study the case $s=\ell$, i.e
    \begin{align}
      & \partial^2_{\ell,\ell}p_r(x_0, \ldots,x_{r-1})\\=& \int_{-\infty}^{x_0}\cdots \int_{-\infty}^{x_{\ell-1}} \int_{-\infty}^{x_{\ell+1}} \cdots  \int_{-\infty}^{x_{\ell-1}}\, \int_{\RR} \left(  \sum_{j=0}^{r-1} \K_{j,\ell} \partial_j  f_{\mathbf{Z}_{t,r}}(\K^{-(\ell)} \mathbf{u}^{-(\ell)} + x_\ell \K^{(\ell)}_{-(\ell)} ) \right)\,  \, du_0 \ldots du_{\ell-1}du_{\ell+1}\ldots du_{r-1}\;  
    \end{align}
    thus, 
     \begin{align}
      \Bigg| \partial^2_{\ell,\ell}p_r (x_0, \ldots,x_{r-1})\Bigg| \leq &\C \sum_{j=0}^{r-1} \int_{\RR^{r-1}}  f_{\mathbf{Z}_{t,r}^{-(j)}}(\K^{-(\ell)} _{-(j)}\mathbf{u}^{-(\ell)} + x_\ell \K^{(\ell)}_{-(j)} )  d\mathbf{ u}^{-(\ell)} \leq \C \;.\;  
    \end{align}
    If $s\neq r$
    \begin{align}
       &\partial^2_{s,\ell}p_r(x_0, \ldots,x_{r-1})\\=& \int_{-\infty}^{x_0}\cdots \int_{-\infty}^{x_{s-1}} \int_{-\infty}^{x_{s+1}} \cdots \int_{-\infty}^{x_{\ell-1}} \int_{-\infty}^{x_{\ell+1}} \cdots  \int_{-\infty}^{x_{\ell-1}}\, f_{\mathbf{Z}_{t,r}}(\K^{-(\ell,s)} \mathbf{u}^{-(\ell)} + x_\ell \K^{(\ell)} + x_s \K^{(s)})\,  \,d\mathbf{u}^{-(\ell,s)}\\
\leq& \int_{\RR^{r-2}} f_{\mathbf{Z}_{t,r}^{-(\ell,s)}} \left( \K^{-(\ell,s)}_{-(\ell,s)} \mathbf{u}^{-(\ell,s)} + x_\ell \K^{(r)}_{-(\ell,s)} + x_s \K^{-(s)}_{-(\ell,s)} \right)  d\mathbf{u}^{-(\ell,s)} \leq \C\;.  
    \end{align}
\end{enumerate}
\end{proof}

\section{Reduction Principle}
\label{appendix:reduction_principle}
In this section we derive the reduction principle for multivariate linear processes that are long-range dependent. All the auxiliary results are listed in the subsection following the proof of Theorem \ref{theorem:reduction_principle}.
\begin{theorem}[Reduction Principle]
\label{theorem:reduction_principle}
Consider a multivariate linear process
\begin{align*}
    \mathbf{X}_t = \sum_{j=0}^\infty A_j \mathbf{Z}_{t-j}\;, \quad A_j \overset{j \to \infty}{\sim} j^{d-1} A_\infty \in   \text{GL}(\mathbb{R}, r)\;, \quad d \in \left(0, \frac{1}{2} \right)\;,
\end{align*}
with i.i.d.\ innovations \( (\mathbf{Z}_j)_{j \in \mathbb{Z}}\) with variance $\Sigma$ and moment condition \( \mathbb{E}[\|\mathbf{Z}_1\|^4] < \infty \). 
If there exists a positive integer \( s_0 \)  and a constant \( C > 0 \) such that
\begin{equation}
    \sup_{\bx\in \mathbb{R}^r} \, \max_{s\geq s_0} \left(  |p_s(\mathbf{x}) |
    + \sum_{i=1}^r
    | \partial_i p_s(\mathbf{x}) |
    + \sum_{i, j=1}^r
    | \partial^2_{i,j} p_s(\mathbf{x}) | \right) <\C,
    \label{eq:LIP-condition1}
\end{equation}
then, for any fixed vector  $\mathbf{x} \in \mathbb{R}^r$, 
\begin{equation}
    n^{\frac{1}{2} - d} \left| \frac{1}{n} \sum_{t=1}^n \bm{1}(  \mathbf{X}_t \leq \mathbf{x} )- p(\mathbf{x}) + \nabla p(\mathbf{x})  \bar{\mathbf{X}}_n \right| \xrightarrow{\mathbb{P}} 0 \;.
\label{uniform_principle}
\end{equation}
\end{theorem}
\begin{proof}[Proof of Theorem \ref{theorem:reduction_principle}] 
For a fixed $\mathbf{x} \in \RR^{r},$ set
$$S_n(\mathbf{x}) := \sum_{t=1}^n\bm{1} \left(\{\Xb_t \leq \mathbf{x} \}\right) - np(\mathbf{x}) + n\nabla p(\mathbf{x})  \Bar{\Xb}_n   \;,$$
so that $(\ref{uniform_principle})$ is equivalently expressed as $n^{- (\frac{1}{2}+d) } |S_n(\mathbf{x}) |\xrightarrow{\PP}0.$ Hence, 
\[S_n (\mathbf{x})= \sum_{t=1}^n \sum_{s=1}^{\infty} \Big( p_{s-1}(\mathbf{x} -\mathbf{R}_{t,s-1}) - p_{s}(\mathbf{x} -\mathbf{R}_{t,s}) \Big) + \nabla p(\mathbf{x})n\Bar{\Xb}_n \;.\]
The leading term of the main sum is the pointwise difference of \textit{nearly} the same functions $p_s, p_{s-1}$ evaluated at two close points. 
We split $S_n(\mathbf{x})$ in three sums,
\begin{align}
\begin{split}
S_n(\mathbf{x})=&\sum_{t=1}^n\sum_{s=1}^{\infty} \Big( p_{s-1}(\mathbf{x} -\mathbf{R}_{t,s-1}) - p_{s}(\mathbf{x} -\mathbf{R}_{t,s}) \Big) + \nabla p(\mathbf{x})n\Bar{\Xb}_n \\
=&\sum_{t=1}^n \sum_{s=1}^{\infty} \Big( p_{s-1}(\mathbf{x} -\mathbf{R}_{t,s-1}) - p_{s}(\mathbf{x} -\mathbf{R}_{t,s}) \pm \bm{1}_{\{s\geq r\}}\nabla p_{s-1}(\mathbf{x} -\mathbf{R}_{t,s}) A_s \mathbf{Z}_{t-s}\Big) +\nabla p(\mathbf{x}) \sum_{t=1}^n \sum_{s=1}^{\infty} A_s \mathbf{Z}_{t-s} \\
=&\underbrace{\sum_{t=1}^n \sum_{s=1}^{\infty} \Big( p_{s-1}(\mathbf{x} -\mathbf{R}_{t,s-1}) - p_{s}(\mathbf{x} -\mathbf{R}_{t,s}) + \bm{1}_{\{s\geq r\}}\nabla p_{s-1}(\mathbf{x} -\mathbf{R}_{t,s})  A_s \mathbf{Z}_{t-s} \Big)}_{=:T_n^{(1)}(\mathbf{x})} \\ 
&-\sum_{t=1}^n \sum_{s=1}^{\infty}  \bm{1}_{\{s\geq r\}}\nabla p_{s-1}(\mathbf{x} - \mathbf{R}_{t,s})  A_s \mathbf{Z}_{t-s}+\nabla p(\mathbf{x})  \sum_{t=1}^n\left( \sum_{s=0}^{r-1}A_s\mathbf{Z}_{t-s}+\sum_{s=r}^{\infty} A_s \mathbf{Z}_{t-s} \right) \\
=& T_n^{(1)}(\mathbf{x})+ \underbrace{\sum_{t=1}^n\sum_{s=r}^{\infty} \Big( \nabla p(\mathbf{x})  -\nabla p_{s-1}(\mathbf{x} - \mathbf{R}_{t,s})\Big)^{\top} A_s \mathbf{Z}_{t-s}}_{=:T_n^{(2)}(\mathbf{x})} + \underbrace{\nabla p(\mathbf{x})   \sum_{t=1}^n \sum_{s=0}^{r-1} A_s \mathbf{Z}_{t-s}}_{=:T_n^{(3)}(\mathbf{x})} \;.
\end{split}
\label{T}
\end{align}
To conclude that $ n^{-( \frac{1}{2}+d) } |S_n (\mathbf{x})|\xrightarrow{\PP}0$, it is enough to show that  $ n^{-( \frac{1}{2}+d) } |T_n^{(m)} (\mathbf{x})|\xrightarrow{\PP}0$ for $m=1, 2,3.$ In fact, by Chebyshev's inequality, for every fixed $\varepsilon >0,$
$$ \PP\left(  |T_n^{(m)} (\mathbf{x})|> \varepsilon n^{d+1/2}\right) \leq \frac{\text{Var}(T_n^{(m)} (\mathbf{x}))}{\varepsilon^2 n^{2d+1}}\xrightarrow{n\to \infty }0, \quad \text{for }m=1,2,3\;.$$

To complete the proof, we will now show that there exists $\xi>0$ sufficiently small such that
\begin{enumerate}
    \item[(i)]  $ \Var\left(T_{n}^{(1)}(\mathbf{x})\right)  = \mathcal{O}( n^{2d})$
    \item[(ii)]  $ \Var\left(T_{n}^{(3)}(\mathbf{x})\right)  = \mathcal{O}( n^{})$
    \item[(iii)]  $\Var\left(T_{n}^{(2)}(\mathbf{x})\right) = \begin{cases}
        \mathcal{O}(n^{}) \quad \text{for }d\in (0,1/4)\\
        \mathcal{O}(n^{4d+\zeta}) \quad \text{for }d\in [1/4,1/2)\\
    \end{cases}\;.$
\end{enumerate}
\textit{Proof for (i):} Defining
\begin{align*}
K_{t,j}(\mathbf{x}) := & p_{j-1}(\mathbf{x} - \mathbf{R}_{t,j-1}) - p_{j}(\mathbf{x} - \mathbf{R}_{t,j}) + \bm{1}_{\{j\geq r\}} \nabla p_{j-1}(\mathbf{x} - \mathbf{R}_{t,j})   A_j \mathbf{Z}_{t-j}\;,
\end{align*}
we have $T_{N}^{(1)}(\mathbf{x}) = \sum_{t=1}^{n} \sum_{j=1}^{\infty } K_{t,j}(\mathbf{x})\;.$  From Lemma \ref{lemma6.4}, $ \EE\left[K_{t,j}(\mathbf{x})  K_{t',j'} (\mathbf{x})\right]=0$ whenever $j'\neq t'-t+j$. 
For the product $ \EE\left[K_{t,j} (\mathbf{x}) K_{t',t'-t+j}(\mathbf{x})\right]$, using Lemma \ref{upperbound_rest} (which can be applied because of assumption \eqref{eq:LIP-condition1}), there exist positive constants $\C_1,\C_2, \C_3, \C_4$ such that 
\begin{align*}
    K_{t,j} (\mathbf{x})  K_{t',t'-t+j}(\mathbf{x}) \leq \C_1 \C_3 \|A_j\|^2\|A_{t'-t+j}\|^2 \left( \|\mathbf{Z}_{t-j}\|^2\|\mathbf{Z}_{t'-(t'-t+j)}\|^2 +\C_4\| \mathbf{Z}_{t-j}\|^2 +\C_2 \| \mathbf{Z}_{t'-(t'-t+j)}\|^2 +\C_2\C_4 \right) \;.
\end{align*}
Therefore, by taking expectation on both sides and using the fact that $A_j \sim j^{d-1} A_\infty$ and that the forth moment of $\mathbf{Z}_1$ is bounded by assumption, we obtain $  \EE\left[K_{t,j} (\mathbf{x}) K_{t',t'-t+j}(\mathbf{x})\right] \leq \C j^{2(d-1)} (j+ t'-t)^{2(d-1)}\;.$ Using the arguments above and because of $d<1$,
\begin{align*}
 \Var \left( T_{n}^{(1)}(\mathbf{x})\right) & \leq \EE\left[ \left( \sum_{t=1}^{n} \sum_{j=1}^{\infty} K_{t,j}\right)^2 \right] = \mathbb{E}\left[\sum_{t=1}^{n} \sum_{j=1}^{\infty} \sum_{t'=1}^{n} \sum_{j'=1}^{\infty} K_{t,j}  K_{t',j'}\right] \\
& = \,\sum_{t'=1}^{n} \sum_{j=1}^{\infty} \sum_{t=1}^{\min\{n,  (t'+j-1)\}} \EE\left[K_{t,j}  K_{t',t'-t+j}\right] \leq \C \,\sum_{t'=1}^{n}  \sum_{t=1}^{n} \sum_{j=1}^{\infty} j^{2(d-1)} (j+ |t'-t|)^{2(d-1)}\;  \quad \\
& \leq \C \,\sum_{t'=1}^{n}  \sum_{t=1}^{n}  (t'-t)^{2(d-1)} 
 \leq  \C n^{2d }, \quad \text{using} \ |t'-t|\leq n \;.\\
\end{align*}
\textit{Proof for (ii):}  By definition of $T_n^{(3)}(\mathbf{x})$ and the Cauchy-Schwarz inequality, $\left(T_n^{(3)}(\mathbf{x})\right)^2= (\nabla p (\mathbf{x})  \sum_{t=1}^{n}\sum_{i=1}^{r-1} A_i \mathbf{Z}_{t-i})^2\leq \C \| \nabla p(\mathbf{x})  \|^2   \|\sum_{t=1}^{n} \sum_{i=1}^{r-1} A_i\mathbf{Z}_{t-i} \|^2\leq \C \sum_{t,t'=1}^n \sum_{i,j=1}^{r-1} \langle A_i \mathbf{Z}_{t-i}, A_j\mathbf{Z}_{t'-j}\rangle.$ Moreover, using the independence of the sequence $(\mathbf{Z}_j)_{j \in \NN},$ $\Var\left(T_{n}^{(3)}(\mathbf{x})\right) \leq \EE[T_{n}^{(3)}(\mathbf{x})^2]\leq \mathcal{C} n.$

\textit{Proof for (iii):} By definition, $T_{n}^{(2)}(\mathbf{x}) =\sum_{t=1}^{n} \sum_{j=r}^{\infty } h_{t,j} A_j \mathbf{Z}_{t-j}$ with
\begin{align*}
    h_{t,j}(\mathbf{x}) := \left( \nabla p(\mathbf{x} - \mathbf{R}_{t,j}) - \nabla p_{j-1}(\mathbf{x} - \mathbf{R}_{t,j}) \right)^\top\;.
\end{align*}
By the convention established at the beginning of this article, gradients are considered row vectors. Therefore, we obtain 
\begin{align*}
    \Var \left(T_{n}^{(2)}(\mathbf{x})\right)  \leq \, &    \mathbb{E}\Big( \sum_{t=1}^{n} \sum_{j=r}^{\infty} h_{t,j}(\mathbf{x}) A_j\mathbf{Z}_{t-j}\Big)^2  \\
  \leq &    \sum_{t=1}^{n} \sum_{j=r}^{\infty} \sum_{t'=t}^{n} \EE\Big( h_{t,j}(\mathbf{x}) A_j\mathbf{Z}_{t-j}  h_{t',j'}(\mathbf{x}) A_{j'}\mathbf{Z}_{t'-j'}\Big)  \;,\quad t'-t+j=j'\;,
\end{align*}
where we used again the fact that if $t-j \neq t' - j' $, then
\[ \EE\Big(h_{t,j}(\mathbf{x})A_{j} \mathbf{Z}_{t-j}h_{t',j'}(\mathbf{x})A_{j'} \mathbf{Z}_{t'-j'} \Big) = 0 \;. \]
In fact,  $R_{t,j}= \sum_{i=j+1}^\infty A_i \mathbf{Z}_{t-i} = \tilde{f}(\mathbf{Z}_{t-j-1},\mathbf{Z}_{t-j-2}, \ldots) $ for some function $\tilde{f}$. Thus, 
$h_{t,j}(\mathbf{x}) A_j \mathbf{Z}_{t-j}= f_1( \mathbf{Z}_{t-j-1},\mathbf{Z}_{t-j-2}, \ldots) A_j \mathbf{Z}_{t-j}$  and $\sigma(\mathbf{Z}_{t-j})$ is independent from $\F_{t-j-1}= \sigma (  \mathbf{Z}_{t-j-1},\mathbf{Z}_{t-j-2}, \ldots)\;. $ Then, if $t-j \neq t'-j'$ (we can assume $t'-j'<t-j$)
\begin{align*}
   & \EE[h_{t,j}(\mathbf{x}) A_j\mathbf{Z}_{t-j} h_{t',j'}(\mathbf{x}) A_{j'}\mathbf{Z}_{t'-j'} ]= \EE [ f_1( \mathbf{Z}_{t-j-1},\mathbf{Z}_{t-j-2}, \ldots) A_j \mathbf{Z}_{t-j} f_2( \mathbf{Z}_{t'-j'-1},\mathbf{Z}_{t'-j'-2}, \ldots) A_{j'} \mathbf{Z}_{t'-j'}  ]\\
    =& \EE \left[  \EE\left[  f_1( \mathbf{Z}_{t-j-1},\mathbf{Z}_{t-j-2}, \ldots) A_j \mathbf{Z}_{t-j} f_2( \mathbf{Z}_{t'-j'-1},\mathbf{Z}_{t'-j'-2}, \ldots) A_{j'} \mathbf{Z}_{t'-j'}\, \Big| \, \F_{t'-j'}\right]  \right]\\
     =& \EE \left[  f_2( \mathbf{Z}_{t'-j'-1},\mathbf{Z}_{t'-j'-2}, \ldots) A_{j'} \mathbf{Z}_{t'-j'} \EE\left[  f_1( \mathbf{Z}_{t-j-1},\mathbf{Z}_{t-j-2}, \ldots) A_j \mathbf{Z}_{t-j} \, \Big| \, \F_{t'-j'}\right]  \right]\\
     =& \EE \left[  f_2( \mathbf{Z}_{t'-j'-1},\mathbf{Z}_{t'-j'-2}, \ldots) A_{j'} \mathbf{Z}_{t'-j'} A_j \EE\left[\mathbf{Z}_{t-j} \, \Big| \, \F_{t'-j'}\right] 
     \EE\left[  f_1( \mathbf{Z}_{t-j-1},\mathbf{Z}_{t-j-2}, \ldots)  \, \Big| \, \F_{t'-j'}\right]  \right]\\
      =& \EE \left[  f_2( \mathbf{Z}_{t'-j'-1},\mathbf{Z}_{t'-j'-2}, \ldots) A_{j'} \mathbf{Z}_{t'-j'} A_j \EE\left[\mathbf{Z}_{t-j} \right] 
     \EE\left[  f_1( \mathbf{Z}_{t-j-1},\mathbf{Z}_{t-j-2}, \ldots)  \, \Big| \, \F_{t'-j'}\right]  \right]=0\;.\\
\end{align*}
Now, for an argument used a few times by now
\begin{align*}
p(\mathbf{x}) &  = \mathbb{P}( \Xb_{t,j-1}+ \mathbf{R}_{t,j-1}  \leq \mathbf{x} )  = \int \mathbb{P}( \mathbf{X}_{t,j-1} \leq \mathbf{x} - \mathbf{t})dF_{ \mathbf{R}_{t,j-1}}(\mathbf{t}) = \int  p_{j-1}(\mathbf{x} - \mathbf{t})dF_{ \mathbf{R}_{t,j-1}}(\mathbf{t})\;,
\end{align*}
hence, by applying the dominated convergence theorem  (The partial derivatives are bounded by assumption by a constant, and every constant is integrable with respect to the measure $dF_{\mathbf{R}_{t,j-1}}$)
\begin{align*}
h_{t,j}(\mathbf{x}) & =  \int \Big( \nabla p_{j-1}(\mathbf{x} - \mathbf{t}) - \nabla p_{j-1}(\mathbf{x} -\mathbf{R}_{t,j})\Big) dF_{  \mathbf{R}_{t,j-1}}(\mathbf{t})\;.
\end{align*}
Therefore, applying the Cauchy-Schwarz inequality and Jensen's inequality, we obtain
 \begin{align*}
 \Bigl|  h_{t,j}( \mathbf{x})A_{j} \mathbf{Z}_{t-j} \Bigr|  =&  \bigl| \langle h_{t,j}( \mathbf{x})\;, A_{j} \mathbf{Z}_{t-j} \rangle \bigl| \\
 \leq& \C \| A_{j} \mathbf{Z}_{t-j} \|  \int  \| \mathbf{R}_{t,j}-\mathbf{t} \|  dF_{ \mathbf{R}_{t,j-1}}(\mathbf{t})\\
\leq&  \C \| A_{j} \mathbf{Z}_{t-j} \|  \left(\, \int  \| \mathbf{R}_{t,j}\| +\|\mathbf{t} \|  dF_{ \mathbf{R}_{t,j-1}}(\mathbf{t}) \right) \\
\leq&  \C \| A_{j} \mathbf{Z}_{t-j} \|  \big( \| \mathbf{R}_{t,j}\|+\EE[ \|  \mathbf{R}_{t,j-1} \|]\big)\;.\\
\end{align*}
Applying these upper bounds for $t-j=t'-j'$ ( this latter condition on the indexes implies that both $\mathbf{R}_{t,j}$  and $\mathbf{R}_{t,j}$ are $\F_{t-j}-$measurable, i.e  $\mathbf{R}_{t,j} = \tilde{f}_1(\mathbf{Z}_{t-j-1}, \mathbf{Z}_{t-j-2}, \ldots)$ and  $\mathbf{R}_{t',j'} = \tilde{f}_2(\mathbf{Z}_{t-j-1}, \mathbf{Z}_{t-j-2}, \ldots)$ )
\begin{align*}
&   \EE\Big[ \left( h_{t,j}( \mathbf{x})A_j \mathbf{Z}_{t-j} \right)\left(   h_{t',j'}( \mathbf{x})  A_{j'} \mathbf{Z}_{t'-j'}\right) \Big] \\
&\leq \EE \left[  \Bigl|  h_{t,j}( \mathbf{x}) A_{j} \mathbf{Z}_{t-j} \Bigr| \,  \Bigl|  h_{t',j'}( \mathbf{x}) A_{j'} \mathbf{Z}_{t'-j'} \Bigr| \right]\\
&\leq  \C \, \|A_j \| \cdot  \|A_{j'} \| \, \EE\Big[ \|\mathbf{Z}_{t-j}\|^2 \|\mathbf{Z}_{t'-j'}\|^2 \left(\| \mathbf{R}_{t,j}\|+\EE[ \|  \mathbf{R}_{t,j} \|] \right)\cdot \left( \| \mathbf{R}_{t',j'}\|+\EE[ \|  \mathbf{R}_{t',j'} \|]\right)\Big]\\
&\leq  \C \, \|A_j \| \cdot  \|A_{j'} \| \, \EE\Big[ \|\mathbf{Z}_{t-j}\|^4\Big] \EE\Big[  \left(\| \mathbf{R}_{t,j}\|+\EE[ \|  \mathbf{R}_{t,j} \|] \right)\cdot \left( \| \mathbf{R}_{t',j'}\|+\EE[ \|  \mathbf{R}_{t',j'} \|]\right)\Big] \quad \text{independence}  \\
& \leq  \C \, \|A_j \| \cdot  \|A_{j'} \| \, \left( \EE\Big[ \|  \mathbf{R}_{t,j}\| \cdot  \| \mathbf{R}_{t',j'}\| \Big] 
+3\EE[ \|  \mathbf{R}_{t,j} \|]\cdot \EE[ \|  \mathbf{R}_{t',j'} \|] \right)  \\
&\leq  \C \, \|A_j \| \cdot  \|A_{j'} \| \,  4  \sqrt{\EE[ \|  \mathbf{R}_{t,j}\|^2]}\sqrt{\EE[ \|  \mathbf{R}_{t',j'}\|^2]} \quad \text{Cauchy-Schwarz}\;.
\end{align*}
Now, we know that the special structure of the coefficients of $\Xb_t$ which ensures  long-range dependent behavior, implies by Lemma \ref{lemma:upper_bound_residuals_R_{t,j}}, there exists $\zeta>0$ sufficiently small such that 
$\mathbb{E} \| \mathbf{R}_{t,j} \|^2 \leq \mathcal{C}  j^{2d -1+\zeta}$ and therefore,
\begin{align*}
    \EE\Big[    h_{t,j}( \mathbf{x}) A_j \mathbf{Z}_{t-j} ,    h_{t',j'}( \mathbf{x})  A_{j'} \mathbf{Z}_{t'-j'} \Big]
&\leq \C \|A_j\| \cdot  \|A_{j'} \|  \sqrt{\EE[ \|  \mathbf{R}_{t,j}\|^2]}\sqrt{\EE[ \|  \mathbf{R}_{t',j'}\|^2]} \\
 &\leq \C  \|A_j \| \cdot  \|A_{j'} \|   j^{d-1/2+\zeta/2} j'^{ d-1/2+\zeta/2} \\
 & \leq C \|A_{j} \| \cdot \| A_{j'} \| (j j')^{d-1/2+\zeta/2} \;.
\end{align*}
Taking into account that $j' = (t'-t)+j $, and that 
 $A_j \sim j^{d-1}A_\infty$, we find 
\begin{align*}
    \Var(T_{n}^{(2)})  \leq &\C \sum_{t=1}^{n} \sum_{t'=t}^{n} \sum_{j=1}^{\infty} \|A_j\| \cdot \| A_{j'}\| (j j')^{d-1/2+\zeta/2}\\
    \leq &\C \sum_{t=1}^{n} \sum_{t'=t}^{n} \sum_{j=1}^{\infty} (j((t'-t)+j))^{d-1+\zeta/2}(j((t'-t)+j))^{d-1/2+\zeta/2} \\
    \leq &\C n \sum_{t=1}^{n} \sum_{j=1}^{\infty} (j((n-t)+j))^{d-1+\zeta/2}(j((n-t)+j))^{d-1/2+\zeta/2} \\
  = & \C  n \sum_{t=0}^{n-1} \sum_{j=1}^\infty (j(t+j))^{d-1+\zeta/2}(j(t+j))^{d-1/2+\zeta/2} \\
  \leq & \C n \sum_{t=1}^n\sum_{j=1}^\infty (j(t+j))^{2d -3/2+\zeta} \;.
\end{align*}
For $d\in (0,1/4),$ we can choose $\zeta$ sufficiently small such that $\gamma := 3/2-2d-\zeta > 1$ and by Lemma \ref{lemma6.5},
     $$n \sum_{t=1}^n \sum_{j=1}^\infty (j(t+j))^{2d -3/2+\zeta}\lesssim  n\sum_{t=1}^n t^{-\gamma} < n\sum_{t=1}^\infty t^{-\gamma }< \mathcal{O}(n)\;. $$
For $d\in [1/4, 1/2),$ we can choose $\zeta$ sufficiently small such that $\gamma := 3/2-2d-\zeta \in (1/2,1)$ and by Lemma \ref{lemma6.5},

     $$n \sum_{t=1}^n \sum_{j=1}^\infty (j(t+j))^{2d -3/2+\zeta}\leq n\sum_{t=1}^n t^{-2\gamma+1} \sim n\int_1^n   t^{-2\gamma+1 }dt = n \mathcal{O}(n^{2-2\gamma}) = \mathcal{O}(n^{4d+2\zeta})\;. $$
     To simplify the expression, we can set $\zeta'\leftarrow 2\zeta$, and rename $\zeta'$ as $\zeta$.
\end{proof}

\subsection{Auxiliary results for reduction principle} 
\begin{lemma} Let $\Xb_t=\sum_{j=0}^\infty A_j \mathbf{Z}_{t-j}$ with $A_j \overset{j \to \infty }{\sim} j^{d-1}L(j)$ with $d\in (0,1/2)$, $L$ is a continuous slowlying varying matrix, and $(\mathbf{Z}_j)_{j \in \ZZ}\overset{i.i.d.}{\sim} \W\N (\mathbf{0},\Sigma)$. For all $0<\zeta<|2d-1|$, the following upper bound for  $ \mathbf{R}_{t,j} $ holds:
\label{lemma:upper_bound_residuals_R_{t,j}}
\begin{equation}
    \label{upper-bound-residuals-of-X_t-tilde}
    \mathbb{E} \left[\| \mathbf{R}_{t,j} \|^2 \right] \leq  \mathcal{C}\,  j^{2d -1+\zeta} \;,
\end{equation}
where $\C>0$ is a constant.
\end{lemma}
\begin{proof} In the following, we will denote with $\C$ a generic constant. Observe that $(\mathbf{Z}_j)_{j \in \ZZ}$ are orthogonal in $L^2(\Omega, \PP)$ with respect to the scalar product $\langle X,Y \rangle:=\EE[XY]$ (note that $\EE[\mathbf{Z}_i]=\mathbf{0} $ for all $i$)    because $\mathbf{Z}_i$ are independent, and therefore $\mathbb{E} \| \mathbf{R}_{t,j} \|^2 \leq \sum_{i >  j} \mathbb{E} \| A_i \mathbf{Z}_{t-i} \|^2 $. Moreover, by the assumptions on the coefficients $A_j$ and boundedness of the variance of $\| \mathbf{Z}_j\| $, we have 
     \begin{align*}
\mathbb{E} \| \mathbf{R}_{t,j} \|^2 \leq \sum_{i >  j} \mathbb{E} \| A_i \mathbf{Z}_{t-i} \|^2 \leq \C  \sum_{i >  j} i^{-2(1-d)} \|L(i)\|^2  \;.
\end{align*} 
We would like to affirm that there exists a constant $\C>0$, such that $\|L(i)\|^2\leq \C i^{\zeta}$ for all $i\geq j$. Note two things:
\begin{enumerate}
    \item $L(\cdot)$ is a matrix whose entries are slowlying varying functions. By the principle of equivalence of norm on finite vector spaces, there exists $\C>0$, such that $\| \cdot \| \leq \C \| \cdot \|_2 $. Further, $\| \cdot \|_2 \leq \| \cdot \|_F^2$, where $ \| \cdot \|_F^2$ is the Frobenius norm, i.e $ \| A \|_F^2 = \left( \sum_{hs}A_{hs}^2\right)^{1/2}$. Thus, 
    $$ \| L(\cdot)\|^2 \leq \C\sum_{h=1}^r\sum_{s=1}^r |L_{hs}(\cdot)|^2 \overset{\text{def}}{=}\mathcal{L}(\cdot)\;.$$
    \item  $\mathcal{L}(\cdot)$ is slowlying varying at infinity, since the square of a slowlying varying function is a slowly varying function, and the sum of slowlying varying functions is  still slowlying varying. Lastly, $\mathcal{L}$ is continuous, as composition of its continuous entries $L_{hs}$. 
\end{enumerate}
In conclusion, $\mathcal{L}(\cdot)$ is slowying varying real function. As a consequence of Karamata's representation theorem (For instance \citet{pipiras_taqqu_2017}, relation 2.2.3 ), if $\mathcal{L}(x)$ is a slowlying varying function at infinity, 
$$ \lim_{x\to \infty} \mathcal{L}(x)x^{-\zeta} =0\quad \text{for all }\zeta>0\;,$$

or equivalently, for all $\varepsilon >0$ there exists $x_0 \in \mathbb{R}$ such that 
$$ \mathcal{L}(x) \leq \varepsilon x^{\zeta}\quad \text{for }x>x_0\;.$$

It must be noticed that the we cannot yet conclude $\|L(i)\|^2\leq \C i^{\zeta}$ for all $i\geq j$, but only for $i\geq x_0$.  If $x_0 \leq j$, we are done. So, we can assume $j<x_0 $. Note that the function $\mathcal{L}(x) x^{-\zeta}$ is continuous and in $[0,x_0]\supset [j, x_0]$, which is compact in $\RR$. Therefore, there exists $C_1$ such that $\mathcal{L}(x) x^{-\zeta}\leq C_1$, for all $x\in [0,x_0]\;.$ 

Now we possess all the elements to conclude. Consider a fixed $\varepsilon>0$ and corresponding $x_0\;.$ Then, 

\begin{align*}
\mathbb{E} \| \mathbf{R}_{t,s} \|^2 \leq& \C \sum_{i >  j} i^{-2(1-d)} \|L(i)\|^2 \leq \C   \sum_{i >  j} i^{-2(1-d)} \mathcal{L}(i) \\ 
=&\sum_{i>j}^{[x_0]}  i^{-2(1-d)} \mathcal{L}(i) + \sum_{i>[x_0]} i^{-2(1-d)} \mathcal{L}(i) \\
\leq&\sum_{i>j}^{[x_0]}  i^{-2(1-d)} i^{\zeta} i^{-\zeta} \mathcal{L}(i) + C_2\sum_{i>[x_0]} i^{-2(1-d) +\zeta}\;, \quad \text{for any }C_2\geq \varepsilon \\
\leq& C_1 \sum_{i>j}^{[x_0]}  i^{-2(1-d)+\zeta}   + C_2\sum_{i>[x_0]} i^{-2(1-d) +\zeta}  \\
\leq& \max\{C_1, C_2\} \sum_{i>j}  i^{-2(1-d)+\zeta}   = \C \sum_{i>j}  i^{-2(1-d)+\zeta} \;. \\
\end{align*} 
Note that the last $
\C$ does not depend on $j$.  Finally, the last sum is asymptotically equivalent to 
$$\sum_{i >  j} i^{-2(1-d)+\zeta}  \sim \int_j^\infty x^{-2(1-d) +\zeta}dx =\frac{1}{2d-1+\zeta} x^{2d-1+\zeta}\Big|^\infty _j = \C j ^{2d-1+\zeta}\;.$$

In the last step we used the fact that we chose $\zeta \in (0,|2d-1|)$, thus $2d-1+\zeta< 0$. Finally, 
 \begin{align*}
\mathbb{E} \| \mathbf{R}_{t,s} \|^2 \leq \sum_{i >  j} \mathbb{E} \| A_i \mathbf{Z}_{t-i} \|^2 \leq  \mathcal{C}^2   \sum_{i >  j} i^{-2(1-d)+\zeta} \leq \mathcal{C}^2  j^{2d-1+\zeta} \;.
\end{align*} 
\end{proof}
\begin{lemma}
  \label{sigma_algebra}
    Let  $\F_{t-j}=\sigma\left( \mathbf{Z}_{t-j}, \mathbf{Z}_{t-j-1}, \ldots\right)$ so that $\F_{t-(j+1)} \subset \F_{t-j} \subset \cdots \subset \F_{t-1}\subset \F_{t}. $ Consider $\Xb_t=\sum_{j=0}^\infty A_j \mathbf{Z}_{t-j}$, where $A_j \in \RR^{r\times r}$. Then, for all $\mathbf{x}\in \RR^r$ the following hold
    \begin{itemize}
        \item [(i)]  $\bm{1}\left(\{\Xb_t\leq \mathbf{x}\}\right)= \PP( \Xb_t \leq \mathbf{x}| \F_{t}) \;,$
        \item  [(ii)] $ \PP( \Xb_t \leq \mathbf{x}) \overset{a.s}{=}\lim\limits_{j\to \infty }\PP \left( \Xb_t \leq \mathbf{x}| \, \F_{t-j}\right)\;,$
        \item [(iii)] $  \PP \left( \Xb_t \leq \mathbf{x} | \, \F_{t-j}\right) =p_j( \mathbf{x} -  \mathbf{R}_{t,j})\;.$
    \end{itemize}
\end{lemma}
\begin{proof} Let $K_{-j}=\EE[ \bm{1}\left(\{\Xb_t\leq \mathbf{x}\}\right)|\F_{t-j}]=\PP( \mathbf{X}_t \leq \mathbf{x}| \F_{t-j})$ for all  $j  \in \NN$. 
    \begin{itemize}
        \item [(i)]  $\mathbf{X}_t$ is $\F_{t}$ measurable, and  the previous theorem yields
$$ \PP( \mathbf{X}_t \leq \mathbf{x}| \F_{t}) = \EE \left[ \bm{1}\left(\{\mathbf{X}_t\leq \mathbf{x}\}\right)|\F_{t} \right] =\bm{1} \left(\{\mathbf{X}_t\leq \mathbf{x}\}\right) =K_0\;.$$
\item [(ii)] It is clear that $K_{-j}$ is $\F_{t-j}$- measurable and integrable. Also, for $j\geq i$ integer,
\begin{align*}
    \EE[K_{-i}|\F_{t-j}]=  \EE[\EE[ \bm{1}\left(\{\Xb_t\leq \mathbf{x}\}\right)|\F_{t-i}]|\F_{t-j}]\underbrace{=}_{\F_{t-j}\subset \F_{t-i}}  \EE[ \bm{1}\left(\{\Xb_t\leq \mathbf{x}\}\right)|\F_{t-j}]=K_{-j}\;.
\end{align*}
Therefore, $(K_{-j})_{j\in \NN}$ is a backwards martingale with respect to $(\F_{t-j})_{j\in \NN}$. Note that we applied the tower property to the conditional expectation. Using the backwards martingale convergence theorem  (for instance \citet{probability_essentials}, p. 233, Theorem 27.4), $K_{-j}:= \PP( \mathbf{X}_t \leq \mathbf{x}| \F_{t-j})$ converges almost surely and in probability to a limit $K_{-\infty}$ that is $\F_{t}= \bigcap_{i=0}^\infty\F_{t-i}$ measurable. By 1), the limit is $\PP( \mathbf{X}_t \leq \mathbf{x}| \F_{t})=\PP( \mathbf{X}_t \leq \mathbf{x})$. 
\item [(iii)] Note that $  \PP( \mathbf{X}_t \leq \mathbf{x}|\F_{t}) =  \PP( \mathbf{X}_{t,j} \leq \mathbf{x}-\mathbf{R}_{t,j} | \F_{t})=p_j(\mathbf{x}-\mathbf{R}_{t,j} ) \;.$
    \end{itemize}
    \end{proof}
\begin{lemma}[ Lemma 6.5 in \citet{Hsing}. ]
    Given a constant $\gamma>1/2$, $l\geq 1$, there exists $\C <\infty$ such that, for all $l \geq 1$, we have 
    $$\sum_{j=1}^{\infty} (j(l+j))^{-\gamma} \leq \begin{cases}
        \C l^{-2\gamma+1}, \quad \gamma \in (1/2, 1) \\
        \C \frac{\log l}{l}, \quad \gamma = 1\\
        \C l^{-\gamma}, \quad \gamma> 1\;.
    \end{cases}$$
    \label{lemma6.5}
\end{lemma}
\begin{lemma}[ Lemma 6.4 in \citet{Hsing}]
For all $\mathbf{x}$, $\mathbf{x'} \in \RR^{r}$, and $t,t',j,j'\geq 1$ such that $t'-j'\neq t-j$, 
\label{lemma6.4}
\begin{itemize}
    \item  [(i)]$\EE[U_{t,j}(\mathbf{x})U_{t',j'}(\mathbf{x})]=\Cov \left( p_{j-1}(  \mathbf{x} - \mathbf{R}_{t,j-1}) - p_{j}( \mathbf{x} - \mathbf{R}_{t, j} ) ,p_{j'-1}( \mathbf{x}  - \mathbf{R}_{t', j'-1}) - p_{j'}( \mathbf{x} - \mathbf{R}_{t', j'} ) \, \right)=0\;.$
    \item  [(ii)] $\Cov \left(p_{j-1}( \mathbf{x} - \mathbf{R}_{t, j-1}) - p_{j}( \mathbf{x} -\mathbf{R}_{t, j} ), \nabla p_{j'-1}( \mathbf{x'}  - \mathbf{R}_{t', j'-1}) \mathbf{Z}_{t'-j'}\right)=0\;.$
\end{itemize}
\begin{proof}
(i) Let $\mathcal{F}_j= \sigma( Z_i \, | \, i<j)$. Notice that from the third point of Lemma \ref{sigma_algebra}, 
$$ \EE[\bm{1} \left( \{\mathbf{X}_t \leq \mathbf{x}\}\right)| \F_{t-j}] = \PP ( \mathbf{X}_t \leq \mathbf{x} | \F_{t-j}) = p_{j}( \mathbf{x} - \mathbf{R}_{t,j})\;.$$
Next, we rename $U_t(\mathbf{x})=\bm{1}\left(\{\mathbf{X}_t \leq \mathbf{x}\}\right) - p(\mathbf{x})$ and let $U_{t,j}(\mathbf{x})$ be the quantity 
\begin{align*}
     U_{t,j}(\mathbf{x})=& p_{j-1}(  \mathbf{x} - \mathbf{R}_{t,j-1}) -p(\mathbf{x}) + p(\mathbf{x})  - p_{j}( \mathbf{x} - \mathbf{R}_{t,j} )\\
     =&\EE[U_t(\mathbf{x})| \F_{t-j+1}] - \EE[U_t(\mathbf{x})| \F_{t-j}]\;.
\end{align*}
Our goal is to prove that $\EE[U_{t,j}(\mathbf{x})U_{t',j'}(\mathbf{x}) ]=0$. Since by assumption $t'-j'\neq t-j$, without loss of generality we can assume $t'-j'< t-j$. Then, $\F_{t'-j'}\subset \F_{t-j}$. Moreover, $\EE[U_t(\mathbf{x})| \F_{t-j}]$ is $\F_{t-j}$ measurable. 
\begin{align*}
  \EE[U_{t,j}(\mathbf{x})U_{t',j'}(\mathbf{x})] =&  \EE\Big[U_{t',j'} (\mathbf{x})\Big(\EE[U_t(\mathbf{x})| \F_{t-j+1}] - \EE[U_t(\mathbf{x})| \F_{t-j}]\Big)  \Big]\quad \text{def of }U_{t,j}\\
  =& \EE\left[\EE\Big[U_{t',j'}(\mathbf{x})\Big(\EE[U_t(\mathbf{x})| \F_{t-j+1}] - \EE[U_t(\mathbf{x})| \F_{t-j}]\Big)  \Big| \F_{t'-j'}\right] \Big] \quad \text{Tower property}\\
   =& \EE\left[U_{t',j'}(\mathbf{x})\EE\Big[ \Big(\EE[U_t(\mathbf{x})| \F_{t-j+1}] - \EE[U_t(\mathbf{x})| \F_{t-j}]\Big)  \Big| \F_{t'-j'}\right] \Big] \quad U_{t'-j'} \text{ is }\F_{t'-j'}\text{-measurable}\\
   =& \EE\left[U_{t',j'}(\mathbf{x})\Big(\EE\Big[ \EE[U_t(\mathbf{x})| \F_{t-j+1}]\Big| \F_{t'-j'}\right] - \EE\left[ \EE[U_t(\mathbf{x})| \F_{t-j}]  \Big| \F_{t'-j'}\right] \Big)\Big] \\
    =&\EE\Big[U_{t',j'}(\mathbf{x})\Big(\EE[U_t(\mathbf{x})| \F_{t'-j'}] - \EE[U_t(\mathbf{x})| \F_{t'-j'}]\Big)  \Big]=0,\quad \quad \F_{t'-j'}\subset \F_{t-j}\subset \F_{t-j+1}\;.\\
\end{align*}
(ii) can be proved similarly since $\nabla p_{j'-1}( \mathbf{x'}  - {\mathbf{R}}_{t', j'-1}) \mathbf{Z}_{t'-j'}$ is $\F_{t'-j'}$ - measurable. 
\end{proof}
\end{lemma}
\begin{lemma} \label{upperbound_rest}
Suppose the assumptions of Theorem \ref{theorem:reduction_principle} are satisfied with integer $r.$ 
Let $K_{t,j}$ be defined as $K_{t,j}(\mathbf{x}) :=  p_{j-1}(\mathbf{x} - \mathbf{R}_{t,j-1}) - p_{j}(\mathbf{x} - \mathbf{R}_{t,j}) + \bm{1}_{\{j\geq r\}} \nabla p_{j-1}(\mathbf{x} - \mathbf{R}_{t,j})   A_j \mathbf{Z}_{t-j}\;.$
Then, there exist two positive a constants $\C_1$ and $\C_2$ such that, for any $j \geq r $ and any $t \in \mathbb{N} $, 
\begin{align*}
| K_{t,j}(\mathbf{x}) | \leq   \C_1\| A_j \|^2 \left( \| \mathbf{Z}_{t-j} \|^2 +\C_2 \right)\;.
\end{align*}
\end{lemma}
\begin{proof}[Proof of Lemma \ref{upperbound_rest}]
   Suppose that relationship \eqref{eq:leading_term} below holds for all \( j \geq r \), where 
\(\delta_j(\mathbf{t}) = \lambda^\ast A_j (\mathbf{Z}_{t-j} - \mathbf{t})\), with \( \lambda^\ast \in (0,1) \), and \(\eta_j = \mu^\ast A_j \mathbf{Z}_{t-j}\), with \( \mu^\ast \in (0,1) \):
\begin{align}
 \label{eq:leading_term}
     K_{t,j}(\mathbf{x})=
     \left\langle \nabla^2 p_{j-1}( \mathbf{x} - \mathbf{R}_{t,j-1} + \eta_{t,j} )  A_j \mathbf{Z}_{t-j},A_j \mathbf{Z}_{t-j} \right\rangle - \frac{1}{2}   \int \left\langle \nabla^2 p_{j-1} \Big(  \mathbf{x}  -\mathbf{R}_{t,j-1} + \delta_{t,j}(\mathbf{t}) \Big) A_j( \mathbf{Z}_{t-j} - \mathbf{t} ), A_j( \mathbf{Z}_{t-j} - \mathbf{t} ) \right\rangle \, d\mathcal{G}(\mathbf{t})\;. 
\end{align}
    Then, it follows 
\begin{align*}
| K_{t,j}(\mathbf{x}) | 
\leq & \Bigl \|  \nabla^2 p_{j-1}(\mathbf{x}- \mathbf{R}_{t,j-1} + \eta_j ) \Bigr \| \cdot \Bigl \| A_j \mathbf{Z}_{t-j} \Bigr \|^2  +   \int  \Bigl \|  \nabla^2 p_{j-1}(\mathbf{x} - \mathbf{R}_{t,j-1} + \delta_{j}(\mathbf{t}))  \Bigr\| \cdot \Bigl\| A_j( \mathbf{Z}_{t-j} - \mathbf{t} ) \Bigr\|^2 \, d\mathcal{G}(\mathbf{t})\\ \leq & \C_1' \|A_j\|^2 \|\mathbf{Z}_{t-j}\|^2 + \C_2'\|A_j\|^2 \left( \|\mathbf{Z}_{t-j}\|^2 +\int \|\mathbf{t}\|^2 \, d\mathcal{G}(\mathbf{t}) \right) \leq \C_1\| A_j \|^2 \left( \| \mathbf{Z}_{t-j} \|^2 +\C_2 \right)\;.
\end{align*}
We used fact that the partial derivatives of $p_j$ are uniformly bounded for every $j$ by assumption \eqref{eq:LIP-condition1} and the assumed finite second moment of $\mathbf{Z}_1$. It remains to prove \eqref{eq:leading_term}.
By \eqref{eq:recursive_G}, we obtain 
\begin{equation} 
\label{p_j in terms of p_{j-1}}
    p_s(\mathbf{x}) = \mathbb{P}(\mathbf{X}_{t,s} \leq \mathbf{x} )  =  \mathbb{P}(\mathbf{X}_{t,s-1} + A_s \mathbf{Z}_{t-s} \leq \mathbf{x} )  =\int p_{s-1} (\mathbf{x}- A_s\mathbf{t} ) \, d\mathcal{G}(\mathbf{t}) \;.
\end{equation}
Via Equation (\ref{p_j in terms of p_{j-1}}), we can write
\begin{align*}
& p_{j-1}(\mathbf{x} -\mathbf{R}_{t,j-1}) - p_{j}(\mathbf{x} -\mathbf{R}_{t,j}) \\
= & \int \Big( p_{j-1}(\mathbf{x} -\mathbf{R}_{t,j-1}) -  p_{j-1} (\mathbf{x} -\mathbf{R}_{t,j}- A_j \mathbf{t} ) \Big) \, d\mathcal{G}(\mathbf{t}) \\
= & \int \Big( p_{j-1}(\mathbf{x} -\mathbf{R}_{t,j-1}) -  p_{j-1} (\mathbf{x} -\mathbf{R}_{t,j}-A_j\mathbf{Z}_{t-j}+A_j\mathbf{Z}_{t-j}  - A_j \mathbf{t} ) \Big) \, d\mathcal{G}(\mathbf{t}) \\
 =& \int \Big( p_{j-1}(\mathbf{x} -\mathbf{R}_{t,j-1}) -  p_{j-1} (\mathbf{x} -\mathbf{R}_{t,j-1}+A_j(\mathbf{Z}_{t-j}- \mathbf{t}) ) \Big) \, d\mathcal{G}(\mathbf{t})\;. \\
\end{align*}
Now, we consider the real valued function 
\begin{align*}
g(\lambda) : =   \int \Big( p_{j-1}( \mathbf{x} -\mathbf{R}_{t,j-1}) -  p_{j-1} ( \mathbf{x}  -\mathbf{R}_{t,j-1} + \lambda A_j( \mathbf{Z}_{t-j} - \mathbf{t} )) \Big) \, d\mathcal{G}(\mathbf{t}) \;.
\end{align*}
Next, we compute the derivatives of $g$. By assumption, the function $\mathbf{t}\mapsto  p_{j-1} (\mathbf{x} -\mathbf{R}_{t,j}- A_j \mathbf{t} )$ and its first derivatives are integrable with respect to the measure $\mathcal{G}(\mathbf{t})$ by assumption for all $j\geq r$. Therefore, we can use the dominated convergence theorem, and by taking derivatives  with respect to $\lambda$, we obtain
\begin{align*}
g'(\lambda) & = -   \int  \nabla p_{j-1} \Big(  \mathbf{x}  -\mathbf{R}_{t,j-1}+ \lambda A_j( \mathbf{Z}_{t-j} - \mathbf{t} )\Big) A_j( \mathbf{Z}_{t-j} - \mathbf{t} ) \, d\mathcal{G}(\mathbf{t})\;, \\
g''(\lambda) & = - \frac{1}{2}   \int  \left\langle\nabla^2 p_{j-1} \Big(  \mathbf{x}  -\mathbf{R}_{t,j-1}+ \lambda A_j( \mathbf{Z}_{t-j} - \mathbf{t} )\Big) A_j( \mathbf{Z}_{t-j} - \mathbf{t} ) ,A_j( \mathbf{Z}_{t-j} - \mathbf{t} ) \right\rangle \, d\mathcal{G}(\mathbf{t}) \;.
\end{align*}
Moreover, since $g \in C^2([0,1])$, the mean value theorem yields, $g(1) = g(0) + g'(0) + \frac{1}{2} g''(\lambda^\ast)$ for some value $\lambda^\ast \in (0,1)$ yields
\begin{align*}
&   p_{j-1}( \mathbf{x} -\mathbf{R}_{t,j-1}) - p_{j}( \mathbf{x} -\mathbf{R}_{t,j}) \\
 = & 0 -    \int  \nabla p_{j-1} \Big(  \mathbf{x}  -\mathbf{R}_{t,j-1} \Big) A_j( \mathbf{Z}_{t-j} - \mathbf{t} ) \, d\mathcal{G}(\mathbf{t}) \\
  - & \frac{1}{2}    \int \left \langle  \nabla^2 p_{j-1} \Big(  \mathbf{x}  -\mathbf{R}_{t,j-1}+ \lambda^\ast A_j( \mathbf{Z}_{t-j} - \mathbf{t} )\Big) A_j( \mathbf{Z}_{t-j} - \mathbf{t} ),A_j( \mathbf{Z}_{t-j} - \mathbf{t} ) \right \rangle \, d\mathcal{G}(\mathbf{t}) \\
  = & -    \nabla p_{j-1} \Big(  \mathbf{x}  -\mathbf{R}_{t,j-1} \Big) A_j \mathbf{Z}_{t-j} \\
  - &  \frac{1}{2}    \int  \left\langle \nabla^2 p_{j-1} \Big(  \mathbf{x}  -\mathbf{R}_{t,j-1}+ \lambda^\ast A_j( \mathbf{Z}_{t-j} - \mathbf{t} )\Big) A_j( \mathbf{Z}_{t-j} - \mathbf{t} ),A_j( \mathbf{Z}_{t-j} - \mathbf{t} ) \right\rangle \, d\mathcal{G}(\mathbf{t}) \;,
\end{align*}
using $\mathbb{E}[\mathbf{Z}_{t-j}]= 0$ for the second equality. However, the first order term of the approximation is $ -    \nabla p_{j-1} \Big(  \mathbf{x}  -\mathbf{R}_{t,j-1} \Big) A_j \mathbf{Z}_{t-j} $, whereas we want to obtain $ -    \nabla p_{j-1} \Big(  \mathbf{x}  -\mathbf{R}_{t,j} \Big) A_j \mathbf{Z}_{t-j} $.  This latter term shall appear in another Taylor decomposition. In other words, we consider
\[ \phi(\mu) :=    \nabla p_{j-1}( \mathbf{x} -\mathbf{R}_{t,j-1}) - \nabla p_{j-1}( \mathbf{x} -\mathbf{R}_{t,j-1} + \mu A_j \mathbf{Z}_{t-j} )\;.\]

There exists $\mu^\ast$ such that $\phi(1) = \phi(0) + \phi'(\mu^\ast)$ 
\[    \Big[ \nabla p_{j-1}( \mathbf{x} -\mathbf{R}_{t,j-1}) - \nabla p_{j-1}( \mathbf{x} -\mathbf{R}_{t,j})\Big]A_j \mathbf{Z}_{t-j} = -   \left\langle\nabla^2 p_{j-1}( \mathbf{x} - \mathbf{R}_{t,j-1} + \mu^\ast A_j \mathbf{Z}_{t-j}  )  A_j \mathbf{Z}_{t-j} , A_j \mathbf{Z}_{t-j} \right\rangle \;.\]
Re-arranging this last expression yields
\begin{align*}
  & -   \nabla p_{j-1}( \mathbf{x} -\mathbf{R}_{t,j-1})A_j \mathbf{Z}_{t-j} =\\ 
  &-   \nabla p_{j-1}( \mathbf{x} -\mathbf{R}_{t,j})A_j \mathbf{Z}_{t-j} +  \left\langle\nabla^2 p_{j-1}( \mathbf{x} - \mathbf{R}_{t,j-1} + \mu^\ast A_j \mathbf{Z}_{t-j}  )  A_j \mathbf{Z}_{t-j} , A_j \mathbf{Z}_{t-j} \right\rangle\;. \\
\end{align*}
Now we can simplify the notation by writing 
\begin{align*}
\begin{cases} 
 \eta_{t,j} = \mu^\ast A_j \mathbf{Z}_{t-j}\;, \\
\delta_{t, j}(\mathbf{t}) = \lambda^\ast A_j( \mathbf{Z}_{t-j} - \mathbf{t} )\;,
\end{cases}
\end{align*}
thus 
\begin{align*}
&   K_{t,j}( \mathbf{x})=  p_{j-1}( \mathbf{x} -\mathbf{R}_{t,j-1}) - p_{j}( \mathbf{x} -\mathbf{R}_{t,j}) +    \nabla p_{j-1} \Big(  \mathbf{x}  -\mathbf{R}_{t,j} \Big) A_j \mathbf{Z}_{t-j} \\
&=   \left[\nabla^2 p_{j-1}( \mathbf{x} - \mathbf{R}_{t,j-1} + \eta_{t,j} )  (A_j \mathbf{Z}_{t-j})^2  - \frac{1}{2}   \int \left\langle  \nabla^2 p_{j-1} \Big(  \mathbf{x}  -\mathbf{R}_{t,j-1} + \delta_{t,j}(\mathbf{t}) \Big) A_j( \mathbf{Z}_{t-j} - \mathbf{t} ) , A_j( \mathbf{Z}_{t-j} - \mathbf{t} ) \right\rangle \, d\mathcal{G}(\mathbf{t}) \right] \;.
\end{align*}
\end{proof}
    
\section{Proofs for Turning rate}
\label{appendix:turning_rate}
In order to prove Theorem \ref{theorem:cusum_statistics_limit} we will need the two following auxiliary results. 
\begin{corollary}
\label{asymptotic_turning_rate}
   Let $\xi_0, \ldots, \xi_{n+1}$ be time series data whose increments $X_1,\ldots, X_{n+1}$ satisfy the linear representation $X_t=\sum_{j=0}^\infty a_j Z_{t-j}$ with $\sum_{j=0}^\infty |a_j|< \infty$,  with  $Z_1$ admitting a continuous and bounded density and finite second moment $\EE \left[|Z_1|^2\right]<\infty$.   Then, for the turning rate estimator \eqref{eq:turning_rate_1} the following limit holds as $n\to \infty$

$$ \frac{1}{\sqrt{n}}\sum_{t=1}^{\lfloor n \tau \rfloor} \sum_{\gamma \in \mathcal{T}} \left( \bm{1}(\Pi(\xi_{t-1}, \xi_{t}, \xi_{t+1})=\gamma)-p(\gamma) \right)\overset{\D[0,1]}{\Longrightarrow} \, \sigma B(\tau)\;, \text{ for } \tau \in [0,1]\;,$$
   with variance 
   $$ \sigma ^2 =  \Var \left( \sum_{\gamma \in \mathcal{T}} \bm{1}\left( V_{\gamma   }\Xb_1 \leq \mathbf{0} \right) \right) + 2 \sum_{j=1}^\infty \Cov \left( \sum_{\gamma \in \mathcal{T}} \bm{1}\left( V_{\gamma   }\Xb_1 \leq \mathbf{0} \right), \sum_{\gamma \in \mathcal{T}} \bm{1}\left( V_{\gamma   }\Xb_{1+j} \leq \mathbf{0} \right) \right)\;,$$
where the matrices $V_\gamma$ are defined in \eqref{eq:the_four_matrices}. 
For $\tau=1$, 
$$\sqrt{n}\Big(\hat{q}_n - q\Big)\xrightarrow{\D}\N(0,\sigma^2)\;.$$ 
\end{corollary}

\begin{proof}[Proof of Corollary \ref{asymptotic_turning_rate}]
The proof is an extension of the argument used in proof of Theorem \ref{theorem:convergence_OP_SRD}. Let $\mathbf{X}_{t+1}=(X_{t+1}, X_{t+2})^\top$. It holds: 
\begin{align}
     \hat{q}_n -q = \frac{1}{n}\sum_{t=0}^{n-1} h( \mathbf{X}_t)\;, \quad h(\mathbf{X}_t):=   \sum_{\gamma \in \mathcal{T}} \bm{1}\left( V_{\gamma   }\Xb_t \leq \mathbf{0} \right)-q\;.
     \label{eq:h}
\end{align}
The martingale decomposition \eqref{eq:martingale_decomposion} still holds with 
    \begin{align}
 h(\mathbf{X}_t)=\sum_{\gamma \in \mathcal{T}} \bm{1}\left( V_{\gamma   }\Xb_t \leq \mathbf{0} \right) - q =  \sum_{j=0}^\infty \sum_{\gamma \in \mathcal{T}} \tilde{U}_{t,j, \gamma}  = \sum_{j=0}^\infty \tilde{U}_{t,j}\;, 
  \end{align}
  where
  \begin{align}
\tilde{U}_{t,j, \gamma}:&= \mathbb{E}\big[\bm{1}( V_{\gamma }\Xb_t\leq \mathbf{0}) \big| \mathcal{F}_{t-j}\big] - \mathbb{E}\big[\bm{1}(V_{\gamma } \Xb_t\leq \mathbf{0}) \big| \mathcal{F}_{t-j-1}\big]=\tilde{p}_{j}(- \mathbf{R}_{t,j}) - \tilde{p}_{j+1}( -\mathbf{R}_{t,j+1}) \nonumber\;,\\
\end{align}
where we set $\tilde{p}_j (\mathbf{x}):= p_j(V_{\gamma }\mathbf{x})\;.$ From Lipschitzness of $p_j$ (Lemma \ref{lem.Lip}), it follows that $\tilde{p}_j$ is Lipschitz as well, with Lipschitz constant $\text{Lip}(\tilde{p}_j)\leq \text{Lip}(p_j)\|V_\gamma\|\;.$ Therefore, the same conclusion of Lemma \ref{lemma:lemma3.4Furmanczyk} holds for $\tilde{U}_{t,j}$: for $j\geq J $
\begin{align*}
    |\tilde{U}_{t,j}|\leq&   \sum_{\gamma \in \mathcal{T}} |\tilde{U}_{t,j, \gamma} | \leq  \sum_{\gamma \in \mathcal{T}} \text{Lip}(p_j)\|V_{\gamma }\|\, \|A_{j+1}\| \, \left( \|\mathbf{Z}_{t-j-1}\| + \sqrt{\EE [ \|\mathbf{Z}_{t-j-1}\|^2]} \right)\quad \text{see Lemma \ref{lemma:lemma3.4Furmanczyk}}\\
    \leq & \C \left( \sum_{\gamma \in \mathcal{T}} \| V_\gamma\|\right) \text{Lip}(p_J)  \|A_{j+1}\| \, \left( \|\mathbf{Z}_{t-j-1}\| + \sqrt{\EE [ \|\mathbf{Z}_{t-j-1}\|^2]} \right)\;.
\end{align*}
Lastly, Lemma \ref{lemma6.4} applies to $\tilde{U}_{t,j}$, i.e $\EE[ \tilde{U}_{t,j} U_{t',j'}]=0$ if $t-j \neq t'-j'\;.$ In fact, 
\begin{align*}
    \tilde{U}_{t,j}=\EE [ h(\mathbf{X}_t)\,|\, \F_{t-j}] - \EE [ h(\mathbf{X}_t)\,|\, \F_{t-j-1}]\;,
\end{align*}
and the conclusion of Lemma \ref{lemma6.4} follows with $h(\mathbf{X}_t)$ replacing $U_t(\mathbf{x})\;.$
Finally, since the same conclusions of Lemma \ref{lemma6.4} and Lemma \ref{lemma:lemma3.4Furmanczyk} hold for $\tilde{U}_{t,j}$, we can redo exactly the same steps of the proof of Theorem \ref{theorem:SRD_multivariate} (i.e the application of Theorem 4.2 of Billingsley) to 
\begin{align*}
    \tilde{W}_n:&=\frac{1}{\sqrt{n}}\sum_{t=1}^{\lfloor n \tau \rfloor} \sum_{j=0}^\infty \tilde{U}_{t,j}\;, \quad \tilde{V}_{u,n}:=\frac{1}{\sqrt{n}} \sum_{t=1}^{\lfloor n \tau \rfloor} \sum_{j=0}^{u-1} \tilde{U}_{t,j}\;, \\
    \tilde{H}_{u,n}:&=\frac{1}{\sqrt{n}}\sum_{j=0}^{u-1} \left( \sum_{t=1}^j \tilde{U}_{t,j} - \sum_{t=\lfloor n \tau \rfloor+1}^{\lfloor n \tau \rfloor +j} \tilde{U}_{t,j}\right)\;.
\end{align*}
\end{proof}

\begin{theorem}
\label{theorem:asymptotic_turning_rates_series}
 $\xi_0, \ldots, \xi_{n+1}$  be a time series. Consider the turning rate series generated by blocks of size \( m+2 \) and \( n_b = \lfloor n/m \rfloor\) blocks, denoted by $\hat{q}_{1,m}, \ldots, \hat{q}_{n_b,m}$. Suppose $\frac{m}{\sqrt{n}}\longrightarrow \infty\;.$
 \begin{itemize}
     \item [(i)]
  If  $X_t = \sum_{j=0}^\infty a_j Z_{t-j}$ satisfies the assumptions on the increments of Theorem \ref{theorem:convergence_OP_SRD}, then
\begin{equation}
    \frac{m}{\sqrt{n}}\sum\limits_{j=1}^{\lfloor n_b \tau \rfloor} (\hat{q}_{j,m} -q) \overset{\D[0,1]}{\Longrightarrow} \sigma B(\tau) \text{ for }\tau \in [0,1] \quad \text{ as } n_b\to \infty \quad  (\text{or }n\to \infty)\;.
     \label{averages_turning_rate_SRD}
\end{equation}
 where $\sigma^2=\EE[h(\Xb_1)^2]+2\sum_{j=1}^\infty \EE[ h(\Xb_1) h(\Xb_{1+j})]\;,$ for $h$ and $\mathbf{X}_j$ as \eqref{eq:h}.
   \item [(ii)] If  $X_t = \sum_{j=0}^\infty a_j Z_{t-j}$ satisfies the assumptions on the increments of Theorem \ref{theorem:convergence_OP_LRD}, then
\begin{equation}
     \frac{m}{n^{1/2+d}}\sum\limits_{j=1}^{n_b} (\hat{q}_{j,m} -q) \overset{\D}{\longrightarrow} \N(0,\sigma^2)\quad \text{ as } n_b\to \infty \quad  (\text{or }n\to \infty)\;.
     \label{averages_turning_rate_LRD}
\end{equation}
where $ \sigma^2= \sum_{\gamma_1, \gamma_2 \in \mathcal{T}} (\nabla p_{\gamma_2}(\textbf{0}))^\top V_{\gamma_1}V V_{\gamma_2}^\top \nabla p_{\gamma_1}(\textbf{0}) \;,$ with $V=  \frac{\Gamma(d)^2 }{\Gamma(2d+2) \cos \left(\pi d \right) }  \mathbf{E}$ and $V_\gamma$ the matrices defined in \eqref{eq:the_four_matrices}.
 \end{itemize}
\end{theorem}
\begin{proof}[Proof of Theorem \ref{theorem:asymptotic_turning_rates_series}]
We set $ h(\mathbf{X}_t):=   \sum_{\gamma \in \mathcal{T}} \bm{1}\left( V_{\gamma   }\Xb_t \leq \mathbf{0} \right)-q\;.$ 
\begin{itemize}
    \item [(i)] Using (\ref{turning_rate_block}),
\begin{align*}
    m(\hat{q}_{j,m} -q)=& \sum_{i=0}^{m-1}  \sum_{\gamma \in \mathcal{T}}  \bm{1}\left(\Pi(\xi_{(j-1)(m+2)+i} ,  \xi_{(j-1)(m+2) +i+1}, \xi_{(j-1)(m+2)+i+2} ) =\gamma \right)-mq \\
    =&  \sum_{i=1}^{m} h\left(X_{(j-1)(m+2)+i},X_{(j-1)(m+2)+i+1}\right)\;.
\end{align*}
Therefore, 
\begin{align*}
  m\sum\limits_{j=1}^{\lfloor n_b \tau \rfloor} (\hat{q}_{j,m}-q)=  \sum\limits_{j=1}^{\lfloor n_b\tau\rfloor}\sum_{i=1}^{m} h\left(X_{(j-1)(m+2)+i},X_{(j-1)(m+2)+i+1}\right)\;.
\end{align*}

Further, we can unroll the previous expression: $\sum\limits_{j=1}^{\lfloor n_b\tau\rfloor}\left( \hat{q}_{j,m}-q \right)$ is the sum 
of consecutive terms of the form $h(X_k, X_{k+1})$, but the  summands $h(X_{jm+2j-1}, X_{jm+2j})$ and $h( X_{jm+2j}, X_{jm+2j+1})$ 
are missing for $j=1,\ldots, \lfloor n_b \tau \rfloor -1\;.$ 

Define $ I_n = \{ j(m +2)-1, j(m+2)\, | \, j=1,\ldots, \lfloor n_bt \rfloor -1 \}$. Given the infinite number of data points for $(X_t)_{t\geq 1}$, we complete the sum by letting  $j=1, \ldots,\lfloor n_b\tau \rfloor  $.  
\begin{align*}
 m   \sum\limits_{j=1}^{\lfloor n_b \tau \rfloor} \left( \hat{q}_{j,m}-q \right) =& \sum_{ k \notin I_n}h( X_{k}, X_{k+1}) \quad \text{add and subtract the missing terms}\\
    =& \sum_{k=1}^{\lfloor n_b \tau\rfloor (m+2) } h(\Xb_k) - \sum_{j=1}^{\lfloor n_b \tau \rfloor } \left(h(\Xb_{j(m+2)-1}) + h(\Xb_{j(m+2)}) \right) \;,
\end{align*}
where $\Xb_k=(X_k, X_{k+1})^{\top}$. 
Dividing by $\frac{1}{\sqrt{n_b(m+2)}}$ 
\begin{align*}
   \frac{m}{\sqrt{n_b(m+2)}}\sum\limits_{j=1}^{\lfloor n \tau \rfloor} \left(\hat{q}_{j,m} -q\right)=&  \frac{1}{\sqrt{n_b(m+2)} } \sum_{k=1}^{\lfloor (n_b (m+2)) \tau \rfloor  } h(\Xb_k) - \frac{1}{\sqrt{n_b(m+2)}}\sum_{j=1}^{\lfloor n_b \tau \rfloor} \left(h(\Xb_{j(m+2)-1}) + h(\Xb_{j(m+2)}) \right)\\
    =& \frac{1}{\sqrt{n_b(m+2)} } \sum_{k=1}^{\lfloor n \tau \rfloor  } h(\Xb_k) - \frac{1}{\sqrt{n_b(m+2)}}\sum_{j=1}^{\lfloor n_b \tau \rfloor } \left(h(\Xb_{j(m+2)-1}) + h(\Xb_{j(m+2)}) \right)\;.\\
\end{align*}
Corollary \ref{asymptotic_turning_rate} yields 
$$\frac{1}{\sqrt{n } }  \sum_{j=1}^{\lfloor n\, \tau \rfloor  } h(\Xb_j) \overset{\D[0,1]}{\Longrightarrow} \sigma B(\tau) \text{ for }\tau\in [0,1] \;. $$
 To conclude, we prove that 
$$\frac{1}{\sqrt{n_b(m+2)}}\sum_{j=1}^{\lfloor n_b \tau \rfloor} h(\Xb_{j(m+2)-1} ) +  \frac{1}{\sqrt{n_b(m+2)}}\sum_{j=1}^{\lfloor n_b \tau \rfloor} h(\Xb_{j(m+2)}) =o_{\PP}(1)\;.$$
 It suffices to prove the claim for one of the two addends, as the argument is identical for both terms. For any $\varepsilon>0$
\begin{align*}
    \PP \left( \sup_{ \tau\in [0,1]} \left| \frac{1}{\sqrt{n_b(m+2)}}\sum_{j=1}^{\lfloor 
n_b \tau \rfloor} h(\Xb_{j(m+2)})\right| > \varepsilon \right) &=  \PP \left( \sup_{k=1, \ldots, n_b}  \frac{1}{\sqrt{n_b(m+2)}}\left| \sum_{j=1}^{k} h(\Xb_{j(m+2)})\right| > \varepsilon \right)\\
& \leq   \PP \left( \bigcup_{k=1}^ {n_b} \left\{  \frac{1}{\sqrt{n_b(m+2)}}\left| \sum_{j=1}^{k} h(\Xb_{j(m+2)})\right| > \varepsilon \right\}\right)\\
&\leq \sum_{k=1}^{n_b} \PP \left( \frac{1}{\sqrt{n_b(m+2)}} \left| \sum_{j=1}^{k} h(\Xb_{j(m+2)})\right| > \varepsilon  \right)\\
&\leq  \sum_{k=1}^{n_b} \frac{1}{n_b (m+2) \varepsilon^2} \Var \left( \sum_{j=1}^{k} 
h(\Xb_{j(m+2)})\right)\;. 
\end{align*}
To upper bound the variance, recall that $h(\Xb_j)$ is a stationary zero mean process, being the composition of a stationary process with a measurable function $h$. Using the definition of variance 
\begin{align*}
    \Var\left( \sum_{i=1}^k h(\Xb_{i(m+2)}) \right) =& \sum\limits_{i=1}^k\sum\limits_{j=1}^k\Cov(h(\Xb_{i(m+2)}), h(\Xb_{j(m+2)})) = k \,\gamma (0) + 2 \sum_{j=1}^{k-1}  (k-j)\gamma(j(m+2))\;,
\end{align*}
where $\gamma(k)=\Cov(h(\Xb_i),h( \Xb_{i+k}))$. It follows that 
\begin{align*}
    \frac{1}{k} \Var\left( \sum_{i=1}^k h(\Xb_{i(m+2)}) \right) \leq & |\gamma(0)|+ 2\sum_{j=1}^{k-1}\left(1-\frac{j}{k} \right)|\gamma(j(m+2))|
     \leq  |\gamma(0)|+ 4\sum_{j=1}^{\infty}|\gamma(j)|< \C\;.\\
\end{align*}
The last inequality follows from Lemma 3.5 in \citet{Furmanczyk}. 
Thus, it follows that $ \Var\left( \sum_{i=1}^k h(\Xb_{i(m+2)}) \right)\leq \C k\;.$ Putting all pieces together, 
\begin{align*}
    \PP \left( \sup_{\tau\in [0,1]} \left| \frac{1}{\sqrt{n_b(m+2)}}\sum_{j=1}^{\lfloor 
n_b \tau \rfloor -1} h(\Xb_{j(m+2)})\right| > \varepsilon \right)  
&\leq  \sum_{k=1}^{n_b-1} \frac{1}{n_b (m+2) \varepsilon^2} \Var \left( \sum_{j=1}^{k} 
h(\Xb_{j(m+2)})\right) \\
&\leq \frac{\C}{n_b (m+2) \varepsilon^2}  \sum_{k=1}^{n_b-1} k = \mathcal{O}\left( \frac{n_b^2}{n_b(m+2)} \right)\\
=&\mathcal{O}\left( \frac{1}{(m/\sqrt{n})^2}\right)
 \xrightarrow{n\to \infty } 0\;.
\end{align*}
\item [(ii)]
Following the same argument of part (i), let us go back to the expression (set $\tau=1$)
\begin{align*}
 m \sum\limits_{j=1}^{n_b} \left(\hat{q}_{j,m} -q\right)=   \sum_{k=1}^{n_b(m+2) } h(\Xb_k) - \sum_{j=1}^{n_b } \left(h(\Xb_{j(m+2)-1}) + h(\Xb_{j(m+2)}) \right)\;.\\
\end{align*}
Multiplying both sides by $(n_b(m+2))^{-(\frac{1}{2}+d)}$
\begin{align*}
 m(n_b(m+2))^{-\left(\frac{1}{2}+d\right)}  \sum\limits_{j=1}^{n_b} \left(\hat{q}_{j,m} -q\right)=& (n_b(m+2))^{-\left(\frac{1}{2}+d\right)}   \sum_{k=1}^{n_b(m+2) } h(\Xb_k) \\ &- (n_b(m+2))^{-\left(\frac{1}{2}+d\right)} \sum_{j=1}^{n_b } \left(h(\Xb_{j(m+2)-1}) + h(\Xb_{j(m+2)}) \right)\;.\\
\end{align*}
We firstly show that the second term converges in probability to 0.  For this, note that 
\begin{align*}
    \PP \left(\left| \sum_{j=1}^{n_b} h(\Xb_{j(m+2)})\right| > \left(n_b(m+2)\right)^{(\frac{1}{2}+d)} \varepsilon  \right)& \leq \varepsilon^{-2} (n_b(m+2))^{-1-2d} \Var \left(  \sum_{j=1}^{n_b } h(\Xb_{j(m+2)}) \right)\\
    &\leq \C \varepsilon^{-2} (n_b(m+2))^{-1-2d} n_b^2 \\
    & =  \mathcal{O}\left( \frac{n_b^{2}}{{(n_b(m+2))}^{1+2d}}\right)  =  \mathcal{O}\left( \frac{1}{\left(\frac{m}{\sqrt{n}}\right)^2n^{2d}}\right)\xrightarrow
    {n\to \infty }0\;.
\end{align*}
 Next, we prove convergence in distribution for the first term $n^{-(\frac{1}{2}+d)} \sum_{j=1}^{n} h(\Xb_j)$ as a consequence of Theorem \ref{theorem:reduction_principle}. For $\gamma \in \mathcal{T}$, let $p_\gamma(\mathbf{x})=\mathbb{P}( V_\gamma \mathbf{X}_1\leq \mathbf{x})$. Then, 
\begin{align*}
n^{-(\frac{1}{2}+d)} \left| \sum_{j=1}^n h(\Xb_j)-q+ \sum_{ \gamma \in \mathcal{T}}\left( (\nabla p_\gamma (\textbf{0}) V_\gamma\right) n\Bar{\mathbf{X}}_n \right| =&n^{-(\frac{1}{2}+d)} \left| \sum_{j=1}^n \sum_{ \gamma \in \mathcal{T}} \left( \bm{1}\left( V_\gamma \mathbf{X}_j\leq \mathbf{0}\right)-p_\gamma(0)\right)+ \sum_{ \gamma \in \mathcal{T}}\left( \nabla p_\gamma (\textbf{0}) V_\gamma\right)\sum_{j=1}^{n}\mathbf{X}_j  \right| \\
    &\leq \sum_{ \gamma \in \mathcal{T}} n^{-(\frac{1}{2}+d)} \left| \sum_{j=1}^n \bm{1}\left( V_\gamma \mathbf{X}_j\leq \mathbf{0}\right)-p_\gamma(0) + \nabla p_\gamma (\textbf{0})V_\gamma \sum_{j=1}^{n}\mathbf{X}_j \right| \xrightarrow{\PP}0\;.
\end{align*}
The last relation is due to the reduction principle \eqref{uniform_principle}. Thus, an application of Theorem 1 in \citet{chung_2002} yields
\begin{align}
    n^{-(\frac{1}{2}+d)} \sum_{j=1}^n h(\Xb_j) \xrightarrow{\D} \N(0, \sigma^2)\;,
\end{align}
where $ \sigma^2= \sum_{\gamma_1, \gamma_2 \in \mathcal{T}} (\nabla p_{\gamma_2}(\textbf{0}))^\top V_{\gamma_1}V V_{\gamma_2}^\top \nabla p_{\gamma_1}(\textbf{0}) \;,$ with $V=\frac{\Gamma(d)^2}{\Gamma(2d+2)\cos(\pi d)} \mathbf{E}$, and $\mathbf{E}$ is the matrix with all entries equal 1.  
\end{itemize}
\end{proof}
\begin{proof}[Proof of Corollary \ref{self_normalized_cusum_turning_rate}]
In order to derive the asymptotic distribution of $SC_{n_b}$, we need to find to first derive the limit of $\hat{V}_{k,n_b}$. The strategy is to write  $\hat{V}_{k,n_b}$ as a functional of the process $W_{n_b}(s) := \frac{m}{\sqrt{n}} \sum_{j=1}^{\lfloor n_bs\rfloor} (\hat{q}_{j,m}-q)$, for which $W_{n_b } \overset{\D[0,1]}{\Longrightarrow} \sigma B(\tau)$; see Theorem \ref{theorem:asymptotic_turning_rates_series}.  
 The variance estimator \eqref{V_{k,n}} is the sum of two terms, namely $V^2_{k, n_b}:=\frac{1}{n_b}\sum_{t=1}^k S_t^2(1,k)+\frac{1}{n_b}\sum_{t=k+1}^{n_b} S_t^2(k+1,n_b)$. 
\begin{enumerate}
    \item For $i=1, \ldots, k$,
$$ S_i(1,k)=\sum_{j=1}^i ( \hat{q}_{j,m} -\Bar{q}_{1,k})= \sum_{j=1}^{i} \hat{q}_{j,m} -\frac{i}{k} \sum_{j=1}^k \hat{q}_{j,m}\;.$$
Denote with $\tau \in [0,1]$ the unique value satisfying $k= \lfloor n_b \tau \rfloor$. Then, 
\begin{align*}
   \frac{1}{n_b}\sum_{i=1}^k S^2_i(1,k)= \frac{1}{n_b}\sum_{i=1}^{\lfloor n_b \tau \rfloor } \left( \sum_{j=1}^{i} \hat{q}_{j,m} -\frac{i}{k} \sum_{j=1}^k \hat{q}_{j,m}\right)^2 = \int_0^\tau \left(  \sum_{j=1}^{\lfloor n_b s\rfloor} \hat{q}_{j,m} -\frac{\lfloor n_b s \rfloor}{\lfloor n_b \tau\rfloor} \sum_{j=1}^{\lfloor n_b \tau \rfloor} \hat{q}_{j,m}\right)^2 \, ds\;.
\end{align*}
Multiplying both sides by $\left( \frac{m}{\sqrt{n}}\right)^2 $ yields 
\begin{align*}
  \left( \frac{m}{\sqrt{n}}\right)^2  \left(  \frac{1}{n_b}\sum_{i=1}^k S^2_i(1,k)\right) = \int_0^\tau \left( W_{n_b}(s) - \frac{\lfloor n_b s \rfloor}{\lfloor n_b \tau\rfloor} W_n(\tau) \right)^2 \, ds\;,
\end{align*}
\item For $i=k+1,\ldots, n_b$ 
\begin{align*}
   S_i(k+1,n_b)=&\sum_{j=k+1}^i ( \hat{q}_{j,m} -\Bar{q}_{k+1,n_b})= \sum_{j=k+1}^{i} \hat{q}_{j,m} -\frac{i-k}{n_b-k} \sum_{j=k+1}^{n_b} \hat{q}_{j,m} \;.
\end{align*}
Then, 
\begin{align*}
 \left( \frac{m}{\sqrt{n}}\right)^2   \left( \frac{1}{n_b} \sum_{i=k+1}^n S_i^2(k+1,n_b) \right) =& \left( \frac{m}{\sqrt{n}}\right)^2   \int_{\tau}^1\left(\sum_{j=k+1}^{\lfloor n_bs \rfloor} \hat{q}_{j,m} -\frac{\lfloor n_bs \rfloor-\lfloor n_b\tau \rfloor}{n_b-\lfloor n_b\tau \rfloor} \sum_{j=k+1}^{n_b} \hat{q}_{j,m} \right)^2\, ds\\
=&\int_{\tau}^1 \left( W_{n_b}(s) - W_{n_b}(\tau) -\frac{\lfloor n_bs \rfloor-\lfloor n_b\tau \rfloor}{n_b-\lfloor n_b\tau \rfloor} \left( W_{n_b}(1)-W_{n_b}(\tau) \right)\right)^2 \, ds\;.\\
\end{align*}
 \end{enumerate}
Since $\frac{\lfloor n_bs \rfloor}{n_b} = s + o(1)$,  we can replace all $\lfloor n_bs \rfloor$ by $s$.  Finally, Theorem \ref{theorem:cusum_statistics_limit} and the continuous mapping theorem yield the desired result. 
\end{proof}

\section{Auxiliary Results} 
\label{app2}
\begin{theorem}[Theorem 2.1 and 2.2 in \citet{Furmanczyk}]
\label{Theorem:Furmanczyk} 
Let $\Xb_t$ the multivariate linear process 
    $$\Xb_t = \sum_{j=0}^\infty A_j \mathbf{Z}_{t-j}$$ 
   where $(\mathbf{Z}_{j})_{j \in \NN} $ is an i.i.d. sequence with variance $\Sigma$ and $(A_j)_{j\in\NN}$ is a sequence of non random matrices in $\RR^p$. Further,  $\EE[\| \mathbf{Z}_1\|^2]<\infty$, where $\|\cdot\|$ represents the standard Euclidean norm. Consider the following 
    \begin{enumerate} 
    \item  The sequence $(A_j)_{j \in \NN}$ satisfies the SRD condition, $\sum\limits_{j=0}^\infty \|A_j\| <\infty$ for $\|\cdot\|$ induced by a vector norm $\|\cdot\|$. 
    \item For every $s \in \NN$, consider the function  $G_s(\textbf{x})=\EE \left[ g\left( \sum_{r=0}^{s-1}A_r \mathbf{Z}_{k-r} +\textbf{x}\right)\right]$, where $g:\RR^p \to \RR$ is a measureable function such that $\EE[ g(\Xb_1)^2] < \infty$.
    \begin{itemize}
        \item []\textbf{(Lip)}   $G_s$  is uniformly Lipschitz over $s$, i.e
$$ \mid G_s(x)-G_s(y)\mid \leq \text{Lip}(G_s)\| x-y\| \quad \forall x,y \in \RR^p $$
    with $\sup\limits_{s\geq 1}\text{Lip}(G_s) < \C$, where $\C$ is a constant indepent from $s$.
     \item []\textbf{(Lip$^\ast$)}   $G_s$  is definitely Lipschitz over $s$, i.e exists $s_0$ such that for all $s>s_0$
$$ \mid G_s(x)-G_s(y)\mid \leq \text{Lip}(G_s)\| x-y\| \quad \forall x,y \in \RR^p $$
    with $\sup \limits_{s\geq s_0}\text{Lip}(G_s) < \C$, where $\C$ is a constant indepent from $s$.
    \end{itemize}
   \item $ \sum\limits_{i=j}^\infty\|A_i\|^2 = \mathcal{O}(j^{-t})$ for some $t>2$. 
    \end{enumerate}
If (1, 2. \textbf{Lip} ) or (1, 2. \textbf{Lip} $^\ast$, 3) are satisfied, then
\begin{equation*}
     n^{-1/2}\sum_{j=1}^{[nt]} g(\Xb_j)  \overset{\D[0,1]}{\Longrightarrow} \sigma B(t) \quad t\in [0,1]\;, 
\end{equation*}
    where $\sigma^2 = \EE[g(\Xb_1)^2]+2\sum_{j=1}^\infty \EE[g(\Xb_1)g(\Xb_{1+j})]\;.$
\end{theorem}

\end{document}